\documentclass{svmult}

\usepackage{type1cm}
\usepackage{makeidx}
\usepackage{graphicx}
\usepackage{multicol}
\usepackage[bottom]{footmisc}
\usepackage{newtxtext} 
\usepackage[varvw]{newtxmath}
\makeindex

\usepackage[colorinlistoftodos,textwidth=97pt]{todonotes}
\usepackage[left=1.5in, right=1.5in, bottom=1.25in,
height=8in]{geometry}

\input xy
\xyoption{all}
\usepackage{accents}
\usepackage{epsfig}
\usepackage{xcolor}
\usepackage{amsthm}
\usepackage{bbm}
\usepackage{amsmath}
\usepackage{amscd}
\usepackage{amsopn}

\usepackage{tikz}
\usetikzlibrary{shapes,shadows,arrows}
\usepackage{tikz-cd}

\usepackage{xspace}

\usepackage{hhline}
\usepackage{easybmat}
\usepackage{caption}   
\usepackage{relsize}

\usepackage{url}
\usepackage{enumitem, hyperref}\hypersetup{colorlinks}

\usepackage{stmaryrd}

\usepackage[T1]{fontenc}

\newtheorem{Theorem}{Theorem}
\numberwithin{Theorem}{section}

\newtheorem{TheoremX}{Theorem}

\newtheorem{innercustomgeneric}{\customgenericname}
\providecommand{\customgenericname}{}
\newcommand{\newcustomtheorem}[2]{%
  \newenvironment{#1}[1]
  {%
   \renewcommand\customgenericname{#2}%
   \renewcommand\theinnercustomgeneric{##1}%
   \innercustomgeneric
  }
  {\endinnercustomgeneric}
}

\newcustomtheorem{customthm}{Theorem}
\newcustomtheorem{customlemma}{Lemma}
\newcustomtheorem{customcor}{Corollary}
\newcustomtheorem{customrk}{Remark}

\newtheorem   {Lemma}[Theorem]{Lemma}

\newtheorem   {Proposition}[Theorem]{Proposition}
\newtheorem   {Problem}[Theorem]{Problem}
\newtheorem   {Corollary}[Theorem]{Corollary}
\theoremstyle {definition}
\newtheorem   {Definition}[Theorem]{Definition}
\theoremstyle {remark}
\newtheorem   {Remark}[Theorem]{Remark}
\newtheorem   {Example}[Theorem]{Example}

\makeatletter
\def\reflb#1#2{\begingroup
    #2%
    \def\@currentlabel{#2}%
    \phantomsection\label{#1}\endgroup
}
\makeatother

\def    \eps    {\epsilon}

\newcommand{\CA}{{\mathcal A}}
\newcommand{\CC}{{\mathcal C}}

\newcommand{\CL}{{\mathcal L}}

\newcommand{\CU}{{\mathcal U}}
\newcommand{\CV}{{\mathcal V}}

\newcommand{\CS}{{\mathcal S}}
\newcommand{\CT}{{\mathcal T}}

\newcommand{\supp}{\operatorname{supp}}

\newcommand{\id}{{\mathit id}}

\newcommand{\const}{{\mathit const}}

\newcommand{\fg}{{\mathfrak g}}

\newcommand{\fb}{{\mathfrak b}}
\newcommand{\fc}{{\mathfrak c}}
\newcommand{\fd}{{\mathfrak d}}

\newcommand{\fh}{{\mathfrak h}}

\newcommand{\ty}{\tilde{y}}

\newcommand{\tL}{\tilde{L}}

\newcommand{\tK}{\tilde{K}}

\newcommand{\tU}{\tilde{U}}

\newcommand{\tu}{\tilde{u}}

\newcommand{\tx}{\tilde{x}}

\newcommand{\tz}{\tilde{z}}
\newcommand{\cz}{\check{z}}
\newcommand{\hz}{\hat{z}}

\newcommand{\tH}{\tilde{H}}

\newcommand{\CB}{{\mathcal B}}

\newcommand{\PP}{{\mathcal P}}

\def    \nat    {{\natural}}
\def    \F      {{\mathbb F}}
\def    \AA    {{\mathbb A}}
\def    \I      {{\mathbb I}}
\def    \C      {{\mathbb C}}
\def    \R      {{\mathbb R}}

\def    \Z      {{\mathbb Z}}
\def    \N      {{\mathbb N}}
\def    \Q      {{\mathbb Q}}

\def    \T      {{\mathbb T}}
\def    \CP     {{\mathbb C}{\mathbb P}}
\def    \RP     {{\mathbb R}{\mathbb P}}

\def    \12     {{\frac{1}{2}}}

\def    \p      {\partial}
\def    \codim  {\operatorname{codim}}

\def    \im     {\operatorname{im}}

\def    \SH     {\operatorname{SH}}

\def    \SL    {\operatorname{SL}}

\def    \HF     {\operatorname{HF}}

\def    \H      {\operatorname{H}}

\def    \CF      {\operatorname{CF}}
\newcommand    \vol     {\operatorname{vol}}

\def    \by      {\bar{y}}
\def    \bz      {\bar{z}}
\def    \bg      {\bar{\gamma}}

\def    \hn    {\scriptscriptstyle{H}}
\def    \Fl    {\scriptscriptstyle{Fl}}

\def    \cf    {\operatorname{c}}

\newcommand    \htop  {\operatorname{h_{\scriptscriptstyle{top}}}}
\newcommand    \hvol  {\operatorname{h_{\scriptscriptstyle{vol}}}}
\newcommand   \hbr {\operatorname{\hbar}}
\newcommand   \hm {\operatorname{h}}

\newcommand    \width {\operatorname{width}}

\newcommand   \slope {\operatorname{\mathit{slope}}}
\newcommand   \rmax {r_{\max}}

\newcommand   \WW {\widehat{W}}

\newcommand \Quad {\R^n_{{\geq 0}}}

\newcommand \diam {\operatorname{diam}}
\newcommand \Stab {\operatorname{Stab}}

\title*{Topics in Symplectic Dynamics: Barcode Entropy}
\author{Erman \c Cineli, Viktor L.\ Ginzburg, Ba\c{s}ak Z.\ G\"urel, and
  Marco Mazzucchelli}
\institute{Erman \c Cineli \at Department of
  Mathematics, ETH Z\"urich, R\"amistrasse 101, 8092
  Z\"urich, Switzerland\\ \email{erman.cineli@math.ethz.ch}
  \and
Viktor L.\ Ginzburg \at Department of Mathematics, UC Santa Cruz, Santa Cruz, CA 95064, USA\\ \email{ginzburg@ucsc.edu}
\and 
Ba\c{s}ak Z.\ G\"urel \at Department of Mathematics, UCF, Orlando, FL 32816, USA\\ \email{basak.gurel@ucf.edu}
\and
Marco Mazzucchelli \at IMJ-PRG, Sorbonne Universit\'e, 75252 Paris Cedex 5, France\\ \email{marco.mazzucchelli@imj-prg.fr}}

\begin{document}

\maketitle

\abstract{\ \ \ Barcode entropy is an invariant of a Hamiltonian system -- a
  Hamiltonian diffeomorphism or a Reeb flow -- measuring its Morse- or
  Floer-theoretic complexity at a small scale. More specifically, it
  is the exponential growth rate of the number of not-too-short bars
  in the Floer or symplectic homology persistence module. Barcode
  entropy is closely related to topological entropy, even though they
  originate in different contexts, and in low dimensions they
  coincide. In these notes, we study barcode entropy and related
  invariants in various settings and explore their connections with
  pure dynamics features and, in particular, topological entropy. The
  methods build on techniques from symplectic topology and Floer
  theory, dynamical systems, and smooth integral geometry. We also
  touch upon some other applications of the machinery we develop. These notes are based on the mini-course given by the second author at the CIME summer school ``Symplectic Dynamics and Topology'' (Cetraro, Italy, June 16-20, 2025).}

\tableofcontents

\markboth{Erman \c Cineli, Viktor L.\ Ginzburg, Ba\c{s}ak Z.\ G\"urel, and Marco Mazzucchelli}{Topics in Symplectic Dynamics: Barcode Entropy}


\section{Introduction}

A distinctive feature of Hamiltonian dynamical systems is the least
action principle: trajectories of such systems are the critical points
of the action functional. This enables one to study Hamiltonian systems via
Morse or Floer theory.
This theory extracts certain information from the system and relays it
to the dynamics, but \emph{a priori} it is not clear which dynamics
features, beyond periodic orbits or chords, are retained and which are
lost:

\medskip
\begin{center}
  \resizebox{5.cm}{2cm}{
      \begin{tikzpicture}[node distance={27mm}, thick, main/.style = {draw}]
  \node[main] (1) {Hamiltonian Systems};
  \node[main] (2) [below right of =1]{Floer/Morse Theory};
  \node[main] (3) [above right of =2]{Dynamics};
 
  \draw[->] (1) -- (2);
  \draw[->] (1) -- (3);
  \draw[->] (2) -- (3);
 \end{tikzpicture}} 
\end{center}
This is not a lossless process.

Here we focus on topological entropy as a simple but
important dynamics invariant reflecting the complexity of a system. For
our purposes, a Hamiltonian dynamical system is either a Hamiltonian
diffeomorphism or a Reeb flow.

Very roughly speaking,  in symplectic dynamics there  are two distinct
sources  of  topological  entropy.  The   first  is  the  topology  or
symplectic  topology   of  the  phase  space,   i.e.,  the  underlying
symplectic or contact  manifold. An example is  the exponential growth
of the  homology of  the loop  space or  more generally  of symplectic
homology.  This growth  results  in positive  entropy  and is  readily
visible in Morse or Floer  theory. For geodesic flows, this connection
between  Morse theory  and  topological entropy  has been  extensively
studied   in,   e.g.,   \cite{Di,Ka82,Pa,Pa:book}.    A   variety   of
far-reaching generalizations of these  classical results to Reeb flows
have been obtained in, e.g., \cite{AASS, Al:Anosov, Al:Leg, ACH, ADMM,
  AM, AP, MS}, relating the topological entropy of Reeb flows to their
Floer theoretic invariants (e.g., symplectic or contact homology).

An important feature of this source of topological entropy is that it
is unconditional: exponential growth of Floer-type homology forces all
Hamiltonian dynamical systems (within a reasonable class) on a given
phase space to have positive entropy. For instance, this is the case
for all geodesic flows on a surface of genus at least two.

However, one can never have growth of Floer homology for Hamiltonian
diffeomorphisms of closed symplectic manifolds due to its invariance.
While this fact is crucial for proving results of the symplectic
topological nature, such as Arnold's conjecture, it is a hindrance to
studying dynamics beyond periodic orbits via Floer theory.

The second source is dynamics, e.g., a horseshoe localized to a small
ball. This source is conditional -- it genuinely depends on a
particular system -- and this is what we are interested in in these
notes. Can Floer or Morse homology detect topological entropy coming
from a localized horseshoe?

The answer depends on what is understood as Floer or Morse
homology. For our purposes, all Morse- or Floer-type homology spaces
are always filtered by the action functional. Thus, for instance, for
a Hamiltonian diffeomorphism $\varphi$ of a closed symplectic manifold
$M$, its Floer homology over a fixed field $\F$ is a family of vector
spaces $\HF^s(\varphi)$ parametrized by the action $s\in \R$. In
addition, we have the structure maps
$\HF^s(\varphi)\to \HF^t(\varphi)$ for $s\leq t$, induced by
inclusions. For a pair of closed monotone Lagrangian manifolds $L$ and
$L'$ of $M$, we get the filtered Floer homology spaces
$\HF^s(L,\varphi(L'))$, $s\in \R$, and again the associated structure
maps.  Likewise, for the Reeb flow on the boundary of a Liouville
domain $W$, its symplectic homology is a family of vector spaces
$\SH^s(W)$, parametrized by $s\in \R$, together with the structure
maps. The same structure is present in Morse theory.

Thus, in many instances, we can think of Floer, Morse or symplectic
homology as a persistence module. This perspective on Floer theory was
pioneered in \cite{PS}. We refer the reader to \cite{PRSZ} for an
introduction to persistence homology theory and its applications in
symplectic geometry. (Basic definitions and results are also reviewed
in Section \ref{sec:persistence}.) In particular, these persistence
modules have barcodes associated to them. Strictly speaking, in some
cases the filtered Floer homology is not literally a persistence
module, but its barcode is nevertheless defined; \cite{UZ}.

Barcode entropy is an invariant of the sequence of Floer homology
persistence modules $\HF^s\big(\varphi^k\big)$ or
$\HF^s\big(L,\varphi^k(L')\big)$ for the iterates of a Hamiltonian
diffeomorphism $\varphi$, or of the symplectic homology persistence
module $\SH^s(W)$ for a Reeb flow $\varphi$ on $\p W$. In both cases,
it measures the exponential growth rate of the number of not-so-short
bars in the barcode. For both Hamiltonian diffeomorphisms and Reeb
flows, it comes in two versions: absolute and relative. In the former,
it depends only on the system; in the latter, it also uses a pair of
Lagrangian submanifolds as above. For Hamiltonian diffeomorphisms,
barcode entropy was originally introduced in \cite{CGG:Entropy}. For
geodesic and Reeb flows, it was defined in \cite{GGM} and \cite{FLS};
see also \cite{CGGM:Reeb,Fe1}. It has since been further studied in,
e.g., \cite{CGG:Growth,CGG:Metric,Fe2,Me}.

To be more precise, let $\fb_\eps\big(\varphi^k\big)$ be the number of
bars of length greater than $\eps$ for $\varphi^k$ in the Hamiltonian
case and let $\fb_\eps(s)$ be the number of such bars beginning below
$s$ for a Reeb flow $\varphi$. This is a decreasing function of
$\eps\geq 0$. We think of $\fb_\eps$ for small $\eps>0$ as a measure
of instability or chaos in the filtered homology. The $\eps$-barcode
entropy $\hbar_\eps(\varphi)$ or, in the relative case,
$\hbar_\eps(\varphi;L,L')$ is the exponential growth rate of
$\fb_\eps\big(\varphi^k\big)$ or $\fb_\eps(s)$. By passing to the
limit as $\eps\to 0^+$, we obtain the barcode entropy $\hbar(\varphi)$
or $\hbar(\varphi;L,L')$; see \eqref{eq:eps-entropy},
\eqref{eq:entropy} and \eqref{eq:eps-entropy2}.

For instance, $\fb_\infty$ is the dimension of the total homology,
i.e, the persistent part of the persistence module. The first source
of positive topological entropy discussed above is the case where
$\fb_\infty(s)$ grows exponentially. Then we automatically have
$\hbar_\eps>0$ for all $\eps$. However, we could still have -- and
often do have -- $\hbar>0$ even when the total homology is zero or just
independent of $\varphi$; see Example~\ref{ex:size_does_not_matter}.

The definition of barcode entropy is modeled on topological entropy,
but \emph{a priori} there is absolutely no reason to think that the
two concepts are related.  Surprisingly, they are. The following three
facts encompass most of what is currently known about this connection:
    \begin{itemize}

\item Theorem \ref{thm:A}:
      $\hbr(\varphi; L,L')\leq \htop(\varphi)$
    and, in particular, $\hbr(\varphi)\leq \htop(\varphi)$;
   
    \item Theorem \ref{thm:B}:
$\hbr(\varphi)\geq \htop(\varphi|_K)$ for any closed hyperbolic
invariant subset $K\subset M$;

\item Theorem \ref{thm:C}:
$\hbr(\varphi)= \htop(\varphi)$ in the
   Hamiltonian case when $\dim M=2$ and in the Reeb case when
   $\dim W=3$.

 \end{itemize}
 Our main goal in these notes is to develop the necessary machinery
 and prove these theorems. We will also briefly touch upon some other
 applications of the techniques on which the proofs are based. These
 applications concern the lower semi-continuity of Lagrangian volume,
 \cite{CGG:Vol}, and the behavior of the spectral norm under iterations,
 \cite{CGG:Spectral}. Theorem \ref{thm:B} has a relative version in
 both the Hamiltonian, \cite{Me}, and Reeb, \cite{Fe2}, settings, but
 for the sake of brevity we do not state or prove it here. Theorem
 \ref{thm:C} is an immediate consequence of Theorems \ref{thm:A} and
 \ref{thm:B} and the results from \cite{Ka,LCS,LCY} asserting that in
 low dimensions all topological entropy comes from hyperbolic sets.

 A pattern the reader should be aware of is that here the definitions
 are generally simpler for Reeb flows, whereas the proofs are simpler
 and conceptually cleaner for Hamiltonian diffeomorphisms. The
 difference in the proofs of Theorem \ref{thm:A} is relatively minor
 and superficial, but the proof of Theorem \ref{thm:B} in the Reeb
 setting is considerably more involved than its Hamiltonian
 counterpart and requires new ideas.  In the proofs, we will mainly
 focus on the Hamiltonian case and only briefly indicate the necessary
 modifications for Reeb flows.

 Beyond barcode entropy, one can think of $\fb_\eps$, which we call
 the barcode function, and its growth rate as invariants of $\varphi$
 that reflect its Floer-theoretic complexity. For instance, in Section
 \ref{sec:integrable} based on \cite{BG}, we establish an upper bound
 on the polynomial growth rate of $\fb_\eps$ for toric integrable
 systems, as a counterpoint to the exponential growth seen in more
 complex systems. Interestingly, the barcode function $\fb_\eps$, while
 not particularly commonly used, arises in some other applications of
 persistent homology to geometric problems not directly related to
 dynamics; see, e.g., \cite{BHPW, BP3S2, CSEHM, PPS}.

 We conclude these notes with an annotated list of several open
 problems; open problems are also occasionally mentioned in other
 sections.  Some work related to barcode entropy is entirely omitted
 from the notes. For instance, we do not discuss the results from
 \cite{ABC} on the categorical approximability of Lagrangian submanifolds,
 which would take us too far from our main subject and techniques. Nor
 do we consider the variant of barcode entropy for Varolg\"une\c s'
 relative symplectic homology introduced in \cite{Ah}.

 The reader is not expected to have prior knowledge of persistence
 modules or topological entropy, although it would, of course, be
 helpful. On the other hand, realistically, some familiarity
 with Hamiltonian/Lagrangian Floer homology and symplectic homology
 is essential. While we recall many of the definitions and results,
 this material is too extensively used to cover all the
 necessary background here. 

 The notes are organized as follows. Section \ref{sec:background}
 covers some of the background material, and the reader may wish to
 consult this section only as needed. Sections \ref{sec:persistence}
 and \ref{sec:Top_Entropy} provide concise introductions to persistence
 modules and topological entropy. In Sections \ref{sec:Lagr} and
 \ref{sec:SH}, we review elements of Floer theory -- Lagrangian and
 Hamiltonian Floer homology and symplectic homology -- and define the
 barcode function $\fb_\eps$. We set our conventions and notation and
 briefly recall the necessary definitions and facts, but this
 discussion is not comprehensive or detailed.

 In Section \ref{sec:barcode_entropy}, we define barcode entropy, state
 and discuss our main results -- Theorems \ref{thm:A}, \ref{thm:B} and
 \ref{thm:C} -- and prove basic properties of barcode entropy. The
 main goal of Section \ref{sec:pf-A2} is to prove Theorem \ref{thm:A}
 in both the absolute and relative (Theorem \ref{thm:A'}) settings. The
 proof, given in Section \ref{sec:strat}, relies on the machinery of
 Lagrangian tomographs, developed in Section \ref{sec:tomographs},
 which essentially belongs to smooth integral geometry. We conclude
 Section \ref{sec:pf-A2} with an application, based on \cite{CGG:Vol},
 to the lower semi-continuity of Lagrangian volume; see Section
 \ref{sec:Lagr-vol}.

 Theorem \ref{thm:B} is proved in Section \ref{sec:pf-B}. The proof is
 based on a version of the crossing energy theorem, going back to
 \cite{GG:hyperbolic,GG:PR}, which is stated and proved in Section
 \ref{sec:energy-dynamics}. The modifications needed in the Reeb case,
 which are non-trivial, are discussed in Section
 \ref{sec:energy-Reeb}. In Section \ref{sec:sectral+inv}, we touch upon
 an application of the crossing energy theorem to the properties of the
 spectral norm of the iterates.

 In Section \ref{sec:integrable}, we prove polynomial growth upper
 bounds for the barcode function of toric integrable Hamiltonians and
 Reeb flows. Finally, in Section \ref{sec:problems}, we
 discuss several open problems.




\section{Background}
\label{sec:background}
In this section, we cover some background material: persistence
modules and barcodes (Section \ref{sec:persistence}), Lagrangian and
Hamiltonian Floer homology (Section \ref{sec:Lagr}), filtered
symplectic homology (Section \ref{sec:SH}), and topological entropy
(Section \ref{sec:Top_Entropy}).  The reader may consider consulting
this section as needed.  The first and the last subsections are
essentially brief introductions which should be sufficient for our
purposes. In the middle two subsections, on Floer theory, we set our
conventions and notation, recall the necessary definitions and facts,
aiming at defining the barcode function. Here, some familiarity with
Floer theory is expected.

\subsection{Persistence modules}
\label{sec:persistence}
Persistence modules play a central role in the statements of our main
results. In this section, closely following \cite{CGGM:Reeb}, we
define the class of persistence modules suitable for our goals and
briefly touch upon their properties. We refer the reader to
\cite{PRSZ} for a general introduction to persistence modules and
their applications to geometry and analysis, although the class of
modules they consider is somewhat narrower than the ones we deal
with here, and also to, e.g., \cite{BV, CdSGO16} for some of the more general
results and further references.

\subsubsection{Definitions}
\label{sec:persistence-def}
Fix a field $\F$ which we will suppress in the notation. Recall that a
\emph{persistence module} $(\CV,\pi)$ is a family of vector spaces $\CV_s$
over $\F$ parametrized by $s\in \R$ together with a functorial family
$\pi$ of structure maps. These are linear maps
$\pi_{st}\colon \CV_s\to \CV_t$, where $s\leq t$ and functoriality is
understood as that $\pi_{sr}=\pi_{tr}\pi_{st}$ whenever
$s\leq t\leq r$ and $\pi_{ss}=\id$. In what follows we often suppress
$\pi$ in the notation and simply refer to $(\CV,\pi)$ as $\CV$.  In such a
general form the concept is not particularly useful and usually one
imposes additional conditions on the spaces $\CV_t$ and the structure
maps $\pi_{st}$. These conditions vary depending on the context. Below
we spell out the framework most suitable for our purposes.

Namely, we require that the following four conditions are met:

\begin{itemize}

\item[\reflb{PM1}{{(PM1)}}] There exists a bounded from below, closed, nowhere
  dense subset $\CS\subset \R$, called the \emph{spectrum} of $\CV$,
  such that the persistence module $\CV$ is \emph{locally constant}
  outside $\CS$, i.e., $\pi_{st}$ is an isomorphism when $s\leq t$
  are in the same connected component of $\R\setminus \CS$.

\item[\reflb{PM2}{\rm{(PM2)}}] The persistence module $\CV$ is
  \emph{q-tame}: $\pi_{st}$ has finite rank for all $s<t$.
 
\item[\reflb{PM3}{\rm{(PM3)}}] \emph{Left-semicontinuity}: For all
  $t\in\R$,
  \begin{equation}
    \label{eq:semi-cont}
    \CV_t=\varinjlim_{s<t} \CV_s.
  \end{equation}

\item[\reflb{PM4}{\rm{(PM4)}}] \emph{Lower bound}: $\CV_s=0$ when $s<s_0$
  for some $s_0\in\R$. (Throughout the notes we will assume that
  $s_0=0$.)

\end{itemize}

A few comments on this definition are in order. First, note that as a
consequence of \ref{PM1} and \ref{PM2}, $\CV_s$ is finite-dimensional
and \ref{PM3} is automatically satisfied when $s\not\in
\CS$. Furthermore, Condition \ref{PM1} does not uniquely determine the
spectrum $\CS$: it is simply a requirement that such a set
exists. While one could take the minimal set with the required
properties as the spectrum of $\CV$, which would make it unique and
determined by $\CV$, this choice is not necessarily the most natural
and we prefer to think of the choice of $\CS$ as a part of the
persistence module structure. For instance, in the setting of Example
\ref{ex:sublevels}, it is convenient to take the set of critical
values of $f$ as $\CS$, but the minimal (i.e.~homological)
spectrum can be strictly smaller. The same phenomenon can happen for
Floer or symplectic homology persistence modules. At the intuitive
level, the reader can safely assume that the spectrum $\CS$ is
discrete as in Example \ref{ex:Filt_complex}.

By Condition \ref{PM3}, a persistence module is completely determined
by its restriction to a dense subset $\Gamma\subset \R$. In other
words, once $\CV_s$ and $\pi_{st}$ are defined for $s$ and $t$ in
$\Gamma$, the persistence module can be extended to $\R$ by
\ref{PM3}. (However, in general the extension might fail to meet some
of the requirements \ref{PM1}--\ref{PM4}.) For instance, Floer or
symplectic homology persistence modules are naturally defined only
for $s\not\in\CS$, i.e., for $\Gamma=\R\setminus\CS$, and then the
definition is extended to all $s\in\R$ via \ref{PM3}. 

By \ref{PM4}, we can always assume that $s_0\leq \inf \CS$, i.e.,
$\CS$ is bounded from below. We emphasize, however, that $\CS$ is not
assumed to be bounded from above, and it is actually not in many
examples we are interested in. In what follows it will sometimes be
convenient to include $s=\infty$ by setting
$$
\CV_\infty=\varinjlim_{s\to\infty} \CV_s.
$$
Finally, in all examples we encounter here $\CS$ has zero measure and,
in fact, zero Hausdorff dimension. This fact is never directly used in the
notes, but we do use the requirement that $\CS$ is nowhere dense.

The above definition extends to graded persistence modules in an obvious
way.  Basic examples motivating requirements \ref{PM1}--\ref{PM4} come
from the sublevel homology of a smooth function and Morse theory.

\begin{Example}[Homology of sublevels]
\label{ex:sublevels}
Let $M$ be a smooth manifold and $f\colon M\to \R$ be a proper
$C^r$-smooth function bounded from below with $r$ sufficiently
large. Set $\CV_s:=\H_*\big(\{f<s\};\F\big)$ with the structure maps
induced by the inclusions. No other requirements are imposed on $f$,
e.g., $f$ need not be Morse. However, it is not hard to see that
conditions \ref{PM1}--\ref{PM4} are met with $\CS$ being the set of
critical values of $f$. Thus $\CV_s$ is a graded persistence module.
We note that one can have $\dim \CV_s=\infty$ for $s\in\CS$ already
when $M=S^1$, unless $f$ meets some additional conditions on $f$,
e.g., that $f$ is real analytic or the critical points of $f$ are
isolated. Furthermore, $r$ should be larger than or equal to $\dim(M)$, so that Sard's lemma
applies to ensure that $\CV$ is q-tame.
\end{Example}

\begin{Example}[Morse theory]
\label{ex:Morse}
Let $Q$ be a closed smooth Riemannian manifold. Denote by $\Lambda$ the space
of smooth loops $\gamma\colon S^1\to Q$ and by $f$ the energy functional
$f\colon\Lambda\to \R$ defined by
$$
f(\gamma)=\int_{S^1}\|\dot{\gamma}(t)\|^2\,dt.
$$
The critical points of $f$ are the closed geodesics parametrized
proportionally to arc-length and the sublevel homology
$\CV_s:=H_*\big(\{f<s\};\F\big)$ form again a graded persistence module. We
can replace here the space of closed loops by the space of paths
connecting two closed submanifolds $Q_0$ and $Q_1$ of $Q$. Then the
critical points of $f$ are the geodesics from $Q_0$ to $Q_1$,
parametrized proportionally to arc-length and perpendicular to the
these submanifolds at the end-points. (For instance, when $Q_0$ and
$Q_1$ are two points we have geodesics from $Q_0$ to $Q_1$.) In any
event, in all these cases the sublevel homology is again a persistence
module.
\end{Example}

\begin{Example}[Filtered complexes]
\label{ex:Filt_complex}
The homology of a bounded from below $\R$-filtered complex over $\F$
is a persistence module. Here is a more detailed description fitting
our purposes and modeling how the action filtration
  arises in Floer theory. Let $(\CC,\p)$ be a complex over $\F$ and
$\Gamma$ a basis of $\CC$, which we require to be countable or
finite. Let $F\colon \Gamma\to \R$ be a bounded from below function
such that $F(\Gamma)$ is discrete and $F(\p v)\leq F(v)$ for all
$v\in \Gamma$. For $w =\sum a_v v\in \CC$ set
$$
F(w)=\max_{a_v\neq 0} F(v)\textrm{ and } F(0)=-\infty.
$$
Then $\CC_s:=\{w\mid F(w)<s\}\subset \CC$ is a subcomplex and the
homology spaces $\CV_s:=\H(\CC_s,\p)$ form a persistence module, which
is graded when $\CC$ is graded. (More generally we
  could have started with a function $F\colon\CC\to\{-\infty\}\cup\R$
  such that $F(x+y)\leq\max\{F(x),F(y)\}$ for all $x,y\in\CC$, and
  $F(0)=-\infty$.)  For instance, when $f$ is a Morse function in
Example \ref{ex:sublevels}, to recover the sublevel homology we can
take as $\CC$ the Morse complex of $f$ and as $\Gamma$ the set of the
critical points $v$ of $f$ with $F(v):=f(v)$. Likewise, under suitable
non-degeneracy conditions (e.g., when the metric is bumpy for closed
geodesics) in Example \ref{ex:Morse}, we obtain the persistence module
$\CV$ when $\CC$ is the underlying Morse complex and $\Gamma$ is the
set of geodesics in question, suitably adjusted in the
  case of closed geodesics.
\end{Example}

Recall furthermore that an \emph{interval persistence module}
$\F_{(a,\,b]}$, where $-\infty<a<b\leq \infty$, is defined by setting
$$
{\F_{(a,\,b]}}_s:=\begin{cases}
  \F & \text{ when } s\in (a,\,b],\\
  0 & \text{ when } s \not\in (a,\,b],
\end{cases}
$$
and $\pi_{st}=\id$ if $a<s\leq t\leq b$ and $\pi_{st}=0$
otherwise. Interval modules are examples of simple persistence
modules, i.e., persistence modules that cannot be decomposed as a
(non-trivial) direct sum of other persistence modules as we will see
in Section \ref{sec:barcodes}.

Morphisms of persistence modules and algebraic operations on persistence
modules are defined in a natural way. In particular, we have exact
sequences and direct sums of persistence modules. However, the next
example provides a word of warning that one should be careful about
the role of structure maps in these constructions

\begin{Example}
  Consider two adjacent intervals $(a,\,b]$ and $(b,\,c]$. Then
  $$
  \F_{(a,\,b]}\oplus \F_{(b,\,c]}\not\cong \F_{(a,\,c]}
  $$
  as persistence modules. Indeed, take $t\in (a,\,b]$ and $s\in
  (b,\,c]$. Then $\pi_{st}=0$ for the module on the left and
  $\pi_{st}=\id$ on the right.

  The natural inclusion map $\F_{(b,\,c]}\hookrightarrow \F_{(a,\,c]}$
  and the quotient map $\F_{(a,\,c]}\to \F_{(a,\,b]}$ are morphisms of
  persistence modules. However, the inclusion
  $\F_{(a,\,b]}\hookrightarrow \F_{(a,\,c]}$ is not. Nor is the
  quotient map $\F_{(a,\,c]}\to \F_{(b,\,c]}$.
\end{Example}  

\subsubsection{Barcodes}
\label{sec:barcodes}

A key fact which we will use in these notes is the normal form or
structure theorem asserting that every persistence module meeting the
above conditions can be decomposed as a direct sum of a countable
collection (i.e., a countable multiset) of interval persistence
modules. Moreover, this decomposition is unique up to re-ordering of
the sum. (In fact, conditions \ref{PM1}--\ref{PM4} are far from
optimal and can be considerably relaxed.) We refer the reader to 
\cite[Thm.\ 3.8]{BV} and \cite{CdSGO16} for proofs of these facts for
a class of persistence modules more general than considered here and
further references, and to, e.g., \cite{CZCG, CB, ZC} for previous or
related results. (Here we only note that our class of persistence
modules directly fits in the framework of \cite[Thm.\
3.8]{BV}. Furthermore, while the interval decomposition fails for
q-tame modules in general, the barcode is still defined and the
isometry theorem holds in this case; see \cite{CdSGO16,Le}.)

The multiset $\CB(\CV)$ of intervals entering this decomposition is
referred to as the \emph{barcode} of $\CV$ and the intervals occurring
in $\CB(\CV)$ as \emph{bars}. The set of end-points of the bars is the
spectrum of the barcode.

Alternatively, and more explicitly, the barcode $\CB(\CV)$ can be described as follows. 
For each $c\in\R$ we define the birth function $\alpha:\CV_c\to[-\infty,c)$ and the death function $\beta:\CV_c\to[c,\infty]$ by
\begin{align*}
\alpha(v)&=\inf\{s<c\ |\ v\in\im(\pi_{sc})\},\\
\beta(v)&=\inf\{s\geq c\ |\ v\in\ker(\pi_{cs})\}.
\end{align*}
In the definition of $\beta$, we adopted the usual convention $\inf\varnothing=\infty$. Property \ref{PM4} implies that $\alpha(v)>-\infty$ if $v\neq0$. 
Let $I$ be an interval of the form $(a,b]$ or $(a,\infty)$, and in the latter case we set $b:=\infty$. We choose $c\in(a,b)$, and set
$n_{I} := \dim(A/(B+C))$,
where
\begin{align*}
A & =\{v\in\CV_c\ |\ \alpha(v)\leq a,\ \beta(v)\leq b\},\\
B & =\{v\in\CV_c\ |\ \alpha(v)< a,\ \beta(v)\leq b\},\\
C & =\{v\in\CV_c\ |\ \alpha(v)\leq a,\ \beta(v)< b\}.
\end{align*}
Property \ref{PM2} guarantees that $A$ is finite dimensional, and therefore the multiplicity $n_{I}$ is finite. A linear algebra argument implies that $n_{I}$ is well defined independently of the choice of $c\in(a,b)$; see \cite[Sec.\ 3]{GGM}. Then $I$ is a bar in $\CB(\CV)$ of multiplicity $n_I$ if $n_I>0$, and not a bar otherwise.

\begin{Example}[Singular value decomposition]
\label{ex:sing-value}
In the setting of Example \ref{ex:Filt_complex}, one can show that one
can find a new basis of the complex $\CC$, which we denote by
$\Gamma'=\{x_i,y_i,z_j\}$, with the properties that
\begin{itemize}
\item $\p x_i=y_i$ and $\p z_j=0$, and
\item $\Gamma'$ is \emph{orthogonal} in the sense that
  $f\big(\sum a_v v\big)=\max f(a_v v)$ where now $v\in\Gamma'$.
\end{itemize}
Such a basis $\Gamma'$ is called a \emph{singular value
  decomposition}. Then the barcode of the persistence module
$\CV_s:=\H(\CC_s,\p)$ is formed by the intervals
$\big(F(y_i), F(x_i)\big]$ and $\big(F(z_j), \infty\big)$. Note that
while the barcode associated with a persistence module is unique, up to
reordering, a singular value decomposition of a filtered complex is not
in general unique.
\end{Example}

\begin{figure}[t]
\begin{center}
\includegraphics{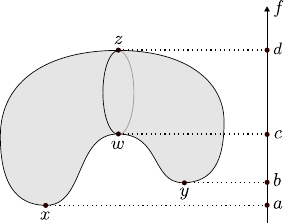}
\caption{A Morse function $f:S^2\to\R$ with four critical points and three critical values.}
\label{f:Morse_function}
\end{center}
\end{figure}

\begin{Example}
  Consider the Morse function $f\colon S^2\to \R$ of Figure
  \ref{f:Morse_function}, which has exactly four
    critical points critical points with critical values
    $a<b<c<d$. The barcode of the associated sublevel homology
    persistence module is $(a,\infty)$, $(b,c]$, and
    $(d,\infty)$.
\end{Example}

The normal form theorem gives a one-to-one correspondence between
isomorphism classes of persistence modules and barcodes, i.e.,
multisets of intervals satisfying some natural additional conditions
corresponding to \ref{PM1}--\ref{PM4}.

One of these conditions is of particular interest to us.  For
$\eps>0$, denote by $\fb_\eps(\CV,s)$ or just $\fb_\eps(s)$ the number
of bars $(a,\,b]$ in $\CB(\CV)$ with $a<s$ of length $b-a> \eps$,
counted with multiplicity. We call $\fb_\eps$ the \emph{barcode
  function} or the barcode growth function.  This is the numerical
invariant of $\CV$, which is essential for our purposes.

Clearly, $\fb_\eps(s)$ is locally constant in $s$ on the complement to
$\CS$ and in $\eps>0$ on the complement of
$\CS-\CS:=\{s-s'\mid s,\, s'\in \CS\}$. The latter can in general have
positive measure. For instance, it can be an interval; cf.\
\cite[Sec.\ 3.2.3]{CGG:Entropy} and \cite[p.\
87]{GO}. Furthermore, $\CS-\CS$ can fail to be closed when $\CS$ is closed but not compact.

\begin{Lemma}
  \label{lemma:b-eps}
  For every $\eps>0$ and $s\in\R$, we have $\fb_\eps(s)<\infty$.
\end{Lemma}

The lemma is quite standard (see, e.g., \cite[Cor.\ 3.5]{CdSGO16}) and
is in fact a consequence of just \ref{PM2}, and we include a proof
below only for the sake of completeness.

By the lemma, $\fb_\eps(s)$ is right semi-continuous in $\eps$ and
left semi-continuous in $s$, i.e.,
\begin{equation}
  \label{eq:bf-semicont}
  \fb_{\eps'}(s')=\fb_\eps(s)
\end{equation}
whenever $\eps'\geq \eps$ is close to $\eps$ and $s'\leq s$ is close
to $s$. Here we are using the condition that the inequalities in the
definition of $\fb_\eps(s)$ are strict.

\begin{proof}[Proof of Lemma \ref{lemma:b-eps}]
  Assume the contrary: for some $s$ and $\eps>0$, there exists an
  infinite sequence of intervals $I_i=(a_i,\, b_i]$ with $a_i<s$ and
  $b_i-a_i>\eps$. We will show that then there exists an interval
  $[t',\, t]$ contained in an infinite number of intervals
  $I_i$. Hence the rank of the map $\pi_{t't}$ is infinite which is
  impossible due \ref{PM2}.

  To this end, note that there are two, not necessarily mutually exclusive
  possibilities: an infinite subsequence of intervals is in
  $(-\infty, s]$ or an infinite subsequence of intervals contains $s$
  as an interior point. Passing to those subsequences, we will treat
  these cases separately.

  In the former case, the intervals are actually contained in
  $(s_0,\, s]$ by \ref{PM4}. Passing to a subsequence again, we can
  assume that $b_i\to b$ for some $b\in (s_0+\eps,\, s]$. Then for
  some $t'<t\leq b$, close to $b$, an infinite number of intervals
  contain $[t',\, t]$.  Similarly, in the latter case there exists a
  pair of points $t'<t$ arbitrarily close to $s$ such that an infinite
  number of intervals contain $[t',\, t]$.
\end{proof}

It is worth keeping in mind that the total number of bars beginning
below $s$ can be infinite. We set
$\fb_\eps(\CV):=\fb_\eps(\CV, \infty)$, to be the total number of bars
of length greater than $\eps>0$. In general, this number can also be
infinite. However, it is finite when $\CV$ is bounded from above,
i.e., $\CS$ is bounded from above, by Lemma \ref{lemma:b-eps} below.

\subsection{Barcode function for Lagrangian Floer homology}
\label{sec:Lagr}
In this section, we briefly discuss the definition of the barcode
function for the filtered Lagrangian Floer homology. The reason that
this case requires a special treatment is that the filtered
Lagrangian or Hamiltonian Floer homology rarely is a persistence
module in the sense of Section \ref{sec:persistence}.  There are
various settings and levels of generality one could work with to
have this homology defined.  For the sake of simplicity, we focus on
the case of Hamiltonian isotopic closed monotone Lagrangian
submanifolds, adopting the framework from \cite{Us} with only minor
modifications.  (For such Lagrangian submanifolds, Floer homology was
originally defined in \cite{Oh}, albeit in a somewhat different
algebraic setting.) We refer the reader to that work and, of course,
to \cite{FOOO} for a much more detailed treatment and further
references.

\subsubsection{Background assumptions, conventions and notation}
\label{sec:conv}
For the reader's convenience, we set here our conventions and notation
and briefly recall some basic definitions; see, e.g., \cite{Us} for
further details and references. The reader may want to consult this
section only as needed.

We assume that the underlying symplectic manifold $(M,\omega)$ is
either closed or ``tame'' at infinity (e.g., convex) so that the
Gromov compactness theorem holds; see, e.g., \cite{McDS}. All
Lagrangian submanifolds are assumed to be closed unless explicitly
stated otherwise, and \emph{monotone}, i.e., for some $\kappa\geq 0$,
we have $\left<\omega, A\right>=\kappa \left<\mu_L,A\right>$ for all
$A\in\pi_2(M,L)$, where $\mu_L\in \H^2(M,L;\Z)$ is the Maslov class.
Then $M$ is also monotone with monotonicity constant $2\kappa$, i.e.,
$\left<\omega, A\right>=2\kappa \left<c_1(TM),A\right>$ for all
$A\in\pi_2(M)$. As in \cite{Us}, we define the \emph{minimal Maslov
  number} of $L$ as the positive generator $N_L$ of the subgroup of
$\Z$ generated by $\left<\mu_L, w\right>$ for all maps
$w\colon \AA\to M$ from the cylinder $\AA:=S^{1}\times [0, 1]$ to $M$
sending the boundary $\p\AA$ to $L$. When this group is trivial, we
set $N_L=\infty$.  In what follows, we require that $N_L\geq 2$ unless
$\kappa=0$.  Note that this definition allows $[\omega]$ to vanish on
$\pi_2(M,L)$ and thus includes weakly exact Lagrangian submanifolds.
The ground field $\F$ is required to have characteristic two; e.g.,
$\F=\F_2$.

Alternatively, one can require $L$ to be oriented, relatively spin and
weakly unobstructed after a bulk deformation as in \cite{FOOO} and
replace the coefficient field by a field of zero characteristic.

A \emph{Hamiltonian diffeomorphism} $\varphi=\varphi_H=\varphi_H^1$ is
the time-one map of the time-dependent flow (i.e., a \emph{Hamiltonian
  isotopy}) $\varphi^t=\varphi_H^t$ of a 1-periodic in time
Hamiltonian $H\colon S^1\times M\to\R$, where $S^1=\R/\Z$.  The
Hamiltonian vector field $X_H$ of $H$ is defined by
$i_{X_H}\omega=-dH$.  All Hamiltonians are assumed to be compactly
supported. In some instances it is more convenient to treat $\varphi$
as an element of the universal covering of the group of Hamiltonian
diffeomorphisms.

The $k$-th iterate $\varphi^k$ is viewed as the time-$k$ map of
$\varphi^t_H$. The $k$-periodic \emph{points} of $\varphi$ are in
one-to-one correspondence with the $k$-periodic \emph{orbits} of $H$,
i.e., of the time-dependent flow $\varphi_H^t$. A $k$-periodic orbit
$x$ of $H$ is said to be \emph{non-degenerate} if the linearized
return map $D_{x(0)}\varphi^k \colon T_{x(0)}M \to T_{x(0)}M$ has no
eigenvalues equal to one. A Hamiltonian $H$ is \emph{non-degenerate}
if all of its 1-periodic orbits are non-degenerate and \emph{strongly
  non-degenerate} when all of its periodic orbits are non-degenerate.

Recall that the \emph{Hofer norm} of $\varphi$ is defined as
$$
\|\varphi\|_{\hn}=\inf_{H}\int_{S^1}\big(\max_M H_t-\min_M H_t\big)\, dt,
$$
where the infimum is taken over all 1-periodic in time Hamiltonians
$H$ generating $\varphi$, i.e., $\varphi=\varphi_H$; see \cite{Ho,
  LMcD, Po:Hofer, Vi}.  The \emph{Hofer distance} between two
Hamiltonian isotopic Lagrangian submanifolds $L$ and $L'$ is
\begin{equation}
  \label{eq:d-Hofer}
d_{\hn}(L,L')=\inf\left\{\|\varphi\|_{\hn}\mid \varphi(L)=L'\right\};
\end{equation}
see \cite{Ch}. When $M$ is compact (or convex at infinity and
$\varphi$ is compactly supported) we also have the spectral norm
$\gamma(\varphi)\leq \|\varphi\|_{\hn}$ of $\varphi$ defined as
\begin{equation}
  \label{eq:spectral}
\gamma(\varphi):=\cf(\varphi)+\cf(\varphi^{-1}),
\end{equation}
where $\cf$ is the spectral invariant associated with the fundamental
class of $M$ and we treat $\varphi$ as an element of the universal
covering of the group of Hamiltonian diffeomorphisms. The spectral
norm is subadditive.  We refer the reader to, e.g., \cite{Oh:gamma,
Sc, Vi} and also, e.g., \cite{CS, EP, FGS, FS, GG:gaps,
  KS, Po:Book, Us0, Us1}, for the original treatment and a detailed
discussion of spectral invariants and these norms, and for further
references. The (ambient) spectral distance $\gamma(L,L') \leq d_{\hn}(L,L')$ is
defined similarly to \eqref{eq:d-Hofer}:
\begin{equation}
  \label{eq:d-gamma}
d_\gamma(L,L'):=\inf\{\gamma(\varphi)\mid \varphi(L)=L'\}.
\end{equation}
Here, of course, we assume that $L$ and $L'$ are Hamiltonian isotopic.

Finally, when $L$ meets certain additional conditions one can also
define the intrinsic spectral distance. Namely, assume that $L$ is
wide in the sense of \cite{BC}, i.e., $\HF(L)=\H(L)\otimes\Lambda$,
where $\Lambda$ is the Novikov ring; see Section
\ref{sec:Floer-complex}. The intrinsic distance between $L$ and $L'$
is defined in a similar fashion but now by using Lagrangian spectral
invariants. Among wide Lagrangian submanifolds are the zero section of
a cotangent bundle and the ``equator'' $\RP^n\subset \CP^n$. On the
other hand, displaceable Lagrangian submanifolds have $\HF(L)=0$ and
hence are not wide.  The intrinsic distance bounds the ambient
distance from below; see, e.g., \cite{KS} and references therein.

\subsubsection{Filtered Floer complex and homology}
\label{sec:Floer-complex}
Let $M$ be a symplectic manifold which is supposed to be sufficiently
``tame'' at infinity (e.g., compact or convex) to guarantee that
compactness theorems hold, and let $L$ and $L'$ be closed monotone
Lagrangian submanifolds intersecting transversely. We assume that $L$
and $L'$ are Hamiltonian isotopic to each other, i.e., there exists a
Hamiltonian isotopy $\varphi_F^t$, $t\in [0, 1]$, such that
$L'=\varphi_F(L)$. For the time being we will treat the isotopy
$\varphi_F^t$ as a part of the data. Furthermore, we require that
$N_{L}\geq 2$, where $N_{L}=N_{L'}$ is the minimal Chern number as
defined in Section \ref{sec:conv}.

Let $\PP=\PP(L,L')$ be the space of smooth paths in $M$ from $L$ to $L'$
and $\pi_1(M;L,L')$ be the set of its connected components. For
instance, the intersection points of $L$ and $L'$ are elements of
$\PP$; however, these elements might be in different connected
components of $\PP$. Fix a reference path $x_\fc$ in each
$\fc\in \pi_1(M;L,L')$. A \emph{capping} $w$ of a path $x\in\fc$ is a
homotopy of $x$ to $x_\fc$ (with end-points on $L$ and $L'$), taken up
to a certain equivalence relation. Namely, two cappings $w$ and $w'$
are equivalent if and only if the cylinder $v$ obtained by attaching
$w'$ to $w$ with the reversed orientation has zero symplectic area and
zero Maslov number. Furthermore, we say that two such cylinders are
equivalent when their symplectic areas and Maslov numbers are equal.
When $w$ and $w'$ are not equivalent, we call $v$, taken up to this
equivalence relation, and also $(x,w')$ a \emph{recapping} of $(x,w)$
and write $(x,w')=:(x,w)\# v$. We usually suppress a capping in the
notation. One can assign a well-defined Maslov index to a capped path
from $L$ to $L'$ by also fixing a trivialization of $TM$ along
$x_\fc$.

For the sake of simplicity, the ground field $\F$ throughout the notes
is $\F_2$. (However, in the Hamiltonian setting of Theorem
\ref{thm:A} and in Theorems \ref{thm:B} and \ref{thm:C} we can work with
any ground field $\F$.) The Floer complexes we mainly consider are
finite-dimensional vector spaces over the ``universal'' \emph{Novikov
  field} $\Lambda$ formed by formal sums
\begin{equation}
  \label{eq:lambda}
\lambda=\sum_{j\geq 0} f_j T^{a_j},
\end{equation}
where $f_j\in \F$ and $a_j\in\R$ and the sequence
$a_j$ (with $f_j\neq 0$) is either finite or $a_j\to\infty$.

The Floer complex $\CF(L,L')$, or just $\CC$ for brevity, is generated
by the intersections $L\cap L'$ with arbitrarily fixed cappings. The
differential $\p_{\Fl}$ is defined by the standard formula counting
holomorphic strips $u$ with boundary components on $L$ and $L'$,
asymptotic to the intersections with now all possible cappings and the
symplectic area $\omega(v)$ of the recapping $v$ contributing the term
$T^{\omega(v)}$ to the differential.

To be more precise, denote the generators by $x_i$. Thus
\begin{equation}
  \label{eq:CC}
\CC=\bigoplus_i \Lambda x_i.
\end{equation}
Then, assuming that the underlying almost complex structure is
regular, we have
\begin{equation}
  \label{eq:d-Fl1}
\p_{\Fl} x_i=\sum_j \lambda_{ij} x_j,
\end{equation}
where
\begin{equation}
  \label{eq:d-Fl2}
\lambda_{ij}=\sum_v f_v T^{\omega(v)}\in \Lambda
\end{equation}
Here the sum extends over all recappings $v$ of $x_j$ such that the
Maslov index difference of $x_i$ and $x_j\# v$ is 1 and $f_v\in \F$ is
the parity of the number of holomorphic strips $u$ asymptotic to $x_i$
and $x_j$, and such that the original fixed capping of $x_i$ is
equivalent to the capping obtained by attaching $x_j\# v$ to these
strips. Then, with the requirement that $L$ and $L'$ are Hamiltonian
isotopic, $\p_{\Fl}^2=0$.

This complex is not graded, due to our choice of the Novikov field, or
only $\Z_2$-graded.  We denote the Floer homology, i.e., the homology
of $\CC$, by $\HF(L,L')$ or just $\HF(L)$. This is also a
finite-dimensional vector space over $\Lambda$.  The complex $\CC$ and
its homology $\HF(L,L')$ split into a direct sum of complexes
$\CC_{\fc}=\CF_\fc(L,L')$ and homology groups $\HF_\fc(L,L')$ over
$\fc \in \pi_1(M;L,L')$.

The complex $\CC$ is naturally filtered by the action. Namely, let us
assign action equal to 0 to every reference path $x_\fc$ and the
action equal to the negative symplectic area $\CA(x,w)=-\omega(w)$ to a
capped path $(x,w)$.  The differential, \eqref{eq:d-Fl1}, is strictly
action-decreasing, and we obtain the required action filtration on
$\CC$ with the convention that for $\lambda\in\Lambda$ given by
\eqref{eq:lambda}, we have
\begin{equation}
\label{eq:filt1}
\nu(\lambda):=\min \{ a_j\}, \text{ where } f_j\neq 0,
\end{equation}
and
\begin{equation}
\label{eq:filt2}
\CA\left(\sum \lambda_i x_i\right):=\max_i\CA(\lambda_i x_i),
\text{ where } \CA(\lambda_i x_i):=\CA(x_i)-\nu(\lambda_i). 
\end{equation}

With this definition, the complex $\CC$ over $\Lambda$ is almost a
filtered complex in the sense of Example \ref{ex:Filt_complex}. The
only difference is that due to \eqref {eq:filt2} the multiplication by
$\lambda\in\Lambda$ shifts the filtration by $ -\nu(\lambda)$:
\begin{equation}
  \label{eq:filt3}
\CA(\lambda c)=\CA(c)-\nu(\lambda) \text{ for } c\in\CC.
\end{equation}

Purely formally, the above constructions can be summarized as a
\emph{Floer package} comprising the following:
\begin{itemize}

\item The finite-dimensional vector space $\CC$ over $\Lambda$ with a
  fixed set of generators $x_i$; see~\eqref{eq:CC}.

\item The differential $\p_{\Fl}$; see \eqref{eq:d-Fl1}.

\item The action filtration on $\CC$ given by \eqref{eq:filt1} and
  \eqref{eq:filt2} such that $\p_{\Fl}$ is strictly action decreasing.

\end{itemize}

A few remarks are due at this point. First of all, the action
filtration depends on the choice of the reference paths
$x_\fc$. Namely, on every direct summand $\CC_\fc$ a change of the
reference path shifts the filtration by a constant. Thus the
filtration is well-defined only up to independent shifts on these
summands. Clearly, this ambiguity does not affect the barcode of $\CC$
introduced further down in Section \ref{sec:Lagrangian-barcode}. In particular, the
bar length and the number of bars of a given length and
$\fb_\eps(L,L')$ are independent of the choice of reference paths.

Furthermore, while the generators of $\CC$ are (capped) intersections,
the differential genuinely depends on the background almost complex
structure. However, the resulting complexes
are chain homotopy equivalent with homotopy preserving the filtration.
Thus, the number of bars of a given length and
$\fb_\eps(L,L')$ are again well-defined.

The Floer homology $\HF(L,L')$, as a vector space over $\F$, inherits
the action filtration from $\CC$. To be more specific, we have a
family of vector spaces $\HF^a(L,L')$ over $\F$, parametrized by
$a\in \R$, where $\HF^a(L,L')$ is the homology of the subcomplex of
$\CC$ comprising the chains $\sum \lambda_i y_i$ with action less than
$a$. The inclusion of subcomplexes for $a<a'$ gives rise to natural
maps $\HF^a(L,L')\to \HF^{a'}(L,L')$. However, 
$\HF^a(L,L')$ is not a vector space over $\Lambda$. Furthermore, the family
$\HF^a(L,L')$ is not a persistence module over $\F$ as defined in
Section \ref{sec:persistence-def}. The reason is that this family of
spaces need not be bounded from below and \ref{PM1} and \ref{PM4} can
fail. (When $L$ and $M$
meet certain additional requirements one can have the Floer complex
defined over $\F$ and then this family turns into a persistence module;
see, e.g., Remarks \ref{rmk:filtr-other} and \ref{rmk:atoroidal}.)

Essentially by construction, we have the ``Poincar\'e
duality''. Namely, for a suitable choice of the auxiliary data,
\begin{equation}
  \label{eq:PD0}
\CF(L',L)=\CF(L,L')^* 
\end{equation}
over $\Lambda$ with ``inverted'' filtration.  (Here we used the fixed
basis $\{x_i\}$ to identify the vector spaces $\CF(L',L)$ and
$\CF(L,L')$ and their duals, and then the Floer differential in
$\CF(L',L)$ turns into the adjoint of the Floer differential in
$\CF(L,'L)$.)  As a consequence, $\HF(L',L)=\HF(L,L')^*$ and this is
also true for $\HF^a(L,L')$ with the inverted filtration.

\begin{Remark}[Exact Lagrangians]
  \label{rmk:filtr-other}
  There are other settings where the filtration on $\CC$ is defined
  and has the desired properties. One is when $(M,\omega)$ is an exact
  symplectic manifold (i.e., $\omega$ is exact) meeting certain
  additional requirements at infinity, and $L$ and $L'$ are exact
  Lagrangian submanifolds (i.e., the restrictions of a primitive of
  $\omega$ to $L$ and $L'$ are exact), not necessarily Hamiltonian
  isotopic. Then we can replace $\Lambda$ by the ground field $\F$ and
  the family $\HF^a(L,L')$ is a true persistence module. In this case,
  sometimes we can even allow one of the manifolds to be
  non-compact. For instance, $L$ can be Hamiltonian isotopic to the
  zero section in a cotangent bundle of a closed manifold and $L'$ can
  be a fiber.
\end{Remark}

\subsubsection{Hamiltonian perspective}
\label{sec:filtration}
It is beneficial to look at this complex $\CC=\CF(L,L')$ from a
different perspective and this is where the condition that $L$ and
$L'$ are Hamiltonian isotopic, which has so far been used only
implicitly, comes to the forefront. Namely, applying the inverse
Hamiltonian isotopy $\big(\varphi_F^t\big)^{-1}$ to paths from $L$ to
$L'$ we obtain paths from $L$ to itself and thus a homeomorphism
between $\PP(L,L')$ and $\PP(L,L)$ and a bijection between
$\pi_1(M;L,L')$ and $\pi_1(M;L) :=\pi_1(M;L,L)$. The intersections
$L\cap L'$ turn into Hamiltonian chords from $L$ to itself for
$\big(\varphi_F^t\big)^{-1}$. Likewise, the reference paths $x_\fc$
become the reference paths $y_\fc$ from $L$ to $L$ and a capping of
$y\in \PP(L,L)$ is a homotopy from $y\in\fc$ to $y_\fc$ with
end-points on $L$ up to the same equivalence relation.

This procedure turns the Cauchy--Riemann equation into the Floer
equation
\begin{equation}
  \label{eq:floer_0}
\frac{\p u}{\p s}+J\frac{\p u}{\p t}=-\nabla H_t(u)
\end{equation}
where $(s,t)\in\R\times [0, 1]$ and $u\colon \R\times [0,1]\to M$, and
$H_t= -F_t\circ \varphi_F^t$ is a Hamiltonian generating the inverse
of the isotopy $\varphi_F^t\colon M\to M$ from $L$ to
$L'$. A holomorphic strip with boundary on $L$ and $L'$ and
asymptotic to $L\cap L'$ becomes a solution $u$ of the Floer equation
with boundary on $L$ asymptotic to Hamiltonian chords. The Floer
differential now counts solutions of the Floer equation. The
Hamiltonian action of a capped path $\by=(y,w)$ is given by the
standard formula
\begin{equation}
  \label{eq:Action}
\CA(\by)=-\int_w\omega+ \int_0^1 H_t(y(t))\,dt,
\end{equation}
and the Floer equation is the $L^2$-anti-gradient flow equation for
$\CA_H$, i.e.,
\begin{equation}
  \label{eq:floer_1}
\p_s u =-\nabla_{L^2}\CA_H(u).
\end{equation}
As a consequence, when $u$ is asymptotic to $\by_1$ at $+\infty$ and
to $\by_0$ at $-\infty$, we have
\begin{equation}
  \label{eq:Energy-Action}
E(u)=\CA_H(\by_1)-\CA_H(\by_0),  
\end{equation}
where the energy $E(u)$ is given by
\begin{equation}
\label{eq:energy}  
E(u)=\int_{S^1\times\R}\|\p_s u\|^2\,dt\,ds.
\end{equation}

The resulting filtered complex over $\Lambda$ and the filtered
homology $\HF^a(L,L')$ are naturally isomorphic to their counterparts
from Section \ref{sec:Floer-complex}. Of course, in the case of
complexes we assume that in both constructions we use matching
auxiliary structures.

Note that whether or not $L$ and $L'$ are transverse, only a finite
set of elements in $\pi_1(M;L,L')\cong\pi_1(M;L)$ is
represented by Hamiltonian chords. Thus $\CC_{\fc}=0$ for all but a
finite collection of $\fc\in\pi_1(M;L,L') \cong\pi_1(M;L)$; this
collection might however depend on the Hamiltonian isotopy; cf.\
\cite[Prop.\ 6.2]{Us}.  In the situation we are interested in there is
usually a natural choice of the Hamiltonian isotopy. For instance, in
Section \ref{sec:barcode_entropy-def-Ham}, the isotopy comes from
$\varphi_H^t$ with time interval $[0, k]$ scaled to $[0, 1]$. 

A particular instance of Lagrangian Floer homology is Hamiltonian
Floer homology. Namely let $\varphi_H$ be a Hamiltonian diffeomorphism
of a closed monotone symplectic manifold $M$. Let $L$ be the
diagonal in the symplectic square $M\times M$ equipped with the
symplectic form $\omega\oplus(-\omega)$ and $L'$ the graph of
$\varphi_H$ in $M\times M$. Clearly, $L$ and $L'$ are Hamiltonian
isotopic with the isotopy $\id\times \varphi_H^t$.  The intersections
$L\cap L'$ are in one-to-one correspondence with the fixed points of
$\varphi_H$. The transverse intersection condition is equivalent to
that all fixed points are non-degenerate. The set $\PP(L,L')$ of paths
in $M\times M$ is in one-to-one correspondence with the space of loops
in $M$ and $\pi_1(M\times M; L, L')$ is the set of free homotopy
classes of loops. A capping is an equivalence class of cylinders $w$
connecting a loop $y$ to a reference loop $y_\fc$ in a free homotopy
type $\fc$. Here two cylinders $w_0$ and $w_1$ are equivalent if the
integrals of $\omega$ and $c_1(TM)$ over the torus $v$ obtained by
attaching $w_0$ to $w_1$ (with the reversed orientation) is zero.  The
action filtration is still given by \eqref{eq:Action}, and
\eqref{eq:floer_0}, \eqref{eq:floer_1} and \eqref{eq:Energy-Action}
still hold. Note also that in the Hamiltonian setting the condition
that $M$ is monotone can be further relaxed (within ``conventional''
Floer theory) to that $M$ is weakly monotone (see \cite{HS}) and here
we can use any field $\F$ as the ground field.

As a result of this construction, assuming that the fixed points of
$\varphi_H$ are non-degenerate, we obtain a filtered complex over
$\Lambda$, denoted $\CF(H)$ or $\CF(\varphi_H)$ and called the Floer
complex of $H$. The resulting filtered homology is the (fixed point)
Floer homology of $H$. The complex and homology are naturally graded
but the grading is inessential for our purposes and we ignore it. The
background conditions from Section \ref{sec:conv} on $L$ and the
ambient symplectic manifold translate to that $M$ is monotone. This
condition can be further relaxed to that $M$ is weakly monotone;
\cite{HS}.  Both the complex and the homology break down into the
direct sums over free homotopy classes $\fc$. The key theorem of Floer
theory is that $\HF_\fc^\infty(H)=\H_*(M;\F)\otimes \Lambda$ when
$\fc=1$ and zero otherwise; see, e.g., \cite{Sa} and references
therein. This fact is crucial for proving the existence results for
fixed or periodic points along the lines of Arnold's or Conley's
conjectures. However, it plays little role in our constructions.

\begin{Remark}
  \label{rmk:non-contr}
  In contrast with some other applications of Floer homology,
  working with periodic orbits (or Lagrangian intersections) in all
  free homotopy class is absolutely essential for our
  purposes. Without this, two of our main results -- Theorems
  \ref{thm:B} and \ref{thm:C} -- would fail; see \cite[Sec.\
  2.3]{CGG:Entropy} for more details.
\end{Remark}

\begin{Remark}
  \label{rmk:atoroidal}
  When $\omega$ is atoroidal, i.e., $\omega(v)=0$ for all maps
  $v\colon \T^2\to M$, there is no need to use the Novikov field
  $\Lambda$, and the Hamiltonian Floer homology is defined over
  $\F$. In this case, the filtered Floer homology is a persistence
  module.
\end{Remark}  

\begin{Remark}[Floer homology in the degenerate case]
  \label{rmk:FH-deg}
  The definition of filtered Floer homology extends by continuity to
  the case where $L$ and $L'$ are not transverse or $\varphi_H$ is
  degenerate, by applying a small perturbation to achieve
  transversality and then passing to the limit. One has to be careful
  however with the specifics of the perturbation process to guarantee
  convergence when the action spectrum is dense. On the other hand,
  when $\omega$ is toroidally rational, i.e., the values $\omega(v)$
  for all $v\colon \T^2\to M$ generate a discrete subgroup of $\R$,
  the action spectrum for every $\fc$ is discrete. Then, for $a$ not
  in the action spectrum, the filtered Floer homology of a small
  perturbation stabilizes. One can extend the definition to all $a$
  by, say, taking the left limit in $a$ as in \ref{PM3}. The
  conventional monotonicity or weak monotonicity requirements concern
  only the values $\omega(v)$ or/and $c_1(TM)(v)$ for $v\in
  \pi_2(M)$. These conditions are essential for ensuring that
  $\p_{\Fl}^2=0$ for all $\fc$ by preventing bubbling off, but they do
  not affect the action spectrum for $\fc\neq 0$; see \cite[Sec.\
  3.2.3]{CGG:Entropy}.  The monotonicity condition can also be used to
  turn the filtered Floer homology for $\fc=1$ (but only for this
  class) into a genuine graded persistence module; \cite{KS}. In any
  event, we do not need this construction to define the Lagrangian
  barcode function in the degenerate case, for it is simpler to do
  this directly.
  \end{Remark}

\subsubsection{Lagrangian barcode function}
\label{sec:Lagrangian-barcode}
In the setting of Section \ref{sec:Floer-complex}, assume that $L$ and
$L'$ are transverse and let $\CC:=\CF(L,L')$ be the Floer complex. As
in Example \ref{ex:sing-value}, there exists a basis
$\Gamma'=\{x_i,y_i,z_j\}$ of the complex $\CC$ over $\Lambda$ such
that
\begin{itemize}
\item $\p x_i=y_i$ and $\p z_j=0$, and
\item $\Gamma'$ is \emph{orthogonal} in the sense that
  $\CA\big(\sum a_c c\big)=\max \CA(a_c c)$ where now $c\in\Gamma'$.
\end{itemize}
Such a basis is again called a \emph{singular value decomposition}. We
refer the reader to \cite{UZ} for the proof of existence.  By
definition, the barcode $\CB(L,L')$ associated with $L$ and $L'$ is
formed by the intervals $\big(\CA(y_i),\CA(x_i)\big]$ and
$\big(\CA(z_j), \infty\big)$ with each interval taken up to an
arbitrary shift. In other words, this is essentially a multiset formed
by the length of these intervals: $\CA(x_i)-\CA(x_i)$ and also
$\infty$ with multiplicity equal to the number of generators
$z_j$. However, we prefer to think of the elements of $\CB(L,L')$ as
bars even though here, in contrast with the barcode of a persistence
module, we discard the location of a bar and keep track only of its
length. In other words, in the terminology from \cite[Rmk.\
3.6]{CGG:Entropy} the bars are unpinned.  The barcode $\CB(L,L')$,
defined in such a way, is unique, i.e., independent of the singular
decomposition (see \cite{UZ}), while the singular value decomposition
is of course not. Considering unpinned bars in this context is
natural: multiplication of the pair $x_i$ and $y_i$ or of $z_j$ and
$\infty$ by $\lambda\in\Lambda$ shifts the bar by $-\nu(\lambda)$ due
to \eqref{eq:filt2} or \eqref{eq:filt3}.

\begin{Definition}
  \label{def:barcode-fn-Lagr}
The barcode function $\fb_\eps(L,L')$ of $L$ and $L'$ is the number of
bars in $\CB(L,L')$ of length strictly greater than $\eps> 0$. 
\end{Definition}

The complexes $\CF(L',L)=\CF(L,L')^*$ and $\CF(L,L')$ have the same
barcode:
$$
\CB(L,L')=\CB(L',L);
$$
cf.\ \cite{UZ}. Hence,
\begin{equation}
  \label{eq:PD}
\fb_\eps(L,L')=\fb_\eps(L',L).
\end{equation}
We also set $\fb(L,L'):=|\CB(L,L')|$ to be the
total number of bars in the barcode. Then
$$
|L\cap L'|=\dim_\Lambda\CF(L,L') = 2\fb(L,L')-\dim_\Lambda\HF(L,L')\geq
\fb(L,L').
$$
In particular, $\fb_\eps(L,L')$ gives a lower bound for the number of
intersections:
\begin{equation}
\label{eq:intersections-b}
|L\cap L'| \geq \fb_\eps(L,L').
\end{equation}

Clearly,
$$
\CB\big(\psi(L),\psi(L')\big)=\CB(L,L')
$$
for any symplectomorphism $\psi\colon M\to M$, and hence 
 \begin{equation}
  \label{eq:inv}
\fb_\eps\big(\psi(L),\psi(L')\big)=\fb_\eps(L,L').
\end{equation}
Furthermore, $\fb_\eps(L,L')$ is constant as a function of $\eps$ on
any interval in the complement of the union of the sets
$\CS_\fc-\CS_\fc$. 

A crucial feature the barcode is that it is fairly insensitive to
small perturbations of the Lagrangian submanifolds with respect to the
Hofer metric. Namely, assume that $d_{\hn}(L',L'')<\delta/2$. Then
\begin{equation}
  \label{eq:insensitive1}
\fb_{\eps+\delta}(L,L') \leq \fb_\eps(L,L'')\leq \fb_{\eps-\delta}(L,L').
\end{equation}
We refer the reader to \cite{KS, PRSZ, UZ} for the proof.

It is convenient, although strictly speaking not necessary, to
introduce the \emph{Floer graph} $G$ associated with a Floer package
$\CC$; cf.\ \cite{CGG:HZ}. The vertices of $G$ are the generators
$x_i$ corresponding to the intersections $L\cap L'$, but not a
singular value decomposition. For each non-zero term $f T^a $ in
$\lambda_{ij}$ (see \eqref{eq:d-Fl1} and \eqref{eq:d-Fl2}) we connect
$x_i$ to $x_j$ by an arrow and label that arrow by the exponent
$a$. (Thus $G$ is a directed graph with finitely many vertices but
possibly infinitely many edges.) The length of an arrow is the action
difference $\CA(x_i)-\CA(T^ax_j)=\CA(x_i)-\CA(x_j)+a$, i.e., the
energy of any underlying Floer cylinder. We say that $x=x_i$ is
$\eps$-isolated if every arrow from or to $x$ has length strictly
greater that $\eps$. For instance, $x$ is $\eps$-isolated if every
Floer cylinder asymptotic to $x$ at $\pm\infty$ has energy strictly
greater than $\eps$. (The converse need not be true.) We will need the
following purely algebraic fact.

\begin{Proposition}
  \label{prop:isolated}
  Assume that $G$ has $p$ $\eps$-isolated vertices. Then
  $\fb_\eps(\CC) \geq p/2$.
\end{Proposition}

The proof of this proposition, which we will leave to the reader as an
exercise, is contained in \cite{UZ} although somewhat implicitly. A
direct proof can be found in \cite{CGG:Entropy}.

Up to this point we have assumed that $L\pitchfork L'$.  To extend the
barcode counting function $\fb_\eps(L,L')$ to the situation where $L$
and $L'$ need not be transverse, set
\begin{equation}
  \label{eq:b-eps3}
\fb_\eps(L,L'):=\liminf_{\tL'\to L'}\fb_\eps(L,\tL')\in \Z.
\end{equation}
Here the limit is taken over Lagrangian submanifolds
$\tL'\pitchfork L$ which are Hamiltonian isotopic to $L'$ and converge
to $L'$ in the $C^\infty$-topology (or at least in the
$C^1$-topology).  Note that, as a consequence,
$d_{\hn}(\tL',L')\to 0$. By \eqref{eq:PD}, we could alternatively
require that $\tL'$ is Hamiltonian isotopic to $L$, transverse to $L'$
and converges to $L$.  Furthermore, since $\fb_\eps(L,L')\in\Z$, the
limit in \eqref{eq:b-eps3} is necessarily attained, i.e., there exists
$\tL'$ arbitrarily close to $L'$ such that
$\fb_\eps(L,L')=\fb_\eps(L,\tL')$.

With this definition, $\fb_\eps(L,L')$ is monotone increasing as
$\eps$ decreases to 0, and \eqref{eq:PD}, \eqref{eq:inv} and
\eqref{eq:insensitive1}
continue to hold. For instance, to prove the first inequality in
  \eqref{eq:insensitive1}, note first that in \eqref{eq:b-eps3} we
  could have replaced the lower limit over $\tL'\to L'$ by the lower
  limit over $\tL\to L$. Then
  $$
  \fb_{\eps+\delta}(L,L'):=\liminf_{\tL\to L}\fb_{\eps+\delta}(\tL,L')\leq
  \liminf_{\tL\to L}\fb_\eps(\tL,L'')=: \fb_\eps(L,L''),
  $$
as desired.

Furthermore, $\fb_\eps(L, L')$ gives a lower bound on the number of
transverse intersections which is in some sense stable under small
perturbations with respect to the Hofer distance. To be more precise,
assume that Lagrangian submanifolds $L$, $L'$ and $L''$ are
Hamiltonian isotopic, $L''\pitchfork L$ and
$d_{\hn}(L',L'')<\delta/2$. Then, regardless of whether $L$ and $L'$
are transverse or not, we have
  \begin{equation}
    \label{eq:intersections}
  |L\cap L''|\geq \fb_{\eps}(L,L'')\geq \fb_{\eps+\delta}(L,L').
\end{equation}
Here the first inequality follows from \eqref{eq:intersections-b}, and
in the second we use \eqref{eq:insensitive1} and the fact that
$d_{\hn}(L',L'')<\delta/2$. These inequalities play a central role in
the proof of Theorem~\ref{thm:A'}.

The above definitions and constructions carry over word-for-word to
Hamiltonian diffeomorphisms. We denote the barcode and the
barcode function associated with $\varphi_H$ by $\CB(\varphi_H)$ and,
respectively, $\fb_\eps(\varphi_H)$. In this case it is sufficient to
assume that $M$ is weakly monotone and we can use any field as the
ground field.

\subsection{Filtered symplectic homology}
\label{sec:SH}
The first versions of symplectic homology were introduced in
\cite{CFH, Vi} and since then the theory has been further
developed. In this section, we only briefly touch upon basic
constructions treating filtered symplectic homology as a persistence
module and following \cite{CGGM:Reeb, CGGM:hyperbolic}. We refer the
reader to, e.g., \cite{Abu, Bo, BO0, BO, Se:biased} for
details.

\subsubsection{Conventions, notation, semi-admissible Hamiltonians, the Floer
equation}
\label{sec:setting}
When working with symplectic homology, it is convenient to adopt
conventions (e.g., signs in the action functional and the Floer
equation) which are different from what we used in Section
\ref{sec:Lagr} for Lagrangian intersections and Hamiltonian
diffeomorphisms.  In this section we spell out these conventions and
notation which are essentially identical to the ones used in
\cite{CGGM:Reeb, CGGM:hyperbolic, GG:LS}, and also recall several
elementary properties of (semi-)admissible Hamiltonians to be used
later.

Let $\alpha$ be the contact form on the boundary $M=\p W$ of a
Liouville domain $W^{2n\geq 4}$. We will also use the same notation
$\alpha$ for a primitive of the symplectic form $\omega$ on $W$. The
grading of Floer or symplectic homology is inessential for our
purposes and we make no assumptions on $c_1(TW)$.  As usual, denote by
$\WW$ the symplectic completion of $W$, i.e.,
$$
\WW=W\cup_M M\times [1,\,\infty)
$$
with the symplectic form $\omega=d\alpha$ extended from $W$ to
$M\times [1,\infty)$ as
$$
\omega := d(r\alpha),
$$
where $r$ is the coordinate on $[1,\,\infty)$. Sometimes it is
convenient to have the function $r$ also defined on a collar of
$M=\p W$ in $W$. Thus we can think of $\WW$ as the union of $W$ and
$M\times [1-\eta,\,\infty)$ for small $\eta>0$ with
$M\times [1-\eta,\, 1]$ lying in $W$ and the symplectic form given by
the same formula.

The action spectrum $\CS(\alpha)\subset (0,\infty)$ is the set of all
actions (aka periods) of closed orbits of the Reeb flow of $\alpha$
on $M$. This is a closed, zero measure set.

Most of the Hamiltonians $H\colon \WW\to \R$ considered in this
section are constant on $W$ and depend only on $r$ outside $W$, i.e.,
$H=h(r)$ on $M\times [1,\,\infty)$, where the $C^\infty$-smooth
function $h\colon [1,\,\infty)\to \R$ is required to meet the
following three conditions:
\begin{itemize}
\item $h$ is strictly monotone increasing;
\item $h$ is convex, i.e., $h''\geq 0$, and $h''>0$ on $(1,\, \rmax)$
  for some $\rmax>1$ depending on~$h$;
\item $h(r)$ is linear, i.e., $h(r)=ar-c$, when $r\geq \rmax$.
\end{itemize}
In other words, the function $h$ changes from a constant on $W$ to
convex in $r$ on $M\times [1,\, \rmax]$, and strictly convex on the
interior, to linear in $r$ on $M\times [\rmax,\, \infty)$.

We will refer to $a$ as the \emph{slope} of $H$ (or $h$) and write
$a=\slope(H)$. The slope is often, but not always, assumed to be
outside the action spectrum of $\alpha$, i.e.,
$a\not\in\CS(\alpha)$. We call $H$ \emph{admissible} if
$H|_W=\const<0$ and \emph{semi-admissible} when $H|_W\equiv 0$. (This
terminology differs from the standard usage, and we emphasize that
\emph{admissible Hamiltonians are not semi-admissible}.) When $H$
satisfies only the last of the three conditions, we call it
\emph{linear at infinity}.

The difference between admissible and semi-admissible Hamiltonians is
just an additive constant: $H- H|_W$ is semi-admissible when $H$ is
admissible. Hence the two Hamiltonians have the same filtered Floer
homology up to an action shift. For our purposes, semi-admissible
Hamiltonians are notably more suitable due to the $H|_W\equiv 0$
normalization.

The Hamiltonian vector field $X_H$ is determined by the condition
$$
\omega(X_H,\, \cdot)=-dH,
$$
and, on $M\times [1,\,\infty)$, 
$$
X_H=h'(r) R_\alpha,
$$
where $R_\alpha$ is the Reeb vector field. We denote the Hamiltonian
flow of $H$ by $\varphi_H^t$, the Reeb flow of $\alpha$ by
$\varphi_\alpha^t$, where $t\in \R$, and the Hamiltonian
diffeomorphism generated by $H$ by $\varphi_H:=\varphi_H^1$.

By a \emph{$\tau$-periodic orbit} of $H$ we will
mean one of several closely related but distinct objects. It can be a
$\tau$-periodic orbit of $\varphi_H$ and then
$\tau\in\N$. Alternatively, it can stand for a $\tau$-periodic orbit
of the flow $\varphi_H^t$ with $\tau\in (0,\,\infty)$. Furthermore,
working with periodic orbits of flows or maps, we might or not have
the initial condition fixed. For instance, without an initial
condition fixed, a non-constant 1-periodic orbit of the flow of $H$
gives rise to a whole circle of 1-periodic orbits (aka fixed points)
of $\varphi_H$. Likewise, a prime $\tau$-periodic orbit of $\varphi_H$
comprises $\tau$ $\tau$-periodic points.  In most cases the exact
meaning should be clear from the context and is often immaterial; when
the difference is essential we will specify whether an orbit is of the
flow or the diffeomorphism and if the initial condition is fixed or
not.

Every $T$-periodic orbit $z$ of the Reeb flow with $T<a=\slope(H)$
gives rise to a 1-periodic orbit $\tz=(z,r_*)$ of the flow of $H$ with
$r_*$ determined by the condition
\begin{equation}
  \label{eq:level}
h'(r_*)=T.
\end{equation}
Clearly, $\tz$ lies in the shell $1<r<\rmax$, and we have a one-to-one
correspondence between 1-periodic orbits of $H$ and the periodic
orbits of $\varphi_\alpha^t$ with period $T<a$ whenever
$a\not\in \CS(\alpha)$. In the pair $\tz=(z,r_*)$, we usually view
$\tz$ as a 1-periodic orbit of the flow $\varphi^t_H$ of $H$ or a
circle of fixed points of the Hamiltonian diffeomorphism $\varphi_H$,
while $z$, contrary to what the notation might suggest, is
parametrized by the Reeb flow but not as a projection of $\tz$ to
$M$. (By \eqref{eq:level}, the two parametrizations of $z$ differ by
the factor of $h'(r_*)=T$.) Fixing an initial condition on $z$
determines an initial condition on $\tz$, and the other way around. In
particular, $z$ gives rise to a whole circle $\tz(S^1)$ of fixed
points of $\varphi_H$.

We say that a $T$-periodic orbit $z$ of the Reeb flow is
\emph{isolated} (as a periodic orbit) if for every $T'> T$ it is
isolated among periodic orbits with period less than $T'$. Clearly,
all periodic orbits of $\alpha$ are isolated if and only if for every
$T'$ the number of periodic orbits with period less than $T'$ is
finite.  For instance, a non-degenerate periodic orbit is isolated.
Note that $\tz$ is isolated as a 1-periodic orbit of the flow of $H$
if $z$ is isolated. No fixed point of $\varphi_H$ on $\tz(S^1)$ is
isolated, but $\tz$ is Morse--Bott non-degenerate, as the set of fixed
points $\tz(S^1)$, if and only if $z$ is non-degenerate.

The action functional $\CA_H$ is defined by
$$
\CA_H(\gamma)=\int_\gamma\hat{\alpha}-\int_{S^1} H(\gamma(t))\, dt,
$$
where $\gamma\colon S^1=\R/\Z\to \WW$ is a smooth loop in $\WW$ and
$\hat{\alpha}$ is the Liouville primitive $\alpha$ of $\omega$ on $W$
and $\hat{\alpha}=r\alpha$ on $M\times [1-\eta,\,\infty)$ for a
sufficiently small $\eta>0$. (Note the sign difference with
\eqref{eq:Action}!) More explicitly, when
$\gamma\colon S^1\to M\times [1,\,\infty)$, we have
$$
\CA_H(\gamma)= \int_{S^1} r(\gamma(t))\alpha\big(\gamma'(t)\big)\, dt
- \int_{S^1} h\big(r(\gamma(t))\big)\, dt.
$$
Thus when $\gamma=\tz=(z,r_*)$ is a 1-periodic orbit of $H$, the
action can be expressed as a function of $r_*$ only:
$$
\CA_H(\tz)=A_h(r_*),
$$
where
\begin{equation}
  \label{eq:AH}
  A_h\colon [1,\,\infty)\to [0,\,\infty)\textrm{ is given by } A_h(r)=r
  h'(r)-h(r).
\end{equation}
Sometimes we will also denote this \emph{action function} by $A_H$.
It is easy to see that this is a monotone increasing function;
\cite[Sec.\ 2.2]{CGGM:Reeb}.
Next fix an almost complex structure $J$ on $\WW$ satisfying the following
conditions:
\begin{itemize}
\item $J$ is compatible with $\omega$, i.e., $\omega(\cdot,\, J\cdot)$
  is a Riemannian metric,
\item $J r\p /\p r=R_\alpha$ on the cylinder $M\times [1,\,\infty)$,
\item $J$ preserves $\ker (\alpha) $.  
\end{itemize}
The last two conditions are equivalent to that
\begin{equation}
  \label{eq:complex_strc}
dr\circ J=-r\alpha.
\end{equation}
We call such almost complex structures \emph{admissible}. If the first
condition still holds on $\WW$, and the second and the third
conditions are met only outside a compact set while within a compact
set $J$ can be time-dependent and 1-periodic in time, we call $J$
\emph{admissible at infinity}.

Let $H$ be a Hamiltonian linear at infinity and let $J$ be an
admissible at infinity almost complex structure. The Floer equation is
the still the 
$L^2$-anti-gradient flow equation, \eqref{eq:floer_1}, for $\CA_H$.
However, with the difference in signs, this equation now explicitly reads
\begin{equation}
\label{eq:floer_2}
\p_s u-J\big(\p_t u- X_H(u)\big)=0,
\end{equation}
 where $u\colon \R\times S^1\to \WW$ and $(s,t)$
are the coordinates on $\R\times S^1$ with $S^1=\R/\Z$. By construction, the
function $s\mapsto \CA_H\big(u(s,\,\cdot)\big)$ is decreasing.

Note that the leading term of this equation is the $\p$-operator, as
opposed to the $\bar{\p}$-operator as in \eqref{eq:floer_0}.  In
other words, when $H\equiv 0$, solutions of \eqref{eq:floer_2} are
anti-holomorphic rather than holomorphic curves. Nonetheless the
standard properties of the solutions of the Floer equation readily
translate to our setting, e.g., via the change of variables
$s\mapsto -s$. We will often refer to solutions $u$ of the Floer
equation as \emph{Floer cylinders}.

By construction, the energy-action relation, \eqref{eq:Energy-Action},
still holds with the energy $E(u)$ defined by \eqref{eq:energy}. When
$H$ is semi-admissible, we can be more precise.  Namely, let $u$ be
asymptotic to $\tx=(x, r^+)$ at $-\infty$ and $\ty=(y,r^-)$ at
$+\infty$. Then
\begin{equation}
  \label{eq:Energy-Action2}
E(u)=\CA_H(\tx)-\CA_H(\ty)=A_H(r^+)-A_H(r^-).
\end{equation}
Here, due to our conventions, $r^+\geq r^-$ -- hence the notation.

\subsubsection{Symplectic homology}
\label{sec:SH-def}
Fix a field $\F$ which we suppress in the notation.  Let $H$ be linear
at infinity with the slope outside $\CS(\alpha)$ and $J$ admissible at
infinity. Then the filtered Floer homology $\HF^\tau(H)$ of $H$ over
$\F$ is defined regardless of whether $H$ is non-degenerate or not;
cf.\ Remark \ref{rmk:FH-deg}. Indeed, let $\CS(H)$ be the action
spectrum of $H$. When $\tau\not\in \CS(H)$, we set
$\HF^\tau(H):=\HF^\tau(\tH)$, where $\tH$ is a sufficiently small
perturbation of $H$, depending on $\tau$. For $\tau\in\CS(H)$, we
define $\HF^\tau(H)$ as the left limit of $\HF^{\tau'}(H)$ as
$\tau'\to \tau-$ and $\tau'\not\in\CS(H)$; cf.\ \ref{PM3}. So defined,
the filtered Floer homology is a persistence module over $\F$.

For instance, assume that the Reeb flow is non-degenerate and $H$ is
semi-admissible with $a=\slope(H)\not\in\CS(\alpha)$. Then, as we have
pointed in the previous section, every non-trivial one-periodic orbit
of $H$ has the form $\bz=(z,r_*)$, where $z$ is a closed Reeb orbit of
period less than $a$. Then one can find a non-degenerate perturbation
$\tH$ such that $z$ gives rise to two generators $\hz$ and $\cz$ in
$\CF(\tH)$. If the grading is defined and $m$ is the Conley--Zehnder
index of $z$, the degree of $\hz$ is $m+1$ and the degree of $\cz$ is
$m$. In addition, $\tH$ is a $C^2$-small function on $W$ with each
critical point contributing one generator to the Floer complex.

In contrast with Hamiltonian Floer homology on closed manifolds, here
a homotopy does not in general induce a continuation map between the
homology for two distinct linear at infinity Hamiltonians. Let $H_s$,
$s\in\R$, be a homotopy between two linear at infinity Hamiltonians
$H_0$ and $H_1$, i.e., $H_s$ is a family of linear at infinity
Hamiltonians such that $H_s=H_0$ when $s$ is close to $-\infty$ and
$H_s=H_1$ when $s$ is close to $+\infty$. (In what follows we will
take the liberty to have homotopies parametrized by $s\in [0,\,1]$ or
some other finite interval rather than $\R$.) There are two cases
where a homotopy gives rise to a map in Floer homology.

The first one is when all Hamiltonians $H_s$ have the same slope. Then
the homotopy induces a continuation map
$$
\HF^\tau(H_0)\to \HF^{\tau +C}(H_1)
$$
shifting the action filtration by
$$
C=\int_{-\infty}^\infty
\, \max_{z\in\WW} \, \max\{0,-\p_s H_s (z)\}\,ds.
$$
Moreover, it is well-known and not hard to show that $\HF(H)$ does not
change as long as $\slope(H)$ stays outside of $\CS(\alpha)$.

The second case is when $H_s$ is monotone increasing, i.e., the
function $s\mapsto H_s(z)$ is monotone increasing for all $z\in
\WW$. In particular, the function $s\mapsto \slope(H_s)$ is also
monotone increasing. Note that while $\slope(H_0)$ and $\slope(H_1)$
are still required to be outside $\CS(\alpha)$, the intermediate
slopes $\slope(H_s)$ can pass through the points of
$\CS(\alpha)$. Such a homotopy induces a map
$$
\HF^\tau(H_0)\to \HF^{\tau}(H_1)
$$
preserving the action filtration.

The \emph{filtered symplectic homology} $\SH^\tau(\alpha)$ or $\SH^\tau(W)$ is defined as
\begin{equation}
  \label{eq:SH}
\SH^\tau(W):=\varinjlim_H \HF^\tau(H),
\end{equation}
where traditionally the limit is taken over all Hamiltonians linear at
infinity and such that $H|_W<0$. Since admissible (but not
semi-admissible) Hamiltonians form a co-final family, we can limit $H$
to this class.  When working with this definition, it is useful to
keep in mind that, as is not hard to see,
\begin{equation}
  \label{eq:spectra-conv}
\CS(H)\to\{ 0\}\cup\CS(\alpha)
\end{equation}
uniformly on compact intervals. Furthermore, even though $H|_W=0$ for
semi-admissible Hamiltonians rather than $H|_W<0$, one can take the
limit over all such Hamiltonians in \eqref{eq:SH}. This fact, which
readily follows from the definition, is useful for computations and
this is how we will usually treat \eqref{eq:SH} in what follows.  The
filtered symplectic homology is a persistence module. As a consequence
we have the barcode growth function $\fb_\eps(W,s)$ or just
$\fb_\eps(s)$ associated with this persistence module.

What is even more useful for computations is that no limit is needed
in the definition of the filtered symplectic homology. Namely, we have
$\HF(H)=\SH^a(W)$ for every semi-admissible Hamiltonian $H$ with
$a=\slope(H)\not\in \CS(\alpha)$. Furthermore, for any two such
Hamiltonians $H_0$ and $H_1$ the persistence modules $\HF^\tau(H_0)$
and $\HF^\tau(H_1)$, differ in the obvious sense by a
reparametrization. This reparametrization can be made arbitrary close
to the identity for a suitable choice of $H_0$ and $H_1$. Furthermore,
let us truncate the persistence module $\SH^\tau(W)$ at $\tau=a$ by
turning all bars passing through $a$ into infinite bars. (This is not
the standard truncation of persistence modules.) Then the persistence
module $\HF(H)$ is again a reparametrization of this truncated
symplectic homology, and moreover the reparametrization is arbitrarily
close to the identity for a suitable choice of $H$. We refer the
reader to \cite{CGGM:Reeb} and \cite{CGGM:hyperbolic} for the proofs
of these facts.

\subsection{Topological entropy}
\label{sec:Top_Entropy}
In this section, we will briefly review the definition and basic
properties of topological entropy. We refer the reader to, e.g.,
\cite{KH} for a much more detailed treatment.

\subsubsection{Basics}
\label{sec:Top_Entropy-Basics}
Let $(M,d)$ be a compact metric space and $\varphi\colon M\to M$ a
continuous map.  Suppose we can only distinguish points in $M$ to an
accuracy of $\eps>0$. How many distinct orbits of $\varphi$ can we see
for the first $k\in \N$ iterations? To answer this question, define
the $k$-shadowing metric on $M$ by
$$
d_k(x,y):=\max_{0\leq i\leq k-1}d\big(\varphi^i(x),\varphi^i(y)\big).
$$
Then the number of orbits we can detect is the $\eps$-separated number
$S_\eps(k)$ for the metric $d_k$, i.e., the maximal number of points
in $M$ with pairwise $d_k$-distance greater than $\eps$. A closely
related number is the $\eps$-covering number $C_\eps(k)$ which is the
minimal cardinality of a covering of $M$ by sets of $d_k$-diameter
less than $\eps$. 
These numbers increase with $k$ and decrease with $\eps$ or rather grow as $\eps\to 0^+$.
It is easy to see that
$$
C_{2\eps}(k)\leq S_\eps(k)\leq C_\eps(k)<\infty.
$$
\begin{Definition}
  \label{def:htop}
The \emph{topological entropy} of $\varphi$ is
\begin{align*}
  \htop(\varphi)
  &:=\lim_{\eps\to 0^+}\limsup_{k\to \infty}\frac{\log
    C_\eps(k)}{k} 
    =\lim_{\eps\to 0^+}\limsup_{k\to \infty}\frac{\log S_\eps(k)}{k}
     \in [0,\,\infty] .    
\end{align*}
\end{Definition}

It is not hard to see that the sequence $\log C_\eps(k)$ is
subadditive in $k$, and hence in the first line the limit (not just
the upper limit) exists and is finite for every $\eps>0$. In general,
the topological entropy of $\varphi$ can be infinite. Clearly, the
requirement that $M$ is compact can be replaced by the condition that
$\varphi$ is compactly supported, working with the restriction $\varphi|_K$, where $K$ is any compact set
  containing $\supp \varphi$.

For a flow $\varphi=\{\varphi^t\mid t\in\R\}$, the topological entropy
is defined in a similar fashion with $k$ replaced by continuous
time. One can show that $\htop(\varphi)$ is simply equal
to $\htop(\varphi^1)$, and this fact can also be taken as the
definition of topological entropy for flows.

In the following proposition we collect several basic properties of
topological entropy which follow, with more or less effort, directly
from the definition; see \cite{KH}.

\begin{Proposition}
  \label{prop:htop-properties}
Let as above $\varphi\colon M\to M$ be a continuous map of a compact
metric space or just a continuous compactly supported map. Then
\begin{enumerate}
  \item $\htop(\varphi)$ is independent of the metric and conjugation
    invariant, i.e., $\htop(\psi\varphi\psi^{-1})$ for any
    homeomorphism $\psi\colon M\to M$;

  \item if $\varphi$ is Lipschitz, then $\htop(\varphi)\leq \dim M\cdot \log^+ (L) $ where $\dim M$ is
    the (box) dimension of $M$ and $L$ is the Lipschitz constant of
    $\varphi$ and $\log^+=\max\{0,\log\}$;

  \item $\htop$ is local, i.e., $\htop(\varphi |_K)\leq
    \htop(\varphi)$ for any closed invariant subset $K$;

  \item $\htop$ is homogeneous:
    $\htop(\varphi^{k})=k\htop(\varphi)$ for $k\in\N$ and
    $\htop(\varphi^{-1})=\htop(\varphi)$ when $\varphi$ is a
    homeomorphism;

  \item for two continuous maps $\varphi$ and $\psi$ of two possibly
    different compact spaces we have
    $\htop(\varphi\times \psi)=\htop(\varphi)+\htop(\psi)$.
    
\end{enumerate}

\end{Proposition}

As a consequence of the second assertion, $\htop(\varphi)<\infty$
whenever $M$ is a manifold and $\varphi$ is a compactly supported
$C^1$-map. 

\begin{Example}[Isometry]
  \label{ex:htop-isometry}
  Assume that $\varphi$ is an isometry. Then $\htop(\varphi)=0$ as
  follows immediately from the definition. For instance, a periodic
  map or a right or left parallel transport in a compact Lie group has
  zero topological entropy.
\end{Example}  

\begin{Example}[Doubling map]
  \label{ex:htop-double}
  Let $\varphi\colon S^1\to S^1$ be doubling map: $\varphi(x)=2x$
  where $x\in S^1=\R/\Z$. Then $\htop(\varphi)=\log 2$ and
  $\htop(\varphi)=\log m$ when we replace 2 by $m$; see \cite[Sec.\
  3.2]{KH}. Furthermore, $\htop(\varphi)\geq \log |\deg(\varphi)|$ for
  any $C^1$-map from a closed orientable manifold to
  itself. \cite[Sec.\ 8.3]{KH}. (This is not true for continuous
  maps.)
\end{Example}

\begin{Example}[Linear maps of tori]
  \label{ex:htop-linear}
  Let $A \colon \T^m \to \T^m$, where $\T^m = \R^m/\Z^m$, be a linear
  map, i.e., $A \in \SL(m, \Z)$. Then
  $\htop(A) \geq \sum_{|\lambda|>1} \log |\lambda|$, where $\lambda$'s
  are eigenvalues of $A$. This inequality turns into equality when $A$
  is hyperbolic, i.e, $A$ has no eigenvalues on the unit circle, but
  in general can be strict; see \cite[p.\ 127]{KH}.
\end{Example}

\begin{Example}[Flows on surfaces]
  \label{ex:htop-flow}
  Assume that $\varphi$ is a smooth flow
    (i.e., the flow of a smooth autonomous vector field) on a closed
    surface. Then $\htop(\varphi)=0$. This fact, which follows for
    instance from Theorem \ref{thm:Katok}, is readily believable, but
    in fact requires a proof.  The condition that $\dim M =2$ is
    essential: in dimensions greater than two a flow can have positive
    topological entropy. This is the case, for instance, for the
    mapping torus flow of a diffeomorphism with positive topological
    entropy.
\end{Example}

\begin{Example}[Bernoulli shift]
  \label{ex:htop-Bernoulli}
  Let $M=\{0,1\}^\Z$ be the space of two-sided infinite sequences
  $\{a_i\}$ of 0's or 1's and let $\varphi\colon M\to M$ be the shift
  by one step to the right: $\varphi(\{ a_i \}) = \{ a_{i-1}\} $. Then
  $\htop(\varphi) = \log 2 = 1$; see \cite[Sec.\ 3.2]{KH}. We have
  $\htop(\varphi) = \log m$ when we replace $\{0,1\}$ by an alphabet
  of $m$ letters.
\end{Example}  

\subsubsection{Positive entropy}
\label{sec:Top_Entropy-positive}
Oversimplifying the picture, we can say that there are two distinct,
albeit interconnected, sources of positive entropy. The first 
comes from global algebraic topological complexity of the map. The
second, which is of primary interest to us, can be local and is
related to dynamical or geometrical complexity.

The algebraic topological source can be illustrated by two phenomena:
the growth of free homotopy classes of periodic orbits and the
homological growth.

To see how the growth of the free homotopy classes works we will focus
on the setting where $\varphi$ is a smooth diffeomorphism connected
to the identity by an isotopy. This is the case, for instance, when
$\varphi$ is the time-one map in a flow or a Hamiltonian
diffeomorphism. Then for each periodic orbit we have its free homotopy
class defined. (As a side remark, in the Hamiltonian case this class
is independent of the Hamiltonian isotopy. This is a consequence of the
proof of Arnold's conjecture.) Assume next that the number of free
homotopy classes populated by periodic orbits of period less than or
equal to $k$ grows exponentially with $k$, i.e., this number is
bounded from below by $e^{ak}$ for some $a>0$. Then
$\htop(\varphi)> 0$. The first result of this type in the context of
geodesic flows goes back to \cite{Di} although, strictly speaking, the
setting there is somewhat different. A curious reader may consult
\cite{Pa:book} for further results, including
the one stated above, and references.

Another instance where the algebraic topological complexity ensures
that topological entropy is positive comes homological growth. This is
the situation where the size of the underlying Morse or Floer or
symplectic homology grows exponentially. Consider, for example, the
geodesic flow on a closed Riemannian manifold $Q$ and let $\Omega^L$
be the space of based loops of length less than $L$. Then, for any
field $\F$, the exponential growth rate of $\dim\H_*(\Omega^L;\F)$
gives a lower bound on the topological entropy of the geodesic flow;
\cite{Pa}. This fact generalizes the lower bound from \cite{Di} where
only $\dim\H_0(\Omega^L;\F)$ is considered, and both of these results
have been generalized to Reeb flows on unit (co)tangent bundles in
\cite{MS}. Combined with some geometry and group theory this tells us
that every geodesic flow on a surface of a genus $g\geq 2$ has
positive entropy. The same goes for all closed manifolds admitting a
metric of negative sectional curvature. Several yet more general lower
bounds on topological entropy of Reeb flows via the growth of various
types of symplectic or contact homology have been proved in
\cite{Al:Anosov,Al:Cyl,Al:Leg,AM,Me:Th}.

A different variant of homological growth lower bounds on topological
entropy comes from Yomdin's theorem; \cite{Yo}. The homological
counterpart of this result asserts that
\begin{equation}
  \label{eq:Yomdin-homology}
    \htop(\varphi) \geq \log\big(\textrm{ spectral radius of } \varphi_* \colon
    \H_*(M;\R) \to \H_*(M;\R)\big)
  \end{equation}
  for a $C^\infty$-diffeomorphism $\varphi\colon M\to M$. In other words,
  exponential growth of the norm of $\varphi_*$ guarantees that
  $\htop(\varphi)>0$.  This inequality is consistent with Examples
  \ref{ex:htop-double} and \ref{ex:htop-linear}.  Conjecturally, the
  same is true under less restrictive smoothness requirements and the
  claim is known as Shub's conjecture. However, in Yomdin's proof
  $C^\infty$-smoothness is absolutely essential. We will return to
  Yomdin's theorem later in this section.

  Yet, a dynamical system can have positive topological entropy coming
  from a completely different and local source. The simplest and
  common situation when this happens is when a diffeomorphism
  $\varphi$ or its iterate has a closed invariant set $K$ such that
  $\varphi|_K$ is the Bernoulli shift from Example
  \ref{ex:htop-Bernoulli}. The set $K$ and even the support of
  $\varphi$ can be contained in a small ball and then there are no
  algebraic topological sources of positive entropy. We will elaborate
  on this phenomenon in the next section.

  A closely related ``universal'' source of positive entropy for
  smooth maps is volume growth. Namely, let $M$ be a manifold and
  $L\subset M$ a compact submanifold of any dimension possibly with
  boundary. For a $C^1$-smooth map $\varphi\colon M\to M$, set
  \begin{equation}
    \label{eq:hvol-rel}
  \hvol(\varphi; L)
  :=\limsup_{k\to \infty}\frac{\log^+\vol\big(\varphi^k(L)\big)}{k}, 
\end{equation}
where $\vol$ stands for the Riemannian volume with respect to a fixed
Riemannian metric. We will call $\hvol(\varphi; L)$, which is
obviously independent of the metric, the \emph{volume growth entropy
  of $\varphi$ relative to $L$}.  The celebrated theorem of Yomdin
connects topological entropy and volume growth.

\begin{Theorem}[Yomdin's Theorem, \cite{Yo}]
  \label{thm:Yomdin}
  For a $C^\infty$-smooth diffeomorphism $\varphi$ of a closed manifold,
   $$
  \sup_L \hvol(\varphi;L)=\htop(\varphi).
  $$
\end{Theorem}  

This is a deep theorem which is by an order of magnitude more
difficult than any result on topological entropy we have mentioned so
far. The homological lower bound, \eqref{eq:Yomdin-homology}, is a
consequence of \eqref{eq:Yomdin-vol}.  The key point here is the
inequality
\begin{equation}
  \label{eq:Yomdin-vol}
    \hvol(\varphi; L)\leq \htop(\varphi), 
  \end{equation}
  which holds for every $L$ as long as $\varphi$ is $C^\infty$-smooth,
  and the $C^\infty$-smoothness condition is essential.  Here the
  inequality in one direction immediately follows from
  \eqref{eq:Yomdin-vol}. The opposite inequality is established in
  \cite{Ne} and requires $\varphi$ to be only $C^{1+\eps}$.

  Next, let us replace $M$ by $M\times M$ and $\varphi$ by
  $\id\times\varphi$. We set $L$ to be the diagonal in $M\times
  M$. Then $\varphi^k(L)$ is the graph $\Gamma_k$ of $\varphi^k$.
  Applying \eqref{eq:hvol-rel} in this setting, we will call
   \begin{equation}
    \label{eq:hvol}
  \hvol(\varphi)
  :=\limsup_{k\to \infty}\frac{\log^+\vol\big(\Gamma_k\big)}{k} 
\end{equation}  
the \emph{volume growth entropy} of $\varphi$. By Yomdin's inequality,
\eqref{eq:Yomdin-vol}, we have
\begin{equation}
  \label{eq:Yomdin-vol2}
    \hvol(\varphi)\leq \htop(\varphi) 
  \end{equation}
  since $\htop(\id\times\varphi)=\htop(\varphi)$. This inequality will
  play an essential role in the proofs of our main theorems.

  The volume growth entropy for a flow $\varphi$ is defined in a similar fashion
  by replacing discrete time $k$ by continuous time $t$.

\subsubsection{Hyperbolicity and periodic orbits}
\label{sec:Top_Entropy-hyp}
One of the most important sources of positive entropy, both local and
global, is hyperbolicity.  Recall that a compact invariant set
$K\subset M$ of a diffeomorphism $\varphi:M\to M$ is said to be
hyperbolic if for some Riemannian metric on $M$ there exist positive
constants $\lambda_- < 1 < \lambda_+$ and a splitting
$T_xM=E^-_x\oplus E^+_x$ for every $x\in K$, invariant under
$D\varphi$, such that
$$
\|D\varphi|_{E^-_x}\|\leq \lambda_-\textrm{ and }
\|D\varphi|_{E^+_x}\|\geq \lambda_+
$$
for all $x\in K$. We refer the reader to, e.g., \cite[Sect.\ 6]{KH}
for a detailed discussion of hyperbolicity. Throughout the
notes such sets are required to be compact by definition.

The definition extends to the flows by adding the neutral direction
tangent to the flow lines in the splitting of $TM$ along $K$. It is
useful to keep in mind that the hyperbolic set of a flow $K$ is not
hyperbolic for the time-$T$ map for any $T$ unless $K$ is comprised of
the fixed points of the flow in $K$.

For instance, a hyperbolic fixed point or a hyperbolic periodic orbit is a
hyperbolic set. A more interesting example is the entire torus
$K=\T^n$ in the setting of Example \ref{ex:htop-linear}. Then $K$ is
hyperbolic when $\varphi=A$ is hyperbolic. Another example is a
horseshoe. For our purposes a horseshoe is a hyperbolic set $K$ such
that some iterate of $\varphi|_K$ is conjugate to the Bernoulli shift
from Example \ref{ex:htop-Bernoulli}.

In low dimensions hyperbolic invariant sets are ubiquitous; \cite{LCS}.

\begin{Theorem}[A. Katok, \cite{Ka}]
  \label{thm:Katok}
  For a $C^{1+\eps}$-diffeomorphism $\varphi$ of a surface
\begin{equation}
  \label{eq:htop-Katok}
  \htop(\varphi)=\sup_K \htop(\varphi|_K)
\end{equation}
where the supremum is taken over all hyperbolic invariant sets and, in
fact, horseshoes when $\htop(\varphi)>0$.
\end{Theorem}

A similar result holds for flows without fixed points on 3-manifolds;
\cite{LY,LS}.

Hyperbolicity has many implications but the most important for us is
its connection with topological entropy via the growth of periodic
orbits. Let $P_k(\varphi)$ be the number of periodic points of
$\varphi$ of period less than or equal to $k$. A compact invariant set $K$ is called \emph{locally maximal} or \emph{isolated} when it is the largest invariant set in some neighborhood $U$, i.e.
\[K=\bigcup_{k\in\Z}\varphi^k(U).\] 
The next result is
quite standard; see, e.g., \cite[Thm.\
18.5.1]{KH}.

\begin{Theorem}
  \label{thm:p-htop}
  Let $K$ be a locally maximal hyperbolic set of $\varphi$. Then
\begin{equation}
\label{eq:p-htop}
\htop(\varphi|_K)=\limsup_{k\to\infty} \frac{\log^+ P_k\big(\varphi|_K\big)}{k}.
\end{equation}
\end{Theorem}
A similar identity holds for flows of nowhere vanishing vector fields;
\cite[Thm.\ 5.4.22]{FH}. Then, as a consequence of
\eqref{eq:htop-Katok}, in dimension two for diffeomorphisms and
dimension three for flows, we have
\begin{equation}
  \label{eq:p-htop2}
\htop(\varphi)\leq\limsup_{k\to\infty} \frac{\log^+ P_k(\varphi)}{k}.
\end{equation}

However, it is important to note that in higher dimensions, without
hyperbolicity-type conditions, the growth of periodic orbits, i.e.,
growth of $P_k$, is essentially unrelated to topological entropy.

\begin{Example}[Entropy without periodic points]
  \label{ex:htop-per}
Let $\psi\colon M\to M$ be a map with positive entropy and $R\colon
S^1\to S^1$ an irrational rotation. Then $R$ and
$$
\varphi:=R\times \psi\colon S^1\times M\to S^1\times M
$$
have no periodic points, but $\htop(\varphi)=\htop(\psi)>0$. In
particular, \eqref{eq:p-htop2} fails in higher dimensions. This
example however is extremely non-generic.
\end{Example}

Going in the opposite direction, in higher dimensions one can have
diffeomorphisms $\varphi$ with $\htop(\varphi)=0$ and arbitrarily fast
growth of $P_k(\varphi)$; \cite{Kal}. In this case the left hand side
in \eqref{eq:p-htop2} is zero, but the right hand side can be $\infty$.
In dimension two, one can also have $P_k(\varphi)$ growing arbitrarily
fast (but with $\htop(\varphi)>0$) and moreover this behavior is in
some sense typical; \cite{As}. Thus in dimension two one can also have
infinite right hand side in \eqref{eq:p-htop2}.




\section{Barcode entropy}
\label{sec:barcode_entropy}
In this section, we define several versions of barcode entropy for
Hamiltonian diffeomorphisms and Reeb flows, state and discuss our main
results, and prove basic properties of barcode entropy.

\subsection{Key definitions}
\label{sec:barcode_entropy-def}

\subsubsection{Reeb flows}
\label{sec:barcode_entropy-def-Reeb}
Throughout this section we keep the notation and conventions from
Section \ref{sec:SH}. Thus let $(W, d\alpha)$ be a Liouville
domain. We will also use the notation $\alpha$ for the contact form
$\alpha|_M$ on the boundary $M=\p
W$. 

Fix a ground field $\F$, which we suppress in the notation, and denote
the (non-equivariant) filtered symplectic homology of $W$ over $\F$
for the action interval $[0,\, s)$ by $\SH^s(W)$. The grading of
symplectic homology plays no role and we view $\SH^s(W)$ as an
ungraded vector space over $\F$.  We make no assumptions on the first
Chern class $c_1(TW)$.

Recall from Section \ref{sec:SH} that the symplectic homology together
with the natural maps $\SH^{s_0}(W)\to \SH^{s_1}(\alpha)$ for
$s_0\leq s_1$ is a persistence module and that $\fb_\eps(W,s)$ or just
$\fb_\eps(s)$, when $W$ is clear from the context, stands for its
barcode function. Thus $\fb_\eps(W,s)$ is the number of bars of
length greater than $\eps$, beginning in the range $[0,\, s)$ in the
barcode $\CB(W)$ of $\SH(W)$. This is an increasing function in $s$
and $1/\eps$, locally constant as a function of $s$ in the complement
to $\CS(\alpha)\cup \{0\}$, where $\CS(\alpha)$ is the action spectrum
of $\alpha$.

The barcode entropy of $\alpha$, denoted by $\hbar(\alpha)$, measures
the exponential growth rate of $\fb_\eps(\tau)$ and is defined as
follows.

\begin{Definition}
  \label{def:barcode_entropy}
The \emph{$\eps$-barcode entropy} of $\alpha$ is 
\begin{equation}
  \label{eq:eps-entropy}
  \hbar_\eps(\alpha): =
  \limsup_{s\to\infty}\frac{\log^+ \fb_\eps(s)}{s},
\end{equation} 
where $\log$ is taken base 2, $\log 0=-\infty$ and
$\log^+:=\max\{0,\log\}$, and the \emph{barcode entropy} of $\alpha$
is
\begin{equation}
  \label{eq:entropy}
  \hbar(\alpha):=\lim_{\eps \to 0^+}\hbar_\eps(\alpha)\in [0,\, \infty].
\end{equation}
\end{Definition}
A few comments are due at this point. First of all,
$\hbar_\eps(\alpha)$ increases as
$\eps \to 0^+$, and hence the limit in \eqref{eq:entropy} does
exist. Furthermore, 
$$
\hbar_\eps(\alpha)\leq \hbar(\alpha),
$$
and, as is easy to see, for any $a>0$,
$$
\hbar_\eps(a\alpha)=a^{-1}\hbar_\eps(\alpha)
\textrm{ and }
\hbar(a\alpha)= a^{-1}\hbar(\alpha).
$$

In what follows we will sometimes use the notation
$\hbar_\eps(\varphi)$ and $\hbar(\varphi)$, where $\varphi$ is the
Reeb flow of the contact form $\alpha$ on $M$. Barcode entropy of Reeb
flows was introduced in \cite{FLS} based on similar definitions of
barcode entropy for geodesic flows, \cite{GGM}, and Hamiltonian
diffeomorphisms in \cite{CGG:Entropy} discussed in the next section;
see also \cite{CGGM:Reeb}.

\begin{Remark}
  Hypothetically, $\hbar_\eps(\alpha)$ and $\hbar(\alpha)$ might
  depend on the entire Liouville domain $(W,d\alpha)$ rather than just
  the contact form $\alpha$ on $M=\p W$. However, we are not aware of
  any examples where this happens. (Theorem \ref{thm:C} 
  implies, in particular, that $\hbar(\alpha)$ is completely
  determined by $(M,\alpha)$ when $\dim M=3$.) Moreover, according to
  \cite[Sect.\ 4.2]{FLS}, $\hbar(\alpha)$ is independent of the
  filling at least when $c_1(TW)=0$ and we believe that this is true
  in general as long as $(M,\alpha)$ is the boundary of the Liouville
  domain and hence $\hbar(\alpha)$ is defined.
\end{Remark}

\begin{Remark}[Invariance]
  \label{rmk:Invariance}
  To what extend the interior $\mathring{W}$ of a Liouville domain
  $W$, taken up to an exact symplectomorphism, determines the Reeb
  dynamics on $\p W$? It is the standard fact that the filtered
  symplectic homology is an invariant of $\mathring{W}$. The first
  variant of this statement was proved in \cite{CFHW}, and since then
  the fact has been revisited many times in different contexts; see,
  e.g., \cite{BG,Gu,Hu,Us-SBM}. Hence, every invariant of the
  symplectic homology is also an invariant of the interior. This
  includes, in particular, barcode entropy or more generally the
  barcode growth function. To some degree this is also true even when
  the boundary $\p W$ is not smooth; see Remark \ref{rmk:Invariance2},
  \cite[Cor.\ 2.9]{BG} and \cite{Hu}. As a consequence of Theorem
  \ref{thm:C} below, the topological entropy of the Reeb flow is, less
  obviously, also an invariant when $\dim W=4$. At the same time, some
  dynamics features of the Reeb flow are not invariant. For instance,
  ergodicity is not: for $W$ with an ergodic Anosov--Katok Reeb
  pseudo-rotation on $\p W$ (see \cite{Ka:Izv}), $\mathring{W}$ is
  symplectomorphic to an open ellipsoid when $\dim W = 4$; \cite{ABE}.
\end{Remark}  

\subsubsection{Hamiltonian diffeomorphisms}
\label{sec:barcode_entropy-def-Ham}

As in Section \ref{sec:Lagr}, consider a symplectic manifold $M$ and a
closed monotone (e.g., exact) Lagrangian submanifold $L\subset M$ and
a second Lagrangian submanifold $L'$ Hamiltonian isotopic to $L$. In
particular $M$ is monotone. The ambient manifold $M$ is not required
to be closed, but it has to have a sufficiently nice structure at
infinity, e.g., to be ``tame'' or convex. In addition, the minimal
Maslov number $N_L$ of $L$ needs to be at least 2. Furthermore, let
$\varphi=\varphi_H\colon M\to M$ be a compactly supported Hamiltonian
diffeomorphism. Set $L^k=\varphi^k(L')$. Thus we have the barcode
growth functions $\fb_\eps\big(L,L^k\big)$ and, when $M$ is closed,
$\fb_\eps\big(\varphi^k\big)$ defined. Here in the former case the
ground field is $\F_2$, but in the latter we can work with any field
as long as $M$ is monotone.

The notion of the barcode entropy of $\varphi$ comes in two versions:
relative to $L$ and $L'$, and absolute.

\begin{Definition}[Relative Barcode Entropy]
  \label{def:hbr-rel1}
  The \emph{$\eps$-barcode entropy of $\varphi$ relative to $L$ and $L'$} is
  \begin{equation}
    \label{eq:eps-entropy2}
  \hbr_\eps(\varphi;L,L'):=\limsup_{k\to \infty}\frac{\log^+
    \fb_\eps\big(L, L^k\big)}{k}
\end{equation}
and the \emph{barcode entropy of $\varphi$ relative to these pair} is
  $$
  \hbr(\varphi;L,L'):=\lim_{\eps\to 0+} \hbr_\eps(\varphi; L, L') \in
  [0,\,\infty].
  $$
\end{Definition}
As for the case of Reeb flows, $\hbr_\eps(\varphi; L, L')$ is
increasing as $\eps\to 0+$, and hence the limit in the definition
of $\hbr(\varphi,L)$ exists although \emph{a priori} it can be
infinite.  We also emphasize again that the global Floer homology
$\HF(L)=\HF(L,L')$ is immaterial for this construction beyond the fact
that it is defined. For instance, $L$ can be a small circle in a
surface with $\HF(L)=0$. When $L'=L$ we will use the notation
$\hbr_\eps(\varphi;L)$ and $\hbr(\varphi;L)$. The notion of barcode
entropy for Reeb flows also has a relative analogue relying on wrapped
Floer homology; see \cite{Fe1}.

\begin{Remark}
It is easy to construct examples of Hamiltonian systems where the
filtered Floer or symplectic homology depends on the ground field
$\F$. However, we do not have examples where the barcode entropy
depends on $\F$. Theorem \ref{thm:C} below asserts that this does not
happen for the absolute entropy in dimensions two and three.
\end{Remark}  

\begin{Definition}[Absolute Barcode Entropy]
  \label{def:hbr-bar}
  The \emph{$\eps$-barcode entropy} of $\varphi$ is
  $$
  \hbr_\eps(\varphi):=\limsup_{k\to \infty}\frac{\log^+
    b_\eps\big(\varphi^k\big)}{k}
  $$
  and the \emph{(absolute) barcode entropy} of $\varphi$ is
  $$
  \hbr(\varphi):=\lim_{\eps\to 0+} \hbr_\eps(\varphi) \in
  [0,\,\infty]
  $$
  or, in other words,
  $$
  \hbr(\varphi):=\hbr(\id\times\varphi;\Delta).
  $$
 \end{Definition}
 Here again $\hbr_\eps(\varphi)$ is increasing as $\eps\to 0+$,
 and hence the limit in the definition of $\hbr(\varphi)$ exists.

\subsection{Main results}
\label{sec:barcode_entropy-res}
As stated in the introduction, barcode entropy is closely related to
topological entropy. In this section we state a series of results to
this account.

\begin{TheoremX}
   \label{thm:A}
   Let $\varphi\colon M\to M$ be a Hamiltonian
   $C^\infty$-diffeomorphism of a closed monotone symplectic manifold
   or the Reeb flow on the boundary of a Liouville domain. Then
   \begin{equation}
     \label{eq:A}
   \hbr(\varphi)\leq \hvol(\varphi)\leq \htop(\varphi).
 \end{equation}
\end{TheoremX}

The new part here is the first inequality, for the second one is a
consequence of Yomdin's theorem; see Section
\ref{sec:Top_Entropy-positive} and, in particular,
\eqref{eq:Yomdin-vol2}. Theorem \ref{thm:A} has a relative analogue
which, for the sake of brevity we state only for Hamiltonian
diffeomorphisms.

\begin{customthm}{A$^{\prime}$}
  \label{thm:A'}
  Let $L$ and $L'$ be Hamiltonian isotopic, closed monotone Lagrangian submanifolds
  with minimal
  Chern number $N_L\geq 2$ in a symplectic manifold $M$ and let
  $\varphi\colon M\to M$ be a compactly supported Hamiltonian
  $C^\infty$-diffeomorphism. Then
  \begin{equation}
    \label{eq:A'}
   \hbr(\varphi; L,L')\leq \hvol(\varphi;L)\leq \htop(\varphi).
 \end{equation}
\end{customthm}
 
Here again the second inequality follows from
\eqref{eq:Yomdin-vol}. The Hamiltonian case of Theorem \ref{thm:A}
follows immediately from Theorem \ref{thm:A'} and both results are
originally proved in \cite{CGG:Entropy}. The Reeb version of Theorem
\ref{thm:A} is established in \cite{FLS}; see also \cite{GGM} for the
case of geodesic flows and also \cite{CGGM:Reeb}. A variant of Theorem
\ref{thm:A'} for Reeb flows is proved in \cite{Fe1}. These theorems
also have analogues for metric entropy. Namely, in \cite{CGG:Metric}
in the Hamiltonian and Reeb settings, for a pair $(L,L')$ we construct
a class of invariant probability measures $\mu$ such that
$\hbar(\varphi; L, L')\leq \hm_{\mu}(\varphi)$ for some, but probably
not all, of these measures. Here $\hm_\mu(\varphi)$ is the metric
entropy of the homeomorphism or the flow $\varphi$.

\begin{Remark}[Growth of periodic points]
  \label{rmk:PO-Growth}
  One cannot replace the number of bars $\fb_\eps(\varphi^k)$ or
  $\fb_\eps(L,L^k)$ in the definition of barcode entropy by the total
  number of $k$-periodic points or Lagrangian intersections, while
  keeping Theorems \ref{thm:A} and \ref{thm:A'}. Indeed, as we have
  already pointed out in Section \ref{sec:Top_Entropy-hyp}, in
  dimension two, the number of periodic points can grow arbitrarily
  fast, and moreover super-exponential growth is in some sense
  typical; see \cite{As}. In higher dimensions, a smooth zero-entropy
  map may have super-exponential orbit growth; \cite{Kal}. Hence, in
  both cases the exponential growth rate of the number of periodic
  points could in general be infinite.
\end{Remark}

In the presence of hyperbolic invariant sets we have inequalities
going in the opposite direction.

\begin{TheoremX}
   \label{thm:B}
   Let $\varphi$ be a Hamiltonian diffeomorphism of a
   closed monotone symplectic manifold $M$ or the Reeb flow on the
   boundary $M$ of a Liouville domain, and let $K\subset M$ be a
   hyperbolic invariant subset. Then

   $$
   \hbr(\varphi)\geq \htop(\varphi|_K).
   $$
\end{TheoremX}

This theorem is proved in \cite{CGG:Entropy} in the Hamiltonian case
and \cite{CGGM:Reeb} for Reeb flows; the case of geodesic flows is
treated in \cite{GGM}. The theorem also has a relative variant: see
\cite{Me} for the statement and the proof in the Hamiltonian case and
\cite{Fe2} for the Reeb version.

In dimension two, $\htop(\varphi)$ is the supremum of
$\htop(\varphi|_K)$ over all $K$ as in Theorem \ref{thm:B}; see
\cite{Ka} and also \cite[Suppl.\ S by Katok and Mendoza]{KH}. A
similar result holds for flows on three manifolds; \cite{LY,LS}.
Hence, by Theorem \ref{thm:A'} (or Theorem \ref{thm:A}) and Theorem
\ref{thm:B}, we have the following.

 \begin{TheoremX}
   \label{thm:C}
   Let $\varphi\colon M\to M$ be a Hamiltonian
   $C^\infty$-diffeomorphism of a closed surface or the Reeb flow on the
   boundary of a 4-dimensional Liouville domain. Then
   \begin{equation}
     \label{eq:bar-top}
   \hbr(\varphi)= \hvol(\varphi)=\htop(\varphi).
 \end{equation}
\end{TheoremX}
Surprisingly, here even the second inequality, which does not
explicitly involve the notion of barcode entropy, appears to be new.

A $C^\infty$ generic Hamiltonian diffeomorphism of a closed surface
has positive topological entropy; \cite{KLCN, LCS}. This is also true
for Reeb flows in dimension three; see \cite{CDHR} and also
\cite{CKMS}. Therefore, in dimensions two and three $\hbar>0$ and
$\fb_\eps$ for small $\eps>0$ grows exponentially fast
$C^\infty$-generically.

\begin{Example}
  \label{ex:size_does_not_matter}
  We emphasize that barcode entropy can be positive even when the
  Floer or symplectic homology is bounded or zero. For instance, we
  can have $\hbar(\varphi)>0$ for a Hamiltonian diffeomorphism of a
  closed symplectic manifold even
  though $\HF(\varphi)=\H_*(M)\otimes \Lambda$ is independent of
  $\varphi$.  Furthermore, assume that $W\subset \R^{2n}$ is a
  star-shaped domain with smooth boundary. Then $\SH(W)=0$;
  \cite{Vi:GAFA}. However, $\hbar(\varphi)>0$ when $\htop(\varphi)>0$
  and $\dim W=4$ or when $\varphi$ has a hyperbolic invariant set with
  positive topological zero. Likewise, let $L\subset \R^{2n}$ be a
  monotone Lagrangian torus. Then we can have $\hbar(\varphi;L)>0$ for a compactly
  supported Hamiltonian diffeomorphism $\varphi$, but $\HF(L)=0$.
\end{Example}

Theorem \ref{thm:C}, at least in the Hamiltonian case, does not
generalize to higher dimensions. Namely, as shown in \cite{Ci}, there
exists a Hamiltonian diffeomorphism $\varphi\colon M\to M$, where
$\dim M\geq 6$, such that $\hbr(\varphi)=0$ but $\htop(\varphi)>0$;
\cite{Ci}. Moreover, $\varphi$ can be made autonomous. We do not know
if a similar counterexamples exist for Reeb flows although we see no
reason why they would not; see Section \ref{sec:problems}. 

Recall also from Remark \ref{rmk:non-contr} that it is absolutely
essential here that in the definition of barcode entropy we took the
Floer homology for all free homotopy classes of loops. Otherwise,
Theorems \ref{thm:B} and \ref{thm:C} would fail; see the discussion in
\cite[Sec.\ 2.3]{CGG:Entropy}.

\begin{Remark}[``Classical'' applications]
  The lower bounds between barcode entropy and topological entropy for
  Reeb and geodesic flows along the lines of Theorems \ref{thm:A} and
  \ref{thm:A'} encompass several ``classical'' results relating
  positivity of topological entropy to topology of the underlying
  Riemannian manifold or symplectic topology of the contact
  manifold. The first of these is Dinaburg's theorem, \cite{Di},
  giving a lower bound for the topological entropy of a geodesic flow
  via the growth of the fundamental group. Likewise, the lower bound
  from \cite{Pa} for the topological entropy by the exponential growth
  rate of the homology of the based loop space immediately follows
  from \cite{GGM}. Both of these results have been further generalized
  to Reeb flows on unit (co)tangent bundles in \cite{MS} and
  \cite[Prop.\ 1.8]{FS}, and this generalization is in turn covered
  in \cite{Fe1}. Several yet more general lower bounds on topological
  entropy of Reeb flows via the growth of various types of symplectic
  or contact homology have been proved in
  \cite{Al:Anosov,Al:Cyl,Al:Leg,AM,Me:Th}. Some of these lower bounds
  also follow, directly or at least conceptually, from the results of
  \cite{CGGM:Reeb,FLS,Fe1}.
\end{Remark}

\begin{Remark}[Other variants of barcode entropy]
  Barcode entropy, as defined here by analogy with topological
  entropy, is just one way to encode the growth rate of the barcode as
  a single number. Two other closely related definitions are
  considered in \cite{CGG:Growth}.  In the first one, we replace the
  upper limit of the barcode function for a fixed $\eps>0$ by the
  upper limit $\fb_{\eps_k}\big(L, L^k\big)$ where $\eps_k\to 0+$
  subexponentially and then take the supremum over all such
  sequences. The resulting ``sequential'' barcode entropy bounds the
  barcode entropy from above but the analogue of our Theorem
  \ref{thm:A'} still holds in this case.  In the second one,
  $\fb_\eps(\varphi)$ is replaced by $\sum\beta_i^d$ where the sum
  extends over all finite bars and $d>0$. When $d$ is
  sufficiently large, e.g., $d\geq 2\dim M-1$, we again have Theorem
  \ref{thm:A}. In dimension two, all these variants of barcode entropy
  agree and are equal to the topological entropy.
\end{Remark}

\begin{Remark}[Equivariant barcode entropy]
 In the Reeb setting, we could have also considered the barcode growth
 in the $S^1$-equivariant symplectic homology and equivariant barcode
 entropy. As pointed out in \cite[Rmk.\ 2.4]{GGM}, by the Gysin
 sequence and \cite[Thm.\ 3.1]{BP3S2}, the resulting entropy would
 bound the non-equivariant barcode entropy from above. Hence, we would
 still have an analogue of Theorem \ref{thm:B}. However, we do not
 know if Theorem \ref{thm:A} would hold or not.
\end{Remark}

\begin{Remark}[Slow barcode entropy]
  \label{rmk:slow}
Finally, rather than focusing on the exponential growth rate, one can
 consider a variant of slow entropy encoding the polynomial growth
 rate of the barcode function; cf.\ \cite{BG,FrLS,FS, KT}.  To this end,
 we simply replace $k$ or $s$ by $\log k$ or $\log s$ in the
 denominator of \eqref{eq:eps-entropy} or \eqref{eq:eps-entropy2}. The
 resulting slow entropy is infinite whenever barcode entropy is
 positive, but it is of interest for ``low complexity'' Hamiltonian
 systems. We will revisit the question of polynomial growth Section
 \ref{sec:integrable}.
\end{Remark}

\subsection{Basic properties of barcode entropy}
\label{sec:barcode_entropy-prop}
In this section we list, for the sake of completeness, some basic
properties of barcode entropy focusing on the Hamiltonian case and
closely following \cite[Sec.\ 4.2]{CGG:Entropy}.

\begin{Proposition}[Properties of Barcode Entropy]
\label{prop:hbar-prop}
We have the following:

  \begin{itemize}

  \item[\reflb{BE1}{\rm{(i)}}] For every $k\in\N$, we have
    $\hbr(\varphi^k; L,L') \leq k\hbr(\varphi; L,L')$. In particular,
    $\hbr(\varphi^k)\leq k\hbr(\varphi)$.

  \item[\reflb{BE2}{\rm{(ii)}}] Assume that the products
    $L_0\times L_1$ and $L'_0\times L'_1$ in $M_0\times M_1$ are
    monotone. Then for Hamiltonian diffeomorphisms
    $\varphi_0\colon M_0\to M_0$ and $\varphi_1\colon M_1\to M_1$, we
    have
  $$
  \hbr(\varphi_0\times\varphi_1; L_0\times L_1, L'_0\times L'_1)\leq
  \hbr(\varphi_0; L_0, L'_0)+\hbr(\varphi_1; L_1, L'_1).
  $$
  In particular,
  $\hbr(\varphi_0\times\varphi_1)\leq
  \hbr(\varphi_0)+\hbr(\varphi_1)$.

\item[\reflb{BE3}{\rm{(iii)}}] We have
  $\hbr(\varphi; L,L')=\hbr(\varphi^{-1};L',L)$ and, in
  particular, $\hbr(\varphi^{-1})=\hbr(\varphi)$.

  \item[\reflb{BE4}{\rm{(iv)}}] For any symplectomorphism
    $\psi\colon M\to M$,
  \[
    \hbr(\varphi;L,L')
    = \hbr\big(\psi\varphi\psi^{-1}; \psi(L),
    \psi(L')\big) \text{ and } \hbr(\varphi)
    =\hbr(\psi\varphi\psi^{-1}).
  \]
  As a consequence, for every $k\in\N$,
  \[
    \hbr(\varphi; L,L')=\hbr\big(\varphi;
    \varphi^k(L),\varphi^k(L')\big).
  \]
 
\item[\reflb{BE5}{\rm{(v)}}] For a fixed Hamiltonian diffeomorphism
  $\varphi$, the barcode entropy $\hbar(\varphi;L,L')$ is lower
  semicontinuous in the pair $(L, L')$ with respect to the Hofer 
  distance or, whenever $L$ is wide (i.e., $\HF(L)=\H_*(L)$), the
  spectral distance. In particular, $\hbar(\varphi;L)$ is lower
  semicontinuous in $L$.

\end{itemize}
\end{Proposition}

\begin{proof}
  Assertion \ref{BE1} is a direct consequence of the definition. To
  prove \ref{BE2} recall that the Floer complex of the pair
  $(L_0\times L_1,L'_0\times L'_1)$ is the tensor product of the Floer
  complexes of $(L_0, L'_0)$ and $(L_1,L'_1)$ over the Novikov field
  $\Lambda$; see, e.g., \cite[Sec.\ 2.6]{HLS} and \cite{Li}. Then a
  singular value decomposition for the product is obtained by taking
  the ``product'' of singular value decompositions of the factors in a
  self-evident way. As a consequence, every pair of bars
  $\beta_0\in \CB(L_0,L'_0)$ and $\beta_1\in \CB(L_1,L'_1)$ gives rise
  to two bars of length $\min\{\beta_0,\beta_1\}$ in
  $\CB(L_0\times L_1,L'_0\times L'_1)$ when both bars are finite. If
  one or both bars in a pair are infinite, the pair gives rise to one
  bar of length $\min\{\beta_0,\beta_1\}$.  Therefore,
  $$
  \fb_\eps(L_0\times L_1,L'_0\times L'_1)\leq 2 \fb_\eps(L_0,L'_0)\cdot
    \fb_\eps(L_1,L'_1),
    $$
    which proves \ref{BE2}.

    Assertion \ref{BE3} follows from \eqref{eq:inv} and the Poincar\'e
    duality, \eqref{eq:PD}. The first identity in \ref{BE4} also
    follows from \eqref{eq:inv}. The second identity is a consequence
    of the fact that when $\varphi$ is non-degenerate, $\varphi$ and
    $\psi\varphi\psi^{-1}$ have isomorphic Floer complexes and hence
    the same barcode. By continuity,
    $\CB(\varphi)=\CB(\psi\varphi\psi^{-1})$ even when $\varphi$ is
    degenerate.  The last one follows from the first identity by
    setting $\psi=\varphi^k$ and using the fact that $\varphi$
    commutes with $\varphi^k$.

    To prove \ref{BE5}, focusing on the Hofer distance it suffices to show that
  \begin{equation}
    \label{eq:LSC-pair}
  \hbar(\varphi; \tL,\tL')\geq \hbar_{4\delta}(\varphi; L,L'),
\end{equation}
whenever $d_{\hn}(L,\tilde{L})<\delta$ and
$d_{\hn}(L',\tilde{L}')<\delta$. Thus assume that $\tL'=\psi(L')$,
where $\|\psi\|_{\hn}<\delta$. Then
$d_{\hn}\big(\varphi^k,\varphi^k\psi\big)<\delta$, and setting
$\tL^k=\varphi^k(\tL')$ and $L^k=\varphi^k(L')$ as above, we have
$d_{\hn}\big(L^k,\tL^k\big)<\delta$. Therefore, by
\eqref{eq:insensitive1},
  $$
  \hbar_\eps(\varphi; \tL,\tL')\geq \hbar_{\eps+4\delta}(\varphi;
  L,L').
  $$
  Passing to the limit as $\eps\to 0+$, we obtain
  \eqref{eq:LSC-pair}. For the spectral distance the argument is
  similar but now we need to use the continuity of the barcode in the
  spectral norm established in \cite{KS}.
\end{proof}

\begin{Remark}
  Recall from Proposition \ref{prop:htop-properties} that for
  topological entropy \ref{BE1} and \ref{BE2} turn into equality. We
  do not know if these inequalities can actually be
  strict. Furthermore, by Theorem \ref{thm:C},
  $\hbr(\varphi^k)=k\hbr(\varphi)$ when $M$ is a surface. Also recall
  that, as was shown in \cite{AM21}, $\htop(\varphi)$ is lower
  semicontinuous in $\varphi$ with respect to the Hofer metric when
  $\dim M=2$. Hence, $\hbar(\varphi)$ is also Hofer lower
  semicontinuous for Hamiltonian diffeomorphisms of surfaces by
  Theorem \ref{thm:C}. This observation leads to the
  question/conjecture if/that this is also true in all dimensions.  We
  also note that $\hbr(\varphi;L,L')$, with $\varphi$ fixed, is quite
  sensitive to deformations of $L$ and $L'$ even by a Hamiltonian
  isotopy; cf.\ \cite[Ex.\ 2.9]{CGG:Entropy}.
\end{Remark}

We expect an analogue of Proposition \ref{prop:hbar-prop} to hold for
Reeb flows with a similar proof and the Hofer distance replaced by a
suitable version of the symplectic Banach--Mazur distance defined in
\cite{PRSZ,SZ-SBM,Us-SBM}; see Problem \ref{prob:basic}.




\section{Proof of Theorems \ref{thm:A} and \ref{thm:A'}}
\label{sec:pf-A2}

As we have already mentioned, the Hamiltonian case of Theorem
\ref{thm:A} is an immediate consequence, via the graph construction,
of its counterpart in Theorem \ref{thm:A'}. In this section, we prove
the latter result, closely following \cite{CGG:Entropy}.  In the Reeb
case, the proof of both theorems (see \cite{CGGM:Reeb,FLS,Fe1,GGM}
uses a similar method and differs only in technicalities, which we
briefly elaborate on at the end of Section \ref{sec:strat}.

The idea of the proof is that $\fb_\eps(L,L')$ gives rise to a lower
bound on the number of intersections $N(\xi)$ not only of $L$ and $L'$
but also of $L_\xi$ and $L'$ when $L_\xi$ is a family of small
Hamiltonian perturbations of $L$. This is a consequence of
\eqref{eq:insensitive1}. Then a variant of Crofton's inequality from
integral geometry gives a lower bound on the volume $\vol(L')$ of $L'$
via the integral of this intersection function $N(\xi)$ over the
parametrizing set $B\ni\xi$. Replacing $L'$ by $L^k:=\varphi^k(L')$,
we obtain a lower bound on the growth of $\vol\big(L^k\big)$ via the
growth of $\fb_\eps(L,L^k)$ and the theorem follows.

  We break the proof down into two subsections. The first of them
  -- Section \ref{sec:tomographs} -- focuses on the machinery of
  Lagrangian tomographs; the actual proof is given in Section
  \ref{sec:strat}. We conclude this section with an application of
  Lagrangian tomographs, based on \cite{CGG:Vol}, to the lower
  semi-continuity of Lagrangian volume; see Section
  \ref{sec:Lagr-vol}.

\subsection{Lagrangian tomographs}
\label{sec:tomographs}
The notion of a Lagrangian tomograph and a variant of Crofton's
inequality, originating in integral geometry, are the key tools used
in the proof of the theorem.  The framework described in this section
is essentially contained in \cite{Ar1, Ar2}, in a setting very close
to ours, and we include the proofs only for the sake of completeness.
(See also \cite{BL, CGG:Vol, Se} for other applications in the context
of symplectic dynamics.) Furthermore, this framework is similar to the
double fibrations utilized in integral geometry for generalizing and
proving Crofton's formula; see \cite{APF98, APF07, GS}.

\subsubsection{Crofton's inequality}
\label{sec:Crofton}

Let $L$ be a closed manifold and $B$ a compact manifold possibly with
boundary. We denote a point in $B$ by $\xi$. Also, let $d\xi$ be a
smooth measure on $B$. In the situation we are mainly interested in,
$B$ is the closed $d$-dimensional ball $B^d$ and $d\xi$ is the
Lebesgue measure, and $L$ is the Lagrangian submanifold.  Sometimes it
is convenient to require $d\xi$ to be supported in the interior of
$B$. Denote by $\pi\colon E=B\times L\to B$ the projection to the
first factor. More generally, we can take $\pi\colon E\to B$ be a
fiber bundle over some compact base $B$ with fiber $L$, and $0\in B$ a
particular point. Furthermore, let
$$
\Psi\colon E \to M
$$ 
be a submersion onto its image where $M$ is a manifold with a fixed
auxiliary Riemannian metric. This manifold need not be compact but,
of course, $\Psi(E)\subset M$ is, since $E$ is compact. We require
$\Psi_\xi:=\Psi|_{\xi\times L}$ to be an embedding for all $\xi$, and hence
$L_\xi:=\Psi_\xi(L)$ is a smooth closed submanifold of~$M$.

In the spirit of integral geometry we will refer to $\Psi$ as a
\emph{tomograph} and call $L_0$, which we will identify with $L$, the
\emph{core} of the tomograph. We say that $\Psi$ is a \emph{Lagrangian
  tomograph} if all submanifolds $L_\xi$ are Lagrangian and
Hamiltonian isotopic to each other.

Finally, let $L'$ be a closed submanifold of $M$ with
$$
\codim L'=\dim L.
$$
Since $\Psi$ is a submersion, $\Psi_\xi\pitchfork L'$ for almost all
$s\in B$. Hence, 
$$
N(\xi):=|L_\xi\cap L'|
$$
is a locally constant function on the complement to a zero measure
closed subset of $B$. As a consequence, $N$ is an integrable function
on $B$.

\begin{Lemma}[Crofton's inequality]
  \label{lemma:Crofton}
  We have
  \begin{equation}
    \label{eq:Crofton}
  \int_B N(\xi)\,d\xi\leq\const \cdot \vol(L'),
\end{equation}
where the constant depends only on $d\xi$,
$\Psi$ and the metric on $M$, but not on $L'$.
\end{Lemma}

Lemma \ref{lemma:Crofton} is proved in \cite{Ar1,Ar2} and can also be
thought of as a direct application of the coarea formula (see, e.g.,
\cite[Sec.\ 13.4.3]{BZ}) or of the variant of Crofton's formula from
\cite{APF98}, which we will briefly touch upon in Section
\ref{sec:Crofton2}.  However, a direct proof is simple and short, and
we include it below for the sake of completeness. This general
framework and Lemma \ref{lemma:Crofton} are also very much in the
spirit of the Gelfand transform in integral geometry and various
versions of Crofton's formula; see, e.g., \cite{APF98,APF07} and
references therein. The key difference is that here and in
\cite{Ar1,Ar2} $\Psi$ is required to be only a submersion, not a
fibration. Furthermore, we note that in this generality, in contrast
with the standard Crofton formula, one cannot expect an inequality
going in the direction opposite of \eqref{eq:Crofton}. Indeed, the
graph of a smooth map or even of a diffeomorphism between two closed
manifolds can have arbitrarily large volume but it intersects every
vertical slice at only one point.

\begin{proof}[Proof of Lemma \ref{lemma:Crofton}]
  Set $\Sigma=\Psi^{-1}(L')$. This is a smooth submanifold of $E$ and
  $$
  \codim \Sigma = \codim L' = \dim L, \textrm{ i.e., } \dim\Sigma =
  \dim B.
  $$
  By construction,
  \begin{equation}
    \label{eq:L-Sigma}
  |(\xi\times L)\cap \Sigma|=|L_\xi\cap L'|=N(\xi).
  \end{equation}

  In the proof it will be convenient to equip $E$ with two different
  auxiliary metrics: the first metric is adapted to $\pi$ and the second
  metric to $\Psi$.

  We begin by fixing some metrics on $B$ and $L$, and assuming first
  that $E=B\times L$ carries the product metric and $d\xi$ is the
  Riemannian volume form or, to be more precise, the volume
  density. (It would be sufficient to require $D\pi$ to be an isometry
  on the normals to the fibers.)  Then, by \eqref{eq:L-Sigma},
  $$
  \int_B N(\xi)\,d\xi=\int_\Sigma\pi^*d\xi\leq \vol(\Sigma).
  $$
  Here, $\pi^*d\xi$ is the pull-back measure or the pull-back density,
  but not the pull-back differential form. The last inequality is a
  consequence of the fact that $D\pi_x\colon T_xE\to T_{\pi(x)}B$,
  $x\in E$, is an orthogonal projection along the fiber, and hence,
  when restricted to $T_x\Sigma$, it can only decrease the
  $\dim B$-dimensional volume. As a consequence, for an arbitrary
  metric on $E$ and an arbitrary smooth measure $d\xi$ on $B$, we have
  \begin{equation}
    \label{eq:N-Sigma}
  \int_B N(\xi)\,d\xi\leq \const\cdot\vol(\Sigma).
\end{equation}

Next, let us equip $E$ with a metric such that the restriction of
$D\Psi$ to the normals to the fibers of $\Psi$ (i.e., the inverse
images $\Psi^{-1}(y)$, $y\in M$) is an isometry. Then, by Fubini's
theorem or, more specifically, the coarea formula (see \cite[Sect.\
13.4.3]{BZ}), we have
\begin{equation}
  \label{eq:const}
  \vol(\Sigma) =\int_{L'}\vol\big(\Psi^{-1}(y)\big)\,dy|_{L'}
  \leq\max_{y\in\Psi(E)}\vol\big(\Psi^{-1}(y)\big)\cdot
  \vol(L'),
\end{equation}
where in the first equality $dy|_{L'}$ stands for the induced volume
form on $L'$. Thus
  \begin{equation}
    \label{eq:Sigma-L}
    \vol(\Sigma)\leq\const\cdot \vol(L').
  \end{equation}
  For an arbitrary metric on $E$, \eqref{eq:Sigma-L} still holds,
  albeit with a different constant. Combining \eqref{eq:N-Sigma} and
  \eqref{eq:Sigma-L}, we obtain \eqref{eq:Crofton}.
\end{proof}

\subsubsection{Existence of Lagrangian tomographs}
\label{sec:ELT}

We establish in this section the existence of Lagrangian
tomographs. Let $L=L_0$ be a closed Lagrangian submanifold of $M$.

\begin{Lemma}
  A Lagrangian tomograph with core $L$ and $\dim B=d$ exists if and
  only if $L$ admits an immersion into $\R^d$.
\end{Lemma}

\begin{proof}
  By the Weinstein tubular neighborhood theorem it is sufficient to
  prove the lemma when $L$ is the zero section in $M=T^*L$.
  
Let $\iota\colon L \to \R^d$ be an immersion, and 
$$
f_\xi=\xi_1 g_1+\ldots + \xi_d g_d,
$$
where $\xi=(\xi_1,\ldots,\xi_d)\in \R^d$ are the coordinate functions
on $\R^d$ and $g_i:=\xi_i\circ \iota$ are the restrictions of the
coordinate functions to $L$. Let us now require $\xi$ to be in a ball
$B\subset \R^d$ centered at the origin. Then, setting
$\Psi_\xi(x)=df_\xi(x)$, $x\in L$, we obtain a map
$\Psi\colon B\times L\to T^*L$. It is easy to see that the condition
that $\iota$ is an immersion is equivalent to that $\Psi$ is a
tomograph as desired.

Indeed, the immersion condition is equivalent to that
$dg_1,\ldots, dg_d$ generate $T_x^*L$ at every point $x\in L$.  In
other words, the map
  $$
  D\Psi_{(0,x)}\colon T_0B\oplus T_xL\to T_{(0,x)}T^*L
  = T^*_xL\oplus T_xL
  $$
  is onto. Therefore, $\Psi$ is a submersion when $B$ is sufficiently
  small if and only if $(g_1,\ldots,g_d)\colon L\to \R^d$ is an
  immersion. (This observation is already contained in, e.g.,
  \cite{GuSt}.)  It is clear from the construction that the
  submanifolds $L_\xi$ are embedded.

  Conversely, assume that $\Psi\colon B\times L\to T^*L$ is a
  tomograph with $\Psi_0=\id$. Then the linearization
  $A=\p \Psi_\xi/\p \xi$ is a linear map from $T_0B$ to the space of exact
  sections of $T^*L$. Pick some functions $\{g_\ell\}$ so that
  $dg_\ell=A(e_\ell)$ where $e_1,\ldots, e_d$ is a basis in
  $T_0B$. Then $\iota=(g_1,\ldots,g_d)$ is an immersion $L\to\R^d$;
  cf.\ \cite{GuSt}. This completes the proof of the lemma.
  \end{proof}

\subsection{From barcodes to volume growth: Proof of
  Theorem \ref{thm:A'}}
\label{sec:strat}
With the machinery of Lagrangian tomographs set up we are now in a
position to prove the Hamiltonian case of Theorem \ref{thm:A'}. As we
have already mentioned, we only need to prove the first inequality in
\eqref{eq:A'}; for the second one follows from Yomdin's theorem.

The key to the proof is the following result. As in
Section \ref{sec:tomographs}, fix an auxiliary
Riemannian metric on $M$ and denote by $\vol(L')$ the Riemannian
volume of a submanifold $L'$.

\begin{Proposition}
  \label{prop:fb-vol}
For every $\eps>0$ and closed monotone Lagrangian submanifold $L$,
there exists a constant $C$ such that
\begin{equation}
  \label{eq:fb-vol}
  \fb_\eps(L,L')\leq C\cdot \vol(L')
\end{equation}
for every Lagrangian submanifold $L'$ Hamiltonian isotopic to $L$.
\end{Proposition}

\begin{proof}
Pick a Lagrangian tomograph with core $L_0:=L$.
Thus we have a family of Lagrangian submanifolds $L_\xi$ smoothly
parametrized by the closed $d$-dimensional ball $B^d=B^d(r)$ of radius
$r>0$ (for some large $d$). By shrinking $B$ if necessary, we can
ensure that these submanifolds have the following properties:
\begin{itemize}
\item[\reflb{ELT1}{(i)}] The Lagrangian submanifolds $L_\xi$ are
  Hamiltonian isotopic to $L_0=L$ and the Hofer distance between $L_0$
  and $L_\xi$ is small:
  \begin{equation}
    \label{eq:H-close}
  d_{\hn}(L_0,L_\xi)<\eps/2
\end{equation}
(In fact, $L_\xi$ can even be taken $C^\infty$-close to
$L_0$. Furthermore, here we can replace the Hofer distance by the
spectral distance when the latter is defined; \cite{KS}.)

\item[\reflb{ELT2}{(ii)}] The Lagrangian submanifold $L_\xi$ is
  transverse to $L'$ for almost all $\xi\in B^d$. Let
  $$
  N(\xi):=|L_\xi\cap L'|
  $$
  be the number of intersections of $L_\xi$ and $L'$. Then $N(\xi)$ is
  a measurable function, finite for almost all $\xi$. Furthermore, by
  Crofton's inequality (Lemma \ref{lemma:Crofton}), we have
\begin{equation}
  \label{eq:vol-int}
  \int_{B^d} N(\xi)\, d\xi\leq \const\cdot \vol (L'),
\end{equation}
where $d\xi$ is the standard Lebesgue measure on $B^d$.
\end{itemize}
Note that all conditions in \ref{ELT1} and \ref{ELT2}, other than
\eqref{eq:H-close}, are satisfied automatically; see Section
\ref{sec:Crofton}. As we have already pointed out, to guarantee
\eqref{eq:H-close} we can simply shrink $B$.

Then, whenever $L_\xi\pitchfork L'$, we have
$$
  N(\xi) \geq \fb_{\eps/2}(L_\xi,L')
         \geq \fb_{\eps}(L_0,L').
$$
This is an immediate consequence of \eqref{eq:intersections} with
Lagrangian submanifolds suitably relabeled, $\eps$ replaced by $\eps/2$
and $\delta=\eps/2$.  (Here we could have required $L_\xi$ to be
sufficiently $C^\infty$-close to $L_0$ avoiding the notion of the
Hofer or spectral distance or used the spectral distance when it is
defined.) Combining this inequality with \eqref{eq:vol-int}, we obtain
\eqref{eq:fb-vol}, which finishes the proof of the proposition and
Theorem \ref{thm:A'} in the Hamiltonian case.
\end{proof}

  There are two ways to adapt this argument to a Reeb flow $\varphi^t$
  on $M=\p W$. For the sake of simplicity, we consider here only the
  absolute variant, i.e., \eqref{eq:A}, referring to \cite{Fe1} for
  the proof of the Reeb analogue of \eqref{eq:A'}.

The first approach, based on a modification of Crofton's inequality,
is roughly speaking as follows. Consider a tomograph $L_\xi$, $\xi\in B$,
in $M\times M$ whose core is the diagonal $L_0$ and let $N_s(\xi)$ be
the number of intersections of the graph $\Gamma_t$ of $\varphi^t$ for all $t\in
(0,\,s)$ with $L_\xi$. This is again an integrable function of $\xi$
and $s$, finite for almost all $\xi$ and $s$, and 
$$
  N_s(\xi) \geq \fb_{\eps}\big(s),
$$
when the size of $B$ is small enough.  Let $v(t)$ be the volume of
$\Gamma_t$. Then Lemma \ref{lemma:Crofton} takes the form
$$
\int_B N_s(\xi)\,d\xi\leq \const \int_0^s v(t)\, dt.
$$
Exponential growth of $\fb_\eps(s)$ forces exponential growth of
$v(s_k)$ for some sequence $s_k\to\infty$ and the proof ends in
essentially the same way as in the Hamiltonian case. This approach is
worked out in detail in \cite{GGM} for geodesic flows.

The second approach is based on a modification of the definition of
barcode entropy. Namely, let $H$ be a semi-admissible Hamiltonian with
slope $a\not\in \CS(\alpha)$. Then the filtered Floer homology
$\HF(H)$ is a genuine persistence module and we have the number
$\fb_\eps(H)$ of bars of length greater than $\eps$ well-defined; see
Section \ref{sec:SH-def}. Next, let us replace $H$ by $s H$ where we
require that $sa\not\in \CS(\alpha)$. Now we have the function
$\fb_\eps(sH)$ and it is not hard to show that
\begin{equation}
  \label{eq:hbar-dyn}
\hbar(\alpha)=\lim_{\eps\to 0+}\limsup_{s\to\infty}\frac{\log^+
  \fb_\eps(sH)}{s };
\end{equation}
see \cite[Thm.\ 4.4]{CGGM:Reeb}. Moreover, here we can take the upper
limit over any sequence $s_k\to\infty$ such that $s_{k+1}/s_k\to
1$. In particular, we can set $s_k=k$ whenever $ka\not\in \CS(\alpha)$
for all $k$. The proof of Theorem \ref{thm:A} in the Hamiltonian case
carries over to semi-admissible Hamiltonians $H$. We use
the tomograph construction for the graph of $\varphi_{H}^k|_U$
where $U$ is a suitable chosen neighborhood of $M$ in
$\WW$. Rescaling to account for the effect of the slope $a$, we
arrive at \eqref{eq:A}. The details of this argument, which is
somewhat different from the proof in \cite{FLS}, can be found in
\cite[Sec.\ 4.2]{CGGM:Reeb}.

\begin{Remark}
The machinery of Lagrangian tomographs has also been used in \cite{BL}
to establish an inequality between barcode entropy and categorical
entropy from \cite{DHKK}. In the next section, we will discuss another
application of the method.
\end{Remark}

\subsection{Digression: Lagrangian volume}
\label{sec:Lagr-vol}
A remarkable feature of surface area is that it is $C^0$ lower
semi-continuous. This is by no means obvious. Surface area is usually
defined as a certain integral over a parametrized surface, involving
the derivatives of the parametrization, and there is absolutely no
reason to expect to have controlled behavior under $C^0$-small
perturbations. The simplest way to prove this fact is perhaps by using
a suitable variant of Crofton's formula.  The result, of course, has
numerous generalizations in differential and integral geometry; see
\cite{BuIv,Ce,Fe52,Iv}.  With this in mind, it should come as no
surprise that this lower semi-continuity has an analogue in symplectic
geometry.

\subsubsection{Lower semi-continuity of Lagrangian volume}
\label{sec:LSC}
Let $M$ be a symplectic manifold as in Section \ref{sec:conv}. Thus
$M$ is monotone and either closed or sufficiently nice at infinity
(e.g., convex) to ensure that the relevant filtered Floer homology is
defined. Furthermore, let $\CL$ be a class of closed monotone
Lagrangian submanifolds $L$ of $M$, Hamiltonian isotopic to each
other. We require in addition that the minimal Chern number of $L$ is
at least 2. Let, as in Section \ref{sec:conv}, $d_\gamma$ be the
ambient spectral distance on $\CL$ defined by \eqref{eq:d-gamma} or
the intrinsic spectral norm from \cite{KS} when it is defined.

Finally, fix a Riemannian metric compatible with $\omega$. Then we
have the volume or, to be more precise, the surface area function:
$$
\vol\colon \CL\to (0,\infty)
$$
sending $L$ to its surface area, which we refer to as the Lagrangian
volume. We conjecture that $\vol$ is lower semi-continuous on $\CL$ with
respect to the $\gamma$-distance; \cite{CGG:Vol}. This conjecture has
been proved in
two disparate cases. The first of these is where $M$ is K\"ahler and
has a large symmetry group (e.g., $M=\C^n$ or $\CP^n$).

\begin{Theorem}[Thm.\ 1.2, \cite{CGG:Vol}]
  \label{thm:symmetry}
  Let $M$ be K\"ahler and the Riemannian metric be the real part of
  the K\"ahler form. Assume furthermore that the group of Hamiltonian
  K\"ahler isometries acts transitively on the Lagrangian Grassmannian
  bundle over $M$. Then $\vol$ is $d_\gamma$-lower semi-continuous on $\CL$.
\end{Theorem} 

When prove this theorem in Section \ref{sec:Crofton2}.  When $L=\T^n$,
the restrictive condition that $M$ has a large symmetry group can be
dropped. In fact, we have a more precise result asserting roughly
speaking that for a fixed Lagrangian submanifold $L_0$ from $\CL$ and
another Lagrangian submanifold $L\in \CL$, which is $d_\gamma$-close
to $L_0$ (depending on $L_0$), the part of $L$ situated $C^0$-close to
$L_0$ is at least almost as large as $L_0$.

\begin{Theorem}[Thm.\ 1.3, \cite{CGG:Vol}]
  \label{thm:torus}
  Assume that $L_0=\T^n\in\CL$ and let $U$ be an arbitrary open subset
  containing $L_0$. Then the function
$$
\vol_U\colon \CL\to [0,\infty)
$$
sending $L$ to the surface area of $U\cap L$ is $d_\gamma$-lower semi-continuous
on $\CL$ at $L_0$.
\end{Theorem} 

The proof of this result is considerably more involved than the proof
of Theorem \ref{thm:symmetry} and we will only briefly discuss in
Section \ref{sec:Crofton}.

\begin{Corollary}
  \label{cor:torus}
  Assume that $L=\T^n$. Then $\vol$ is $d_\gamma$-lower semi-continuous on $\CL$.
\end{Corollary}

Since $d_\gamma$ is bounded from above by the Hofer distance, in the
setting of this corollary or of Theorems \ref{thm:symmetry} and
\ref{thm:torus}, the $\vol$ function is also lower semi-continuous
with respect to the Hofer distance.

\begin{Example}
  \label{ex:min}
  The function $\vol$ is automatically lower semi-continuous at $L$
  when $L\subset M$ is a local or global volume minimizer in
  $\CL$. For instance, as is easy to see, this is the case for the zero
  section of the cotangent bundle equipped with the Sasaki
  metric. Likewise, the standard $\RP^n\subset \CP^n$ and the Clifford
  torus are volume minimizers with respect to the Fubini--Studi metric
  on $\CP^n$; \cite{Oh:Invent}. The same is true for the product tori in
  $\C^n$ with respect to the standard metric; \cite{Oh:MathZ}. In
  contrast, Theorem \ref{thm:symmetry} and Corollary \ref{cor:torus}
  assert lower semi-continuity at every point of $\CL$ and in the case
  of the corollary for a broad class of metrics on $M$.
\end{Example}

Denote by $\hat{\CL}$ the Humili\`ere completion of $\CL$, i.e., its
completion with respect to $d_\gamma$; cf., \cite{Hum}. The
Corollary \ref{cor:torus} and Theorem \ref{thm:symmetry} are
equivalent to the following result.

\begin{Corollary}
  \label{cor:compl}
  Assume that $L=\T^n$ or that $M$ is as in Theorem
  \ref{thm:symmetry}. Then $\vol$ extends to a lower semi-continuous
  function on $\hat{\CL}$.
\end{Corollary}

  We do not know if in general the function $\vol$ is bounded away
  from zero on $\CL$ or equivalently on $\hat{\CL}$. This is obviously
  so when $\vol$ has a global minimizer as in the setting of Example
  \ref{ex:min}. Furthermore, lower bounds for $\vol(L)$ in terms of
  the displacement energy of $L$ are obtained in \cite{Vi:metric} when
  $M=\R^{2n}$ or $\CP^n$ or a cotangent bundle.

  \begin{Remark} 
    Theorems \ref{thm:symmetry} and \ref{thm:torus} resemble and are
    partially inspired by the deep lower semi-continuity and
    robustness results for topological entropy in low dimensions from
    \cite{ADMM, ADMP, AM:braid}. However, we are not aware of
    any direct formal connections between the two settings.
\end{Remark}

\subsubsection{Crofton's formula revisited}
\label{sec:Crofton2}
In this section we briefly discuss the proofs of these results, which
are also based on a variant of Crofton's formula and Lagrangian
tomographs. To this end, it is convenient to cast the argument in the
framework of densities.

Let $P$ be the Stiefel bundle over a manifold $M^m$, i.e., $P$ is
formed by $k$-frames $\bar{v}=(v_1,\ldots,v_k)$.  Recall that a
$k$-density $\fd$ on $M$ is a function $\fd\colon P\to \R$ such that
\begin{equation}
  \label{eq:density}
\fd(\bar{v}')=|\det A|\fd(\bar{v}),
\end{equation}
where $A$ is the linear transformation of the span of $\bar{v}$
sending $\bar{v}$ to $\bar{v}'$; see, e.g., \cite{APF98}. Sometimes it
is convenient to drop the condition that the vectors from
$\bar{v}=(v_1,\ldots, v_k)$ are linearly independent by setting
$\fd(\bar{v})=0$ otherwise.

Here are several examples of densities: A Riemannian or Finsler metric
on $M$ or, more generally, any homogeneous degree-one function
$TM\to\R$ is a 1-density. For instance, in self-explanatory
notation, the functions $|dx|$, $|dy|$, $|dx|+|dy|$ and
$\sqrt{dx^2+dy^2}$ are 1-densities on $\R^2$.  Furthermore, for every
$k\leq m$ a Riemannian metric gives rise to a $k$-density $\fg_k$
defined by the condition that $\fg_k(\bar{v})$ is the volume of the
parallelepiped spanned by $\bar{v}$. Thus $\fg_m$ is the Riemannian
volume. For a differential $k$-form $\alpha$ its absolute value
$|\alpha|$ is a $k$-density. The sum of two $k$-densities is again a
$k$-density.

A $k$-density $\fd$ can be integrated over a compact $k$-dimensional
submanifold $L$ without requiring $L$ to be oriented or even
orientable.  Similarly to differential forms, densities can be pulled
back and, under suitable additional conditions, pushed forward. When
it is defined, the push-forward $\Psi_*\fd$ of $\fd$ by a map
$\Psi\colon M'\to M$ is characterized by the condition that
$$
\int_L \Psi_*\fd=\int_{\Psi^{-1}(L)}\fd.
$$

Next let $\Psi\colon E\to M$ be a tomograph in the sense of Section
\ref{sec:Crofton}. Thus, in particular, a Hamiltonian diffeomorphism
we have a fiber bundle $\pi\colon E\to B$ with fiber $L$ over a
compact base $B$, and a smooth measure $d\xi$ on $B$ which we will
treat as a density. It is convenient to require $d\xi$ to be supported
in the interior of $B$ when $B$ is a manifold with boundary. We also
have embedded submanifolds $L_\xi:=\Psi\big(\pi^{-1}(\xi)\big)$,
$\xi\in B$, in $M$. Let $d=\dim B$ be the dimension of the tomograph
and $k=\dim L$.  The pull-back/push-forward density
\begin{equation}
  \label{eq:fd}
\fd_{\Psi}:=\Psi_*\pi^*\, d\xi
\end{equation}
is a smooth $k$-density. We call $\supp \fd_{\Psi}\subset \Psi(E)$ the
support of the tomograph $\CT$.

Let $L'$ be a closed submanifold of
$M$ such that $\codim L'=\dim L$. Set
$$
N(\xi):=|L_\xi\cap L'|\in [0,\,\infty].
$$
Since $\Psi$ is a submersion, $\Psi_\xi\pitchfork L'$ for almost all
$\xi\in B$. Hence $N(\xi)<\infty$ almost everywhere and $N$ is an
integrable function on $B$. We refer the reader to, e.g., \cite{APF98}
for a proof of the following simple but important result, which 
generalizes Lemma \ref{lemma:Crofton} and is proved in a
similar way.
\begin{Proposition}[Crofton's formula; \cite{APF98}]
  \label{prop:Crofton2}
  We have
  \begin{equation}
\label{eq:Crofton2}
  \int_B N(\xi)\,ds=\int_{L'}\fd_{\Psi}.
\end{equation}
\end{Proposition}
It is easy to see that Crofton's inequality, \eqref{eq:Crofton}, from
Lemma \ref{lemma:Crofton} is indeed a consequence of
\eqref{eq:Crofton2}.

Returning to the setting of Section \ref{sec:LSC}, assume that $\Psi$ is a
Lagrangian tomograph with core $L_0\in\CL$.

\begin{Example}[Classical Lagrangian tomographs]
  \label{exam:symmetry}
  Assume that $M$ is K\"ahler and that the group $G$ of Hamiltonian
  K\"ahler isometries acts transitively on $M$. Let $L_0$ be any
  closed Lagrangian submanifold of $M$. Set $E=L_0\times G$ with $\pi$
  being the projection to the second factor $B=G$ and
  $\Psi(x,\xi):=\xi(x)$, where $x\in L_0$ and $\xi\in G$. Finally, we
  let $d\xi$ be a Haar measure on $G$. Then, as is easy to see, we
  obtain a Lagrangian tomograph with support $M$.  This is essentially
  the classical setting of Crofton's formula -- see the reference
  cited above. (Usually one replaces the base $B=G$ by the space
  $G/\Stab(L_0)$ formed by all images of $L_0$ in $M$ under $G$. Here,
  however, we prefer to keep $B=G$.) This construction applies to
  $M=\C^n$ and $\CP^n$ and, more generally, to any simply connected
  homogeneous K\"ahler manifold, although $G$ is not compact when $M$
  is not compact.
\end{Example}

\begin{Theorem}
  \label{thm:density}
  The function
  $$
I_\Psi\colon  L\mapsto \int_L\fd_{\Psi}
  $$
  is $d_\gamma$-lower semi-continuous on $\CL$.
\end{Theorem}

The proof of this theorem, left to the reader as an exercise, is
similar to the proof of Theorem \ref{thm:A'} in Section
\ref{sec:strat} and also based on \eqref{eq:intersections}; see
\cite{CGG:Vol}.

\begin{proof}[Proof of Theorem \ref{thm:symmetry}]
  In the setting of the theorem, consider the Lagrangian tomograph
  $\CT$ from Example \ref{exam:symmetry}. Both the
  push-forward/pull-back density $\fd_\CT$ and the metric $n$-density
  $\fg$ are invariant under the group $G$ of (Hamiltonian) K\"ahler
  isometries. Since $G$ acts transitively on the Lagrangian
  Grassmannian bundle, these two densities agree up to a factor on the
  frames spanning Lagrangian subspaces;
  cf. \cite[Lem. 5.4]{APF07}. Hence, the function
$$
L\mapsto \int_L\fg
$$
is also $d_\gamma$-lower semi-continuous on $\CL$ by Theorem
\ref{thm:density}.
\end{proof}

\begin{proof}[On the proof of Theorem \ref{thm:torus}]
  In general, there is no hope to exactly match a Lagrangian metric
  density $\fg$ by the push-forward/pull-back density $\fd_\Psi$ of a
  localized tomograph as in the proof of Theorem
  \ref{thm:symmetry}. However, to prove the theorem it, is sufficient
  to tightly bound $\fg$ from below.

Let $L=\T^n\subset M^{2n}$ be a Lagrangian torus. Fix a compatible
metric on $M$ and denote by $\fg$ the metric density. We consider
Lagrangian tomographs $\Psi$ with fiber diffeomorphic to $L$ and a
ball $B^d$ serving as the base $B$. Thus $E=L\times B$ and $\pi$ is
the projection to the second factor. Fix an open set $U\supset L$ and
$\eta>0$. Then there exists a
  Lagrangian tomograph $\Psi$ as above with $L=L_0=\Psi(\pi^{-1}(0))$,
  supported in $U$, and such that
  \begin{equation}
    \label{eq:LT1}
  \fd_\Psi|_L=\fg|_L\textrm{ pointwise, }
\end{equation}
and
   \begin{equation}
    \label{eq:LT2}
  \fd_\Psi\leq (1+\eta)\fg.
\end{equation}
The theorem easily follows from these two conditions on $\Psi$
and Theorem \ref{thm:density}; see
\cite[Sec.\ 3]{CGG:Vol}. However, the existence of a Lagrangian
tomograph meeting these requirements is non-trivial, and this is where
the condition that $L$ is a torus is used. The reader can find a
detailed argument in \cite[Sec.\ 4]{CGG:Vol}.
\end{proof}



\section{Crossing energy and the proof of Theorem \ref{thm:B}}
\label{sec:pf-B}
In this section, we discuss the proof of Theorem \ref{thm:B} focusing
again on the absolute variant of the Hamiltonian case following
\cite{CGG:Entropy}. We break the proof down into a few
steps, which are of independent interest.  The main idea is
that all Floer cylinders asymptotic to a periodic orbit of $\varphi$
in $K$ at at least one end have energy bounded from below by some
constant $\eps_K>0$ independent of the cylinder and the period. This
result (Theorem \ref{prop:energy}), established in Section
\ref{sec:energy-dynamics}, is an easy consequence of the crossing
energy theorem -- Theorem \ref{thm:cross-energy} below -- which asserts a
similar fact for just locally maximal sets $K$ and Floer cylinders
``crossing'' an isolating neighborhood of $K$. Then, in the hyperbolic
case, the entropy of $K$ is captured by the exponential growth rate of
$k$-periodic orbits in $K$ by Theorem \ref{thm:p-htop}. Now, Theorem
\ref{thm:B} readily follows by purely algebraic means such as Proposition
\ref{prop:isolated} giving a lower bound on $\fb_\eps\big(\varphi^k\big)$ in
terms of the size of an $\eps$-isolated set of generators. 

In contrast with Theorem \ref{thm:A}, the transition to the Reeb case
is quite non-trivial due to difficulties arising in the proof of
crossing energy bounds, and we will touch upon it in Section
\ref{sec:energy-Reeb}. We refer the reader to \cite{Fe2,Me} for the
proofs of the relative Hamiltonian and Reeb variants of the theorem.
A direct proof of Theorem \ref{thm:B} and the crossing energy theorem
for geodesic flows, relying on finite-dimensional approximations in
Morse theory, is given in \cite{GGM}.

In Section \ref{sec:sectral+inv}, we discuss an application of the
crossing energy theorem to the properties of the spectral norm of the
iterates.

\subsection{Crossing energy and dynamics}
\label{sec:energy-dynamics}

The iterate Hamiltonian diffeomorphism $\varphi^k$ is the time-$k$ map
in the Hamiltonian isotopy $\varphi_H^t$ generated by $H$. In what
follows, when working with the Floer equation for this iterate, it is
convenient to denote the Hamiltonian by $H^{\sharp k}$ and refer to
it, somewhat abusing terminology, as the \emph{iterated Hamiltonian}.
We emphasize that $H^{\sharp k}$ is
the same Hamiltonian as $H$, but now viewed as $k$-periodic in
time. Likewise, the solutions of the Floer equation, \eqref{eq:Floer},
are allowed to be $k$-periodic in time rather than $1$-periodic or,
more generally, defined on a closed domain
$\Sigma\subset \R\times S^1_k$, where $S^1_k=\R/k\Z$, rather than a
domain in $\R\times S^1$. There are, of course, other Hamiltonians,
with easily adjustable period, generating $\varphi^k$ and giving rise
to the same filtered Floer complex, but this is a natural and
convenient choice from the dynamics perspective. Moreover, this choice
becomes essential for the proof of Theorem \ref{thm:cross-energy}.

Thus, consider solutions $u\colon \Sigma\to M$ of the Floer
equation
\begin{equation}
  \label{eq:Floer}
 J\p_s u=\p_t u-X_{H}
\end{equation}
for the iterated Hamiltonian $H^{\nat k}\colon S^1_k\times M\to \R$
with $S^1_k=\R/k\Z$, where $\Sigma\subset \R\times S^1_k$ is a closed
domain, i.e., a closed subset with non-empty interior and $J$ is a
background $k$-periodic in time almost-complex structure. This
equation is equivalent to \eqref{eq:floer_0} and \eqref{eq:floer_1}.
Recall from \eqref{eq:energy} that the energy of $u$ is
$$
 E(u)=\int_\Sigma \|\p_s u\|^2 \, ds dt.
$$
where $\|\cdot\|$ stands for the norm with respect to
$\left<\cdot\, ,\cdot\right>=\omega(\cdot,J\cdot)$, and hence
$\|\cdot\|$ depends on $J$.  When $\Sigma=\R\times S^1_k$
and $u$ is asymptotic to $k$-periodic orbits $x$ at $-\infty$ and $y$
at $\infty$, we have
$$
E(u)=\CA(x)-\CA(y),
$$
by \eqref{eq:Energy-Action}, i.e., $E(u)$ is the action difference
between $x$ and $y$. Here we treat $x$ and $y$ as capped $k$-periodic
orbits of $H$ with the capping of $y$ obtained by ``attaching'' $u$ to
the capping of $x$.

Throughout the proof, it will be convenient to work with the extended
phase space $\tilde{M}=S^1\times M$ with $S^1=\R/\Z$. The
time-dependent flow $\varphi^t$ lifts as the genuine flow
$\tilde{\varphi}^t$ on $\tilde{M}$ given by
$$
 \tilde{\varphi}^t(\theta,p)=\big(\theta+t,\varphi^t(p)\big)
$$
generated by the vector field $\p_\theta +X_H$, where in the first
term $t$ is viewed as an element of $S^1=\R/\Z$. Likewise, any map
$z\colon \R\to M$ or $z\colon S^1_k\to M$ lifts to the map
$\tilde{z}(t)=(t, z(t))$, and in a similar vein a solution $u$ of the
Floer equation lifts to a map $\tilde{u}\colon \Sigma\to
\tilde{M}$. If $u$ is asymptotic to $x$ and $y$, the lift $\tilde{u}$
is asymptotic to $\tx$ and $\ty$ in the natural sense. In what
follows, a lift from $M$ to $\tilde{M}$ will always be indicated by
the tilde, and we will identify $M$ with $\{0\}\times M$.

A loop $x\colon S^1_k\to M$ is a $k$-periodic orbit of $\varphi^t$ if
and only if its lift $\tx$ is a $k$-periodic orbit of
$\tilde{\varphi}^t$ and if and only if the sequence
$\hat{x}=\{x_i:=x(i)\mid i\in\Z_k\subset S^1_k\}$ formed by the
intersections of $\tx$ with the cross-section $M$ is a $k$-periodic
orbit of $\varphi$.

Next let us recall the crossing energy theorem (see \cite[Thm.\
6.1]{GG:PR} and also \cite{GG:hyperbolic}), which is crucial to the
proof.  Let $K\subset M$ be a compact invariant set of a Hamiltonian
diffeomorphism $\varphi$ of $M$.  Recall that $K$ is said to be
\emph{locally maximal} or \emph{isolated} (as an invariant set) if there exists a neighborhood $U\supset K$ such that for no initial
condition $p\in U\setminus K$ the orbit through $p$ is contained in
$U$, i.e., there exists $k\in\Z$, possibly depending on $p$, such that
$\varphi^k(p)\not\in U$ or, equivalently,
$$
K=\bigcap_{k\in\Z}\varphi^k(U).
$$
The neighborhood $U$ is called an isolating neighborhood of $K$. Then
any neighborhood of $K$ contained in $U$ is also isolating, and hence
such neighborhoods can be made arbitrarily small. In other words,
whenever $U\supset V\supset K$ and $U$ is an isolating neighborhood
and $V$ is open, $V$ is also an isolating neighborhood. For flows,
locally maximal sets are defined in a similar fashion.

The set $K$ naturally lifts to an invariant set
$\tilde{K}\subset \tilde{M}$ of the flow $\tilde{\varphi}^t$, which is
the union of the integral curves through $K=\{0\}\times K$. (Since $K$
is invariant it suffices to take only $t\in [0,1]$.) The set $K$ is
locally maximal for $\varphi$ if and only if $\tilde{K}$ is locally
maximal for the flow.

Let $u\colon \Sigma\to M$ be a
solution of the Floer equation, \eqref{eq:Floer}, where
$\Sigma\subset \R\times S^1_k$ is a closed domain.  We say that $u$ is
\emph{asymptotic} to $K$ at $\infty$ (or at $-\infty$) if for any
neighborhood $\tilde{U}$ of $\tilde{K}$ there is a half-cylinder
$[s_{\tilde{U}},\,\infty)\times S^1_k$ (or
$(-\infty,\,s_{\tilde{U}}]\times S^1_k$) in $\Sigma$ which is mapped
into $\tilde{U}$ by $\tilde{u}$. For instance, $u$ is asymptotic to
$K$ if $u(s,\cdot)$ uniformly converges as $s\to \infty$ or
$s\to -\infty$ to a $k$-periodic orbit $x$ with $x(0)\in K$ (but not
necessarily with $x(t)\in K$ for all $t\in S^1_k$). In this case,
abusing terminology, we will also say that $u$ is asymptotic to the
$k$-periodic orbit $\hat{x}:=\{x_i:=x(i)\mid i\in\Z_k\}$ of
$\varphi$. We emphasize that here the domain $\Sigma$ of $u$ need not
be a cylinder, although to be asymptotic to $K$ it must contain a
half-cylinder.

Finally, fix a (sufficiently small) isolating neighborhood $\tilde{U}$
of $\tilde{K}$. Set
$\p \tilde{U}:=\mathrm{closure}(\tilde{U})\setminus \tilde{U}$.

\begin{Theorem}[Crossing Energy Theorem, I; Thm.\ 6.1 in \cite{GG:PR}]
  \label{thm:cross-energy}
  Fix a 1-periodic in time almost complex structure $J$ on $M$.  Let
  $J'$ be a $k$-periodic in time almost complex structure on $M$ which
  is sufficiently $C^\infty$-close to $J$, depending on $k$, uniformly
  on $U$. Furthermore, let $u\colon \Sigma\to M$, where
  $\Sigma\subset \R\times S^1_k$, be a solution of the Floer equation
  for $J'$ and $H^{\sharp k}$, asymptotic to $K$ as $s\to\infty$ or
  $s\to-\infty$, and such that
 \begin{itemize}

 \item[\reflb{CEa}{(a)}] either $\p\Sigma\neq \emptyset$ and
   $\tilde{u}(\p\Sigma)\subset \p \tilde{U}$

 \item[\reflb{CEb}{(b)}] or $\Sigma=\R\times S^1_k$ and
   $\tilde{u}(\Sigma)\not\subset \tilde{U}$.
 \end{itemize}
 Then there exists a constant $c_\infty>0$, independent of $k$, $J'$,
 $u$ and $\Sigma$ such that
 \begin{equation}
   \label{eq:cross-energy}
   E(u)>c_\infty .
 \end{equation}

\end{Theorem}

For a fixed $k$, the lower bound, \eqref{eq:cross-energy}, readily
follows from Gromov's compactness and the main point of the theorem is
that $c_\infty$ can be taken independent of $k$. This is where the
condition that $K$ is locally maximal is essential as the following
example shows.

\begin{Example}
  Let $\varphi$ be an irrational rotation of $S^2$ about the $z$-axis
  and let $K$ be the North Pole. (Thus $H$ is a scaled height
  function.)  Then for every $k$ and arbitrary capping of the South
  Pole there exists a Floer cylinder $u_k$ for $H^{\# k}$ connecting
  $K$ with the South Pole. (This follows from the structure of the
  quantum homology of $S^2$.)  For a certain sequence $k_i\to\infty$
  one can find a capping of the South Pole such that
  $E\big(u_{k_i}\big)\to 0$. This example generalizes to
  pseudo-rotations of all complex projective spaces; see
  \cite{GG:gaps,GG:PR}.
\end{Example}  

\begin{Remark}
  The lower bound $c_\infty$ depends on the choice of an isolating
  neighborhood $\tU$ of $\tK$: a smaller neighborhood might
  necessitate a smaller lower bound. The threshold on how close $J'$
  and $J$ need to be for \eqref{eq:cross-energy} to hold depends on
  $k$. Finally, as readily follows from the proof, the lower bound
  $c_\infty$ can also be chosen to be stable with respect to
  $C^\infty$-small perturbations of $H^{\sharp k}$, i.e., so that
  \eqref{eq:cross-energy} holds for solutions of the Floer equation
  for all $k$-periodic Hamiltonians $C^\infty$-close to
  $H^{\sharp k}$.
\end{Remark}

\begin{Example}
  \label{ex:hyperbolic}
  Assume that $K$ comprises just one fixed point $p$ of $\varphi$ and this
  point is hyperbolic. Hence $p$ is isolated as a periodic point of
  $\varphi$ and as an invariant set. Then there exists a constant
  $c_\infty>0$, independent of $k$ and $u$, such that $E(u)>c_\infty$
  for all Floer cylinder $u$ for $H^{\# k}$ asymptotic to $p$, or to
  be more precise to the $k$-th iterate $p^k$, at either end. Moreover,
  $c_\infty$ can also be chosen to be stable with respect to
  $C^\infty$-small perturbations of $H^{\sharp k}$.
\end{Example}  

Here we are mainly interested in Case \ref{CEb} of the theorem. It is
easy to see that this case is a consequence of the more general Case
\ref{CEa} which was originally established in \cite{GG:PR} by using a
variant of the Gromov Compactness Theorem from \cite{Fi}.  However,
Case \ref{CEb} can also be proved directly by an argument which is
considerably simpler than that proof, under the slightly more
restrictive requirement that $J'$ is sufficiently $C^\infty$-close to
$J$ on $M$, but not just on $U$, in the class of $k$-periodic in time
almost complex structures. This is sufficient for our purposes and we
will outline the proof below.

The proof and applications of the theorem hinge on the following
simple but important observation. Recall that a sequence
$$
\hat{z}:=\{z_i\mid i\in \Z_k:=\Z/k\Z\}
$$
is called \emph{$k$-periodic $\eta$-pseudo-orbit} of $\varphi$, or
just a pseudo-orbit for the sake of brevity, if
\begin{equation}
      \label{eq:pseudo-orbit}
      d\big(\varphi(z_i),z_{i+1}\big)<\eta \textrm{ for all $i\in\Z_k$},
\end{equation}
where $d$ the distance on $M$ with respect to an arbitrary auxiliary
Riemannian metric $M$. As in Theorem \ref{thm:cross-energy}, fix an
almost complex structure $J$ and assume that $J'$ is $k$-periodic and
sufficiently close to $J$.

\begin{Proposition}
      \label{prop:pseudo-orbit}
      Let $u\colon \R\times S^1_k\to M$ be a solution of the Floer
      equation, \eqref{eq:Floer}. Assume that $E(u)$ is sufficiently
      small, i.e., $E(u)<e$ with an upper bound $e>0$ depending only
      on $(M, \omega, J)$ and $H$. Then for every $s\in\R$, the
      sequence
    $$
    \hat{z}:=\big\{z_i:=u(s,i)\mid i\in \Z_k\big\}
    $$
    is a $k$-periodic $\eta$-pseudo-orbit of $\varphi$,
    where we can take
    \begin{equation}
      \label{eq:eta-energy}
    \eta= O\big(E(u)^{1/4}\big)
  \end{equation}
  uniformly in $k$ and $J'$ as long as $J'$ is close to $J$.
\end{Proposition}

The key point of this proposition is that every ``circle'' in a Floer
cylinder $u$ for $H^{\# k}$ is an $\eta$-pseudo-orbit with
$\eta= O\big(E(u)^{1/4}\big)$, provided that $E(u)$ is below a certain
threshold $e>0$ which depends only on $(M, \omega, J)$ and $H$, but
not $u$ or $k$.
  
\begin{proof}
  Recall that when $E(u)$ is sufficiently small (with an upper bound
  $e$ depending on $M$ and $H$ but not $u$ and $k$), we have the
  pointwise upper bound
  \begin{equation}
  \label{eq:ptwise-bnd}
    \| \p_s u\|\leq \const\cdot E(u)^{1/4}=O\big(E(u)^{1/4}\big),
  \end{equation}
  where the constant is again independent of $u$ and $k$ and of $J'$
  when $J'$ is close to $J$; see \cite[Sect.\ 1.5]{Sa} or
  \cite[p.\ 542--543]{GG:LS} or, for a different proof,
  \cite{Br}. (Note that it is essential here that the domain of $u$ is
  the entire cylinder $\R\times S^1_k$.) Set $\gamma(t):=u(s,t)$. Then
  we have
  $$
  \big\|\gamma'(t)-X_H\big(\gamma(t)\big)\big\|\leq
  O\big(E(u)^{1/4}\big)
  $$
  for all $t\in S^1_k$ by the Floer equation, \eqref{eq:Floer}. A
  simple calculation along the lines of Gr\"onwall inequality, which
  we leave to the reader as an exercise, shows that
  \begin{equation}
    \label{eq:po-distance}
d\big(\varphi^t\big(\gamma(i)\big),\gamma(i+t)\big)\leq
e^{ct}O\big(E(u)^{1/4}\big)
\end{equation}
for some constant $c$ and for all $t$ and $i\in\Z_k$. Setting $t=1$,
we obtain \eqref{eq:pseudo-orbit} and \eqref{eq:eta-energy}.
  \end{proof}

\begin{proof}[Proof of Theorem \ref{thm:cross-energy}]
  Now we are in a position to prove Case \ref{CEb} under the
  requirement that $J'$ is sufficiently $C^\infty$-close to $J$ in the
  class of $k$-periodic in time almost complex structures on $M$. Let
  $U$ be an isolating neighborhood of $K$ such that the closure of $U$
  is also contained in an isolating neighborhood.

  Arguing by contradiction, assume that $E(u)$ can be arbitrarily
  small, i.e., there exists a sequence
  $u_k\colon \R\times S^1_k\to M$, where $k=k_i\to \infty$ with
  $E(u_k)\to 0$ such that the image of $\tu_k$ is not entirely
  contained in the closure of $\tU$.  Consider the largest
  half-cylinder in $\R\times S^1_k$ whose image is contained in the
  closure. By Proposition \ref{prop:pseudo-orbit}, the restriction
  $\gamma_k$ of $u_k$ to the boundary of this cylinder gives rise to a
  $k$-periodic $\eta_k$-pseudo-orbit, with
  $\eta_k=O\big(E(u_k)\big)\to 0$, through a point $p_k$ close to
  $\p U$. (Here we implicitly use \eqref{eq:po-distance}: the moment
  of time where $\gamma_k$ is tangent to $\p U$ might not be an
  integer. However, \eqref{eq:po-distance} ensures that the value of
  $\gamma_k$ at the closest integer point is close to $\p U$.) Thus we
  obtain longer and longer two-directional $\eta_k$-pseudo-orbits with
  $\eta_k\to 0$ passing through some point $p_k$ near $\p U$. Passing
  to a subsequence, we can ensure that the sequence $p_{k_i}$
  converges and taking the limit as $E(u_k)\to 0$, we obtain an entire
  orbit of $\varphi$ which is contained in $\bar{U}$, but not in
  $K$. This is impossible since $K$ is locally maximal. The argument
  is spelled out in detail in a very similar context in
  \cite{CGGM:Reeb,CGGM:hyperbolic}.
\end{proof}

Theorem \ref{thm:cross-energy} gives a lower bound on the energy of a
Floer cylinder $u$ asymptotic on one side to a periodic orbit with
initial condition in $K$ and to a periodic orbit outside $K$ on the
other, provided that $K$ is locally maximal. When $K$ is in addition
hyperbolic, it is enough to require that one ``end'' of $u$ is in $K$.
To be more precise, consider a Floer cylinder
$u\colon \Sigma=\R\times S^1_k\to M$ for some $k$-periodic almost
complex structure $J'$ sufficiently close to a fixed 1-periodic almost
complex structure $J$ as in Theorem \ref{thm:cross-energy} and
asymptotic to $k$-periodic orbits $x$ and $y$ of $\varphi^t$ with
$x(0)\in K$.  (It does not matter if $u$ is a asymptotic to $x$ at
$\infty$ or $-\infty$, and whether $y(0)$ is in $K$ or not.) Then we
have the following refinement of Theorem \ref{thm:cross-energy}, which
is an easy consequence of the theorem and the Anosov Closing Lemma;
\cite[Thm.\ 6.4.15]{KH}.

The latter lemma asserts that whenever $\hz$ is a $k$-periodic
$\eta$-pseudo-orbit of $\varphi$ contained in a sufficiently small
neighborhood $U$ of $K$, there exists a true $k$-periodic orbit
$\hat{w}=\{w_0,\ldots,w_{k}=w_0\}$ in $K$ shadowing $\hz$, i.e., such that
$d(z_i,w_i)< C\eta$. Here $\eta$ has to be sufficiently small, but the
upper threshold is completely determined by $\varphi$ and $U$ and
independent of $k$ and $\hz$. Likewise, $C$ depends only on $U$ and
$\varphi$, but not on $k$ or $\hz$.

\begin{Proposition}
  \label{prop:energy}
  Assume that $K$ is a locally maximal hyperbolic invariant set of
  $\varphi$. Then
  \begin{equation}
    \label{eq:gap}
  E(u)=|\CA(x)-\CA(y)|>\eps_K, \textrm{ unless } E(u)=0,
 \end{equation}
 for some constant $\eps_K>0$, independent of $u$ and $x$ and $y$, and
 also of $k$ and $J'$ as long as $J'$ is sufficiently $C^\infty$-close
 to $J$ in the class of $k$-periodic in time almost complex structures
 on $M$.
  \end{Proposition}

This proposition generalizes Example \ref{ex:hyperbolic}.  
  
  \begin{proof} 
    We consider two cases depending on the location of the orbit~$y$.
    The first case is when $y(0)\not\in K$. Then $\ty$ is not entirely
    contained in any isolating neighborhood of $\tilde{K}$. Applying
    Theorem \ref{thm:cross-energy} (Case b), we obtain \eqref{eq:gap}
    with $\eps_K=c_\infty$.

  The second case is when $y(0)\in K$. Then both $\hat{x}$ and
  $\hat{y}:=\{y_i:=y(i)\mid i\in\Z_k\}$ are $k$-periodic orbits of
  $\varphi$ in $K$. Let us assume that $u$ is asymptotic to $x$ at
  $-\infty$; the other case is handled similarly.  We will show that
  then $x=y$ and $E(u)=0$ when $E(u)$ is below a certain threshold
  which depends only on $M$, $J$, $K$ and $H$, but neither on $u$ nor
  $x$ nor $y$ nor~$k$.

  To this end, fix a sufficiently small isolating neighborhood $U$ of
  $K$. In particular, we may assume that the Anosov Closing Lemma
  applies to $\varphi$ on $U$. Then, by Theorem
  \ref{thm:cross-energy}, $\tilde{u}$ is entirely contained in
  $\tilde{U}$. Hence, $\hat{z}$ is contained in $\tilde{U}\cap M=U$
  for all $s\in \R$.  Assume furthermore that $E(u)$ is so small that
  Proposition \ref{prop:pseudo-orbit} applies and thus
  $\hat{z}=\{u(s,i)\mid i\in\Z_k\}$ is a periodic $\eta$-pseudo-orbit
  in $U$ for every $s$.

  Therefore, by the Anosov Closing Lemma, there exists a true periodic
  orbit $\hat{w}$ in $K$ shadowing $\hat{z}$. Namely, we have
  $d(z_i,w_i)<C\eta$, $i\in\Z_k$, for some constant $C>0$, which
  depends only on $U$ and $\varphi$. By \cite[Cor.\ 6.4.10]{KH},
  $\varphi|_K$ is expansive: there is a constant $\delta>0$ such that
  any two distinct orbits $\{v_i\}$ and $\{v'_i\}$ of $\varphi$ in $K$
  are at least $\delta$ apart, i.e., $d(v_j,v'_j)>\delta$ for some
  $j\in \Z$. It follows that when $E(u)$ and hence $\eta$ are small
  enough (e.g., $2C\eta<\delta$), the orbit $\hat{w}$ is unique and
  depends continuously on $\hat{z}$ and thus on $s\in\R$. Therefore,
  again since $\varphi|_K$ is expansive, $\hat{w}$ is independent of
  $s\in\R$. Clearly, when $s$ is close to $-\infty$, we have
  $\hat{w}=\hat{x}$, and $\hat{w}=\hat{y}$ when $s$ is close to
  $\infty$. Thus, $x=y$, and setting $u(\infty, t)=x(t)=y(t)$, we can
  view $u$ as a $C^0$-map from
  $\T^2=\big(\R\cup\{\infty\}\big)\times S^1_k$ to $M$. This map is
  smooth on the complement to $\{\infty\}\times S^1_k$. Furthermore,
  it is easy to see from \eqref{eq:ptwise-bnd} and Proposition
  \ref{prop:pseudo-orbit} that for every $s\in \R$ the loop
  $t\mapsto u(s,t)$ is $C^0$-close to the loop $x=y$ pointwise
  uniformly in $s$.  Hence, the loop $s\mapsto u(s,t)$ lies in a small
  neighborhood of $x(t)$. As a consequence, $u$ contracts to $x$ in
  $M$, and hence $E(u)=0$.  Indeed, $E(u)$ is the difference of
  actions of capped periodic orbits. Since $x=y$, this difference is
  the integral of $\omega$ over~$u$. The cycle represented by $u$ is
  homologous to zero, and hence the integral is zero.
\end{proof}

\begin{Remark}
An ``elementary'' proof of the crossing energy theorem in the case
where $K$ is a single point and $M=\CP^n$ relying on the machinery of
generating functions is given in \cite{Al}.
\end{Remark}  

\subsection{From crossing energy to barcode growth}
\label{sec:energy-B}    
Now we are ready to prove Theorem \ref{thm:B}. Without loss of
generality, we may assume that $\htop\big(\varphi|_K\big)>0$; for
otherwise the statement is vacuous.

A minor technical issue we need to first deal with is that a
hyperbolic invariant set need not be locally maximal, and hence
Theorem \ref{thm:cross-energy} and Proposition \ref{prop:energy} do
not directly apply.  However, there automatically exists a locally
maximal hyperbolic invariant set $K'$ with
$\htop\big(\varphi|_{K'}\big)$ arbitrarily close to
$\htop\big(\varphi|_K\big)$.  This is a consequence of \cite[Thm.\
3.3]{ACW} and the Variational Principle, \cite[Thm.\ 4.5.3]{KH}.  (The
hyperbolicity condition is essential for the former result, but not for
the Variational Principle.)  To be more precise, for every $\delta>0$,
one can find $K'$ such that
$\htop\big(\varphi|_{K'}\big)\geq \htop
\big(\varphi|_K\big)-\delta$. As a consequence, without loss of
generality, we can assume the hyperbolic set $K$ to be locally maximal.

Denote by $p(k)$ the number of $k$-periodic points of
$\varphi|_K$. Since $K$ is hyperbolic, by Theorem \ref{thm:p-htop}, we
have
\begin{equation}
\label{eq:p-htop}
\htop(\varphi|_K)=\limsup_{k\to\infty} \frac{\log^+ p(k)}{k}.
\end{equation}
Hence, to prove the theorem, it is sufficient to show that 
\begin{equation}
  \label{eq:p-b}
\fb_\eps\big(\varphi^k\big)\geq p(k)/2
\end{equation}
when $\eps>0$ is small. We will use Proposition \ref{prop:energy} to
prove this for $\eps<\eps_K$.

Fix $k\geq 1$ and recall from Section \ref{sec:Lagrangian-barcode} that the limit,
\eqref{eq:b-eps3}, in the definition of $\fb_\eps\big(\varphi^k\big)$ is
attained, i.e, there exists non-degenerate, arbitrarily
$C^\infty$-small perturbations $\psi$ of $\varphi^k$ such that
\begin{equation}
  \label{eq:psi-phi}
  \fb_\eps(\varphi^k)= \fb_\eps(\psi).
\end{equation}

All $k$-periodic points of $\varphi$ in $K$ are non-degenerate and
hence persist as fixed points of $\psi$ or, equivalently, as
$k$-periodic orbits $S^1_k\to M$ of the Hamiltonian flow $\psi^t$. In
fact, we can take $\psi=\varphi$ on a small neighborhood of $K$.
Moreover, by Proposition \ref{prop:energy}, $E(u)>\eps_K$
for any Floer cylinder $u$ asymptotic to such an orbit $x$ of $\psi$
when $\psi$ is sufficiently $C^\infty$-close to $\varphi^k$.

Let $\mathcal{K}$ be the collection of fixed points of $\psi$
corresponding to the $k$-periodic points of $\varphi$ in $K$. There
are exactly $p(k)$ of them: $|\mathcal{K}|=p(k)$.  By Proposition
\ref{prop:energy} every such fixed point is $\eps$-isolated in the
sense of Section \ref{sec:Lagrangian-barcode} and, by Proposition
\ref{prop:isolated} and \eqref{eq:psi-phi},
$$
\fb_\eps(\varphi^k)= \fb_\eps(\psi)\geq |\mathcal{K}|/2=p(k)/2,
$$
which proves \eqref{eq:p-b} and completes the proof of the
theorem. \qed

\subsection{Theorem \ref{thm:B} and crossing energy for Reeb flows}
\label{sec:energy-Reeb}
The proof of Theorem \ref{thm:B} in the Reeb case follows roughly the
same path as in the Hamiltonian case, but the argument is more delicate
and there is one serious additional difficulty to overcome. The proof is
still based on the crossing energy theorem which now takes the
following admittedly rather involved form, in which we rely on the
notation and conventions from Section \ref{sec:setting}.

\begin{Theorem}[Crossing Energy Theorem, II; Thm.\ 5.1 in \cite{CGGM:Reeb}]
  \label{thm:CE}
  Let $K$ be a hyperbolic, locally maximal compact invariant set of
  the Reeb flow of $\alpha$. Fix an interval
  $$
  I=[r_-,\, r_+]\subset (1,\,\rmax)
  $$
  and let $H(r,x)=h(r)$ be a semi-admissible Hamiltonian with
  $\slope(H)=:a \not\in\CS(\alpha)$ such that
\begin{equation}
  \label{eq:h'''1}
  h'''\geq 0 \textrm{ on } [1,\, r_++\delta]
\end{equation}
for some $\delta>0$ with $r_++\delta<\rmax$. Fix an admissible almost
complex structure $J$. Furthermore, let $z$ be a $\tau$-periodic orbit of
$\alpha$ in $K$ and $\tz:=(z, r^*)$ be the corresponding 1-periodic
orbit of the flow of $sH$. (Hence, $sa\geq \tau$ and $r^*$ depends on
$s$.) Assume that $r^*\in I$.

Then there exists $\sigma >0$ independent of $s$ and $z$ such that
$E(u) \geq \sigma$ for any Floer cylinder
$u \colon \R \times S^1 \to \WW$ of $\tau H$ asymptotic to $\tz$ at either
end, unless $E(u)=0$ and thus $u$ is trivial, i.e., independent of the
first coordinate.
\end{Theorem}

Here again the point that $\sigma$ is independent of $\tau$ is
crucial for our purposes. (For $\sigma$ depending on $\tau$ the assertion
is an easy consequence of Gromov compactness and requires no addition
conditions.) Furthermore, fix $0<\sigma'<\sigma$ and $\tau\geq 0 $. Consider
a $C^\infty$-small $\tau$-periodic in time, non-degenerate perturbation
$H'$ of $\tau H$ and a $C^\infty$-small compactly supported generic
$\tau $-periodic perturbation $J'$ of $J$. The 1-periodic orbit $\tz$ of
$\tau H$ from Theorem \ref{thm:CE} splits into several non-degenerate
periodic orbits of $H'$ contained in a small tubular neighborhood of
$\tz$. It follows again from a suitable version of the Gromov
compactness theorem (see, e.g., \cite{Fi}) that every Floer cylinder
of $H'$ asymptotic to any of these orbits at either end has energy
greater than $\sigma'$.

The Reeb case of Theorem \ref{thm:B} follows from Theorem \ref{thm:CE}
roughly in the same fashion as its Hamiltonian counterpart; see
\cite[Sec.\ 5.2]{CGGM:Reeb} with \eqref{eq:hbar-dyn} used as the
definition of barcode entropy. The parallels between dynamics of
hyperbolic sets for flows and maps, and in particular
\eqref{eq:p-htop}, are laid out in detail in \cite{FH}. An extra step
needed in the proof is showing that the growth of the number of
1-periodic orbits of $\tau H$ in $K\times I$ is the same on the
exponential scale as the growth of periodic orbits of the Reeb flow in
$K$ with period up to $\tau$.  The technical condition, \eqref{eq:h'''1},
does not cause any complications because one can always find an
admissible Hamiltonian meeting this requirement.

A new difficulty arises in the proof of Theorem \ref{thm:CE}
and this is where \eqref{eq:h'''1} enters the picture. (We do not know
if this requirement is essential or not.)  The difficulty manifests
itself in several ways.

First, the set $K\times \{r\}\subset \WW$, for any $r\geq 1$, is
never locally maximal for the flow of $H$ regardless of whether $K$ is
locally maximal or not. Indeed, when $K$ is invariant all sets
$K\times \{r\}$ are invariant. Furthermore, this set is never
hyperbolic. As a consequence, one cannot directly apply Theorem
\ref{thm:cross-energy} or Proposition \ref{prop:energy} to $H$. On a
finer level, the problem is that the Reeb flow slows down as $r\to 1+$
turning into the identity flow on $W$ where $H\equiv 0$. Hence, the
part of $\gamma(t):=u(s,t)$ lying in a small neighborhood $r\leq r_0$
of $W$ is automatically an $\eta$-pseudo-orbit for a large interval of
time, shadowing a constant orbit, when $E(u)$ is small. Here
$\eta\to 0+$ as $E(u)\to 0$ and the time interval goes to infinity as
$r_0\to 1$, and the proof of Theorem \ref{thm:cross-energy} breaks
down: the pseudo-orbits $\gamma$ simply converge to a point in
$M=\p W$. (Note that a part that is contained in $W$ of any Floer
cylinder is necessarily holomorphic, and hence $E(u)$ is automatically
bounded away from 0 by the standard monotonicity whenever $u$ reaches
into $W$ beyond a fixed neighborhood of $M=\p W$. As a consequence,
the problem occurs only when $\min r\circ u$ approaches 1 from above
or below.)

In other words, the issue arises when there is a sequence of Floer
cylinders $u_k$ for $\tau_k H$ with $\tau_k\to\infty$ and
$E(u_k)\to 0$ asymptotic to orbits $\tz=(z,r^*)$ in
$K\times [1,\infty)$, perhaps with some control of the second
component $r^*$, such that $\min r\circ u_k\to 1$. This is precluded
by the following.

\begin{Theorem}[``Minimum Principle''; Thm.\ 6.1, \cite{CGGM:hyperbolic}]
\label{thm:location}
Let $H(r,x) =h(r)$ be a semi-admissible Hamiltonian.  Assume that
$1<r_*^-\leq r_*^+$ and $\delta >0$ are such that
$$
1<r_*^--\delta \text{ and }
r_*^+ +\delta \leq \rmax,
$$
and
\begin{equation}
  \label{eq:h'''2}
  h''' \geq 0 \text{ on }
  [1, \, r_*^++\delta] \subset [1, \, \rmax).
\end{equation}
Fix an admissible almost complex structure $J$. Then there exists
$\sigma_0 >0$ such that for any $\tau >0 $ and any Floer cylinder
$u \colon \R \times S^1 \to \WW$ for $\tau H$ with energy
$E(u) \leq \sigma_0$ and asymptotic, at either end, to a periodic
orbit in $M\times [r_*^-,\,r_*^+]$, the image of $u$ is contained in
$M\times (r_*^--\delta,\, r_*^++\delta)$.
\end{Theorem}

In other words, a small energy Floer cylinder cannot bulge down too
much: such a cylinder for $\tau H$ asymptotic at either end to a
periodic orbit in the shell $M\times [r_*^-,\,r_*^+]$ must be entirely
contained in a slightly larger shell
$M\times (r_*^--\delta,\, r_*^++\delta)$. The new part of this theorem
is the lower bound $r_*^--\delta$; hence, the ``Minimum
Principle''. The upper bound is ``Bourgeois--Oancea monotonicity''
established in \cite[p.\ 654]{BO}; see also \cite[Lemma 2.3]{CO}. With
this result in mind, Theorem \ref{thm:CE} is proved by exactly the
same argument as Theorem \ref{thm:cross-energy}. We refer the reader
to \cite{CGGM:Reeb,CGGM:hyperbolic} for detailed proofs. 

\subsection{Digression: Spectral norm and other applications}
\label{sec:sectral+inv}
The goal of this section is to examine another application of the
crossing energy theorem. Namely, we use that theorem to investigate
the behavior of the spectral norm of the iterates.

Throughout the section $M$ will be a closed symplectic manifolds as in
Sections \ref{sec:conv} and \ref{sec:Floer-complex}. The spectral norm
$\gamma(\varphi)$ of a Hamiltonian diffeomorphism
$\varphi\colon M\to M$ defined by \eqref{eq:spectral} measures the
symplectic size of $\varphi$, and hence the distance between
Hamiltonian diffeomorphisms. In a variety of questions in symplectic
dynamics it is useful to understand the behavior of the sequence
$\gamma\big(\varphi^k\big)$.

For some manifolds, such as $\CP^n$, $\gamma(\varphi)$ is \emph{a
  priori} bounded from above; see \cite{EP} and also, e.g.,
\cite{KS}. Furthermore, $\gamma(\varphi)$ is bounded from above by
twice the displacement energy of $\supp\varphi$ and here $M$ can also
be open but convex at infinity; see, e.g.,
\cite{FGS,FS:convex,Oh:spec,Sc,Vi}.  On the other hand, when $M$ is a
surface of positive genus, it is easy to construct $\varphi$ such that
$\gamma\big(\varphi^k\big)\to \infty$. Moreover, it is not hard to
find $\varphi$ such that $\gamma\big(\psi^k\big)\to \infty$ whenever
$\psi$ is $\gamma$-close to $\varphi$, i.e.,
$\gamma(\varphi\psi^{-1})$ is small. To the best of our knowledge, it
is unknown whether or not $\gamma\big(\varphi^k\big)\to \infty$
$C^r$-generically, for any $r\geq 0$, for surfaces of positive genus;
see Problem \ref{prob:pattern}.

Here, however, we are interested in the question if and when
$\gamma\big(\varphi^k\big)$ can be arbitrarily small. The question
arises, for instance, in connections with some attempts to establish a
higher-dimensional analogue of the $C^\infty$-closing lemma or
Lagrangian Poincar\'e recurrence; see \cite{CS,GG:PR,JS} and
references therein. Set
$$
  \bg(\varphi)=\liminf_{k\to\infty} \gamma\big(\varphi^k\big)\in
  [0,\,\infty].
$$
When $\varphi$ is a rotation of $S^2$, we have $\bg(\varphi)=0$. We
leave the proof of this fact to the reader as an exercise. The same is
true for Hamiltonian diffeomorphisms of $\CP^n$ generated by quadratic
forms and Hamiltonian pseudo-rotations of $\CP^n$; see
\cite{GG:hyperbolic,GG:PR,GG:PRvsR} and Remark \ref{rmk:PRs}. However,
this behavior is atypical as the following results show.

\begin{Theorem}[Thm.\ 3.1, \cite{CGG:Spectral}]
  \label{thm:gamma}
  Let $\varphi\colon M\to M$ be a Hamiltonian diffeomorphism of a
  closed symplectic manifold $M$ with more than $\dim \H_*(M;\Q)$
  hyperbolic periodic points.  Then $\bg(\varphi)>0$. Moreover,
  $\bg$ is locally uniformly bounded away from zero near
  $\varphi$, i.e., 
  there exists $\delta>0$, possibly depending on $\varphi$, and
  a sufficiently $C^\infty$-small neighborhood $\CU$ of $\varphi$ such
  that
  $$
    \bg(\psi)>\delta \textrm{ for all } \psi\in \CU.
  $$
\end{Theorem}

Without the moreover part, this theorem was originally proved in
\cite{CGG:Growth}. Conceptually, the proof of this and other results
below hinge on two crucial facts. The first one is the crossing energy
theorem -- Theorem \ref{thm:cross-energy} -- applied in the case where
$K$ is just one hyperbolic point; see Example \ref{ex:hyperbolic}. The
second fact, proved in \cite{KS}, is that $\gamma(\varphi)$ is an
upper bound on the boundary depth (see \cite{Us1,Us}), i.e., the
length $\beta_{\max}(\varphi)$ of longest finite bar in the barcode of
$\varphi$:
\begin{equation}
  \label{eq:bmax-gamma}
  \beta_{\max}(\varphi)\leq \gamma(\varphi).
\end{equation}  
Strictly speaking, the argument in \cite{KS} deals only with the
contractible part of the Floer homology, but the proof extends to all
free homotopy classes word-for-word.

For instance, to prove Theorem \ref{thm:gamma} without the moreover
part we argue as follows. The condition that $\varphi$ has more
hyperbolic fixed points than $\dim \H_*(M)$ guarantees that
$\varphi^k$, for every $k\in \N$, has a finite bar bounded away from 0
by a constant $c_\infty>0$ independent of $k$. Here we are using the
crossing energy theorem and Proposition \ref{prop:isolated}. Hence, by
\eqref{eq:bmax-gamma},
$$
\gamma\big(\varphi^k\big)\geq \beta_{\max}\big(\varphi^k\big)\geq c_\infty
$$
for all $k$. As a consequence, $\bg(\varphi)\geq c_\infty$.

\begin{Example}
  Assume that $M$ is a closed surface and $\htop(\varphi)>0$. Then
  $\varphi$ has infinitely many hyperbolic periodic points,
  \cite{Ka}. Hence, $\bg(\varphi)>0$. Moreover,
  $\bg(\psi)>\delta$ for some $\delta>0$ and all $\psi$ which are
  $C^\infty$-close to $\varphi$. Also note in connection with Theorem
  \ref{thm:generic} and Corollary \ref{cor:sphere} below that
  $\htop>0$ is a $C^\infty$-generic condition in dimension two;
  \cite{LCS}.
\end{Example}

\begin{Theorem}[Thm.\ 3.3, \cite{CGG:Spectral}]
\label{thm:generic}
Let $M$ be a closed symplectic manifold.  The function $\bg$ is
locally uniformly bounded away from zero on a $C^\infty$-open and
dense set of Hamiltonian diffeomorphisms $\varphi\colon M\to M$, i.e.,
for every $\varphi$ in this set there exists $\delta>0$, possibly
depending on $\varphi$ but not on $\psi$, such that
\[
  \bg(\psi)>\delta
\]
whenever $\psi$ is sufficiently $C^\infty$-close to $\varphi$.
\end{Theorem}

\begin{proof}[Outline of the proof] For the sake of simplicity, we
  focus on the statement that $\bg(\varphi)>0$ for $\varphi$ from a
  $C^\infty$-open and dense set. (We do not claim however that this is
  an open condition.) Furthermore, we assume in addition that the
  ambient manifold $M$ meets one of the following requirements:
\begin{itemize}
   \item[(i)] $n:=\dim M/2$ is odd;
   \item[(ii)] $\H_{\scriptscriptstyle{odd}}(M)\neq 0$;
   \item[(iii)] the minimal Chern number of $M$ is greater than 1.
\end{itemize}
Below we refer to a closed symplectic manifold meeting at least one of
these requirements as a \emph{Sugimoto manifold}; cf.\
\cite{Su21}. For this class of manifolds, every strongly
non-degenerate Hamiltonian diffeomorphism $\varphi$ of a closed
symplectic manifold $M^{2n}$ has a non-hyperbolic periodic point or
infinitely many hyperbolic periodic points.

Indeed, it follows from Floer theory that under one of the conditions,
every non-degenerate Hamiltonian diffeomorphism $\varphi$ must have a
fixed point of Conley--Zehnder index of the same parity as $n$; see
\cite{Su21}. Such a point is either non-hyperbolic or hyperbolic with
some real negative eigenvalues. In the former case, we are done for
$\varphi$. In the latter, we replace $\varphi$ by $\varphi^2$ and
apply the argument again. Proceeding inductively, we obtain a
non-hyperbolic periodic point or infinitely many hyperbolic periodic
points.

Next denote by $\CV_m$, $m\in \N$, the set of Hamiltonian diffeomorphisms
with at least $m$ hyperbolic points. 
We claim that for every $m\in \N$ the set $\CV_m$ is $C^\infty$-open
and dense. Then the theorem, in the limited form we are proving here,
follows from the claim and Theorem \ref{thm:gamma} with
$m>\dim \H_*(M;\Q)$.

The statement that $\CV_m$ is $C^\infty$-open and even $C^1$-open is
obvious. (It is essential here that $m$ is finite.)  To show that it
is $C^\infty$-dense, we argue as follows. Given a Hamiltonian
diffeomorphism $\varphi$, we need to find $\psi\in \CV_m$ arbitrarily
$C^\infty$-close to $\varphi$.  Since the set of strongly
non-degenerate Hamiltonian diffeomorphisms is $C^\infty$-dense, we can
assume that $\varphi$ is in this class. Then $\varphi$ has infinitely
many hyperbolic periodic points or a (non-degenerate) non-hyperbolic
periodic point. In the former case, $\varphi\in\CV_m$ for all
$m\in\N$.  In the latter case, by \cite[Prop.\ 8.2]{Arn}, for any
$m\in\N$ there exists $\psi\in\CV_m$ arbitrarily close to $\varphi$.
\end{proof}

\begin{Remark}
  The above argument is taken from \cite{Su21} where it is used to
  prove the $C^\infty$-generic Conley conjecture for Sugimoto
  manifolds, i.e., that $C^\infty$-generically $\varphi$ has
  infinitely many periodic orbits.  In fact, it shows slightly more:
  $C^\infty$-generically $\varphi$ has infinitely many hyperbolic
  periodic orbits. The actual proof of Theorem \ref{thm:generic} in
  full generality follows a different path and, while also using
  \cite[Prop.\ 8.2]{Arn}, it mainly relies on the approach developed
  in \cite{CS} in connection with the strong closing lemma in all
  dimensions.
\end{Remark}

In several situations, Theorem \ref{thm:generic} can be made slightly
more precise. For instance, we have the following result, also
originally proved in \cite{CGG:Growth} without the moreover part.

\begin{Theorem}[Cor.\ 3.4, \cite{CGG:Spectral}]
  \label{cor:sphere}
  Assume that $M$ is a surface and $\varphi$ is strongly
  non-degenerate. Then $\bg(\varphi)>0$ when $M$ has positive
  genus. When $M$ is the two-sphere, $\bg(\varphi)=0$ if and only if
  $\varphi$ is a pseudo-rotation, i.e., $\varphi$ has exactly two
  periodic points. Moreover, $\bg$ is locally uniformly bounded from 0
  on the set of all strongly non-degenerate Hamiltonian
  diffeomorphisms $\varphi$ when $M$ has positive genus and on the set
  of such $\varphi$ with at least three fixed points when $M=S^2$.
\end{Theorem}

\begin{proof}
  When $M$ has positive genus, a Conley conjecture type argument
  guarantees that $\varphi$ has infinitely many hyperbolic periodic
  points; see \cite{FrHa, GG:survey, SaZe} or \cite{LCS}. Thus, in this
  case, the statement follows directly from Theorem \ref{thm:gamma}.

  Concentrating on $M=S^2$, first recall that for all pseudo-rotations
  the sequence $\gamma\big(\varphi^k\big)$ contains a subsequence
  converging to zero. Hence $\bg(\varphi)=0$. In the opposite
  direction, when $M=S^2$, the existence of one positive hyperbolic
  periodic point is enough to ensure that $\bg(\varphi)>0$ by Example
  \ref{ex:hyperbolic}, and, moreover, $\bg$ is locally uniformly
  bounded away from zero. Hence, more generally, without any
  non-degeneracy assumption, if $\bg(\varphi)=0$, then all periodic
  points of $\varphi$ are elliptic. For strongly non-degenerate
  Hamiltonian diffeomorphisms $\varphi$, this forces $\varphi$ to be a
  pseudo-rotation.
\end{proof}

\begin{Remark}[Pseudo-rotations and invariant sets]
  \label{rmk:PRs}
  Another large class of applications of the crossing energy theorem
  concerns invariant sets of Hamiltonian pseudo-rotations. In
  dimension 2, we say that an orientation and area preserving
  diffeomorphism $\varphi\colon S^2\to S^2$ is a pseudo-rotation (PR)
  if it has exactly two periodic points which are then necessarily the
  fixed points. Franks' theorem asserts that $\varphi$ is a PR if and
  only if it has finitely many periodic points; see
  \cite{Fr96,Fr99,LeC}. Furthermore, by a theorem of Le Calvez--Yoccoz,
  no fixed point of a PR of $S^2$ is locally maximal i.e.,
  isolated as an invariant set; see \cite{LCY} and also
  \cite{Fr99,FM}.

  Both of these results have analogues in symplectic dynamics in
  higher dimensions. For our purposes, it is convenient to call a
  Hamiltonian diffeomorphism $\varphi\colon \CP^n\to \CP^n$ a PR if it
  has exactly $n+1$ periodic points which are then automatically the
  fixed points; \cite{GG:PR}. In this setting Franks' theorem was
  conjectured by Hofer and Zehnder in \cite{HZ} (i.e., that $\varphi$
  is a PR if and only if it has finitely many periodic points) and
  proved, under minor non-degeneracy conditions, in \cite{Sh:HZ}.
  A direct analogue of Le Calvez--Yoccoz theorem for $\CP^n$ was
  established in \cite{GG:PR} by using the crossing energy theorem,
  Theorem \ref{thm:cross-energy}. A closely related result relying on
  an earlier version of the crossing energy theorem is that any
  Hamiltonian diffeomorphism of $\CP^n$ with a hyperbolic periodic
  point must have infinitely many periodic orbits;
  \cite{GG:hyperbolic}.

  The Hofer--Zehnder conjecture (aka Franks's theorem) also has a variant for Reeb flows on $S^{2n-1\geq 5}$; see
  \cite{CGG:HZ}. Likewise, Reeb analogues of the Le Calvez--Yoccoz
  theorem and the hyperbolic fixed point theorem for certain Reeb
  flows on $S^{2n-1\geq 5}$ are proved in \cite{CGGM:hyperbolic}.

  Finally, we note that PRs can have extremely interesting and
  surprising dynamics. A PR of $S^2$ can have
  exactly three ergodic measures: the fixed point and the area form;
  see \cite{AK} and also \cite{FK} for a survey of results in general
  dynamics of PRs. Examples for $\CP^n$ with exactly $n+2$ ergodic
  measures were constructed in \cite{LRS}. A Reeb PR on $S^{2n+1}$ can
  be ergodic; \cite{Ka:Izv}.
\end{Remark}  




\section{Toric integrable Hamiltonian systems}
\label{sec:integrable}
It is illuminating to contrast the exponential growth of the barcode,
reflecting complexity of the underlying dynamics, with its behavior for
simple dynamical systems such as completely integrable systems
with toric singularities, considered below. We follow \cite{BG}.

\subsection{Toric Hamiltonians}
\label{sec:Toric-Ham}
Let $M^{2n}$ be a toric symplectic manifold, i.e., a manifold admitting a
faithful Hamiltonian $\T^n$-action; see, e.g., \cite[Chap.\
XI]{CdS}. Let
$$
\mu\colon M\to \R^n
$$
be the moment map and $\Delta:=\mu(M)$ the moment polytope, usually
referred to as the Delzant polytope in the symplectic context. (This
polytope completely determines $(M, \omega)$ together with the action;
\cite{De}.)

Let now $h\colon \Delta\to \R$ be a smooth function. Setting
$H=h\circ \mu$ we obtain an (autonomous) Hamiltonian on $M$. The
Hamiltonian flow of $H$ is completely integrable in a very strong
sense with $n$ first integrals $x_i\circ \mu$, where $x_1,\ldots, x_n$
are the coordinates on $\R^n$, generating a Hamiltonian $\T^n$-action
commuting with the flow. (Complete integrability is understood here as
that the functions $x_i\circ\mu$ are in involution, invariant under
the flow of $H$ and independent almost everywhere on $M$; see
\cite[Sec.\ 49]{Ar}.) By construction, these integrals generate a
Hamiltonian $\T^n$-action. Hence, in addition, the set where these
functions fail to be independent has a rather simple structure. It is
customary to say that the first integrals have toric singularities.

\begin{Theorem}[Thm.\ 2.7, \cite{BG}]
  \label{thm:toric-sympl}
Assume that $h$ is convex or concave or real
analytic. Then
\begin{equation}
  \label{eq:toric-sympl}
\fb_\eps\big(\varphi^k_H\big)\leq C_n(h) k^n + C_0(h)
\end{equation}
for all $k\in \N$, where the constants $C_n(h)$ and $C_0(h)$ are independent of
$\eps>0$ and, of course, $k$.
\end{Theorem}

Here $h$ is said to be real analytic whenever it extends to a real
analytic function on a neighborhood of $\Delta\subset\R^n$. Note also
that since $k\geq 1$, we can equivalently set $C_0(h)=0$ in
\eqref{eq:toric-sympl} at the expense of increasing $C_n(h)$. 

\begin{proof}[Outline of the proof] The idea of the proof is to
  construct a non-degenerate perturbations $\psi$ of $\varphi_H^k$
  such that the number of fixed points of $\psi$ is bounded similarly to the
  right hand side of \eqref{eq:toric-sympl} with both constants
  independent of $k$; cf.\ \cite[Thm.\ 3.1]{BG}.  Then
  \eqref{eq:toric-sympl} will follow. We will focus on the
  convex case and then briefly indicate the changes needed when $h$ is
  real analytic. The proof for a concave function $h$ is identical.

  Let $F$ be an open face of $\Delta$ of dimension $d$ and
  $V=V_F\subset \R^n$ the linear space parallel to the affine space
  containing $F$. This is (the dual of) the Lie algebra of a torus, a
  certain quotient of $\T^n$, and hence $V$ carries a canonical
  lattice and a canonical measure.  Fixing a basis in that lattice we
  identify it with $\Z^d$ and $V$ with $\R^d$. Set
  $G_F=\nabla (h|_F)$. We can treat $G_F$ as a map $F\to V$. 

  The inverse image $L_w:=\mu^{-1}(w)$, where $w\in F$, is an
  isotropic torus of dimension $d$ invariant under the flow of $H$ and
  an orbit of the $\T^n$-action. This torus is comprised entirely of
  periodic orbits of the flow if and only if $G_F(w)$ is a rational
  vector in $V$. We will call such tori rational. (Otherwise, no
  point in $L_w$ is periodic.) Denote the components of $G_F(w)$
  by $p_i/q_i$, where, of course, $p_i$ and $q_i$ are relatively
  prime.  The minimal period $q$ of any $x\in L_w$ is the least common
  multiple of the denominators $q_1,\ldots, q_d$.

  Next, let us first assume that $h$ is strictly convex. Then $G_F$ is
  a smooth embedding. In particular, $G_F$ is one-to-one and
  $G_F(F)\subset V$ is an open subset. In fact, the closure of this
  set is a manifold with corners.  The number of rational points in
  $G_F(F)$ is bounded from above by roughly
  $\vol\big(G_F(F))\big) k^d+ O(k^{d-1})$. We can be even less precise
  here and say that it is bounded from above by $C_F k^d$, where $C_F$
  is completely determined by $\diam G_F(F)$ and continuous in
  $\diam G_F(F)$. For instance, we can set
  $C_F=\big(2\diam G_F(F)\big)^d$. The total number of rational tori,
  for all faces $F$, is bounded from above by $C k^n$ where, e.g.,
  $C=\sum_F C_F$.

  Furthermore, it is easy to see that all rational tori $L_w$ are
  Morse--Bott non-degenerate. This is a consequence of the fact that
  $w$ is automatically a regular point of $G_F$ due to again strict
  convexity. Applying a $C^\infty$-small perturbation to
  $\varphi_{H}^k$, we can split each $L_w$ into $2^d$ non-degenerate
  fixed points. Setting $C_n(h)=2^nC$ and $C_0(h)=0$ we obtain a
  perturbation $\psi$ of $\varphi_{H}^k$, with the number of fixed
  points bounded from above by $ C_n(h) k^n$ as required.

  When $h$ is just convex but not necessarily strictly convex,
  consider its $C^\infty$-small strictly convex perturbation. Then, in
  the argument above, we can take the constants independent of this
  perturbation, provided that it is sufficiently close to $h$. For
  instance, we can still set $C_F=\big(2\diam G_F(F)\big)^d$ and
  $C=\sum_F C_F$, where $G_F=\nabla(h|_F)$ for the original function
  $h$. Combining the two perturbations, we obtain a non-degenerate
  perturbation $\psi$ of $\varphi_H^k$ with the desired bound on the
  number of fixed points.

  Next, assume that $h$ is real analytic, i.e., by definition, $h$
  extends to a real analytic function on a neighborhood of
  $\Delta$. Note that in the proof strict convexity of $h$ has been
  used only at two points.  The first one is to ensure that $G_F$ is
  one-to-one, and hence $L_w$ is a single torus. Here it would be
  enough to guarantee that $|G_F^{-1}(v)|$ is bounded from above by a
  constant $N$ which is completely determined by the original function
  $h$ for every regular value $v$.  The second one is that $G_F$ has
  no rational critical values, and hence $L_w$ is Morse--Bott
  non-degenerate. Both requirements can be satisfied by taking a
  suitable arbitrarily $C^\infty$-small real analytic perturbation of
  $h$.

  For instance, let $B$ be a sufficiently small closed ball in $\R^n$
  centered at the origin. Set
  $h_\lambda(x):=h(x)+\left<\lambda,x\right>$, where $\lambda\in
  B$. In particular, $h_0=h$, and $h_\lambda$ and
  $G_{F,\lambda}:=\nabla (h_\lambda|_F)$ are small perturbations of
  $h$ and, respectively, $G_F$. Denote by $\lambda_F$ the orthogonal
  projection of $\lambda$ to $V_F$. Clearly,
$$
G_{F,\lambda}=G_F+\lambda_F.
$$
Now it is easy to show that for a full measure second Baire category
set of $\lambda\in B$, the set of critical values of all
$G_{F,\lambda}$ contains no rational points. Therefore, for such
parameters $\lambda$, the rational tori $L_w:=\mu^{-1}(w)$ are
Morse--Bott non-degenerate for the flow of $h_\lambda\circ
\mu$. Hence, the second requirement is satisfied for all such
$h_\lambda$. The fact that for every regular value $v$ the number
$|G_{F,\lambda}^{-1}(v)|$ is bounded from above by a constant $N$
completely determined by $h$, and not just $h_\lambda$, follows from
that $h_\lambda(x)$ is analytic in both $x$ and $\lambda$. We refer
the reader to \cite[Sec.\ 3]{BG} for details.
\end{proof}

\subsection{Toric Reeb flows}
\label{sec:Toric-Reeb}
Flows on toric Liouville domains are Reeb analogues of toric
integrable Hamiltonians. The definition of a toric domain, while
conceptually well-established, varies subtly between different sources
(see, e.g., \cite{CG,GH,Hu:GT16}) and is sometimes difficult to track
down specifically. Here, following \cite{BG}, we spell out the
definition suitable for our purposes.

Consider the standard action of the torus $\T^n$ on $\C^n=\R^{2n}$ by
rotations of the complex coordinate axes and let
$$
\mu\colon \C^n\to \Quad, 
\quad (z_1,\ldots,z_n)\mapsto
\pi\big(|z_1|^2,\ldots,|z_n|^2\big)
$$
be its moment map, where $\R^n_{\geq 0}\subset\R^n$ is the positive
quadrant.  We denote the intersection of 
$\R^n_{\geq 0}$ with the unit sphere $S^{n-1}\subset \R^n$ by $\Delta$. Let
$f\colon \Delta\to (0,\infty)$ be a continuous function. Note
that $f$ extends to a positive continuous function on a neighborhood
of $\Delta$ in $S^{n-1}$. Let $(r,\theta)$, where
$r\in [0,\infty)$ and $\theta\in S^{n-1}$, be the standard polar
coordinates on $\R^n$. Set
\begin{equation}
  \label{eq:Omega}
\Omega_f:=\{(r,\theta)\in\R^n\mid r\leq f(\theta), \,\theta\in
\Delta\}\subset \R^n_{\geq 0}
\end{equation}
and
$$
\p\Omega_f:=\{(r,\theta)\mid r= f(\theta), \,\theta\in
\Delta\}.
$$
Clearly, $\Omega_f$ is a star-shaped domain in the positive quadrant. By
definition, a \emph{toric} domain in $\R^{2n}$ is the
$\T^{n}$-invariant closed star-shaped domain
$$
W_f=\mu^{-1}(\Omega_f)\subset \R^{2n}.
$$

We say that $W_f$ is \emph{smooth} when $f$ is smooth, i.e., $f$
extends to a smooth function on a neighborhood of $\Delta$ in
$S^{n-1}$. Then
$$
\p W_f=\mu^{-1}(\p \Omega_f)
$$
is a smooth hypersurface in $\R^{2n}$ and the lines through the origin
are never tangent to $\p W_f$.  Likewise, $W_f$ is \emph{real
  analytic} if $f$ admits a real analytic extension to a small
neighborhood of $\Delta$ in $S^{n-1}$. Then $\p W_f$ is a real
analytic hypersurface in $\R^{2n}$. The Reeb flow on $\p W_f$, to the
extent it is defined (e.g., when $W_f$ is smooth), is again completely
integrable and the first integrals have simple singularities.

Somewhat misleadingly, we say that $W_f$ is \emph{convex} if for some
extension of $f$ (still denoted by $f$) to a small neighborhood $U$ of
$\Delta$ in $S^{n-1}$ the set
\begin{equation}
  \label{eq:convex-ext}
\{(r,\theta)\mid r\leq f(\theta), \,\theta\in
U\}\subset \R^n
\end{equation}
is convex. Then the closed domain $\Omega_f$ is
automatically convex, but our convexity requirement is somewhat
stronger than that; see \cite[Ex.\ 2.2]{BG}. We say that
$W_f$ is \emph{concave} if for some extension of $f$ to a small
neighborhood $U$ of $\Delta$ in $S^{n-1}$ the set
\begin{equation}
  \label{eq:concave-ext}
\{(r,\theta)\mid r\geq f(\theta), \,\theta\in
U\}\subset \R^n
\end{equation}
is convex. (Note that the inequalities in \eqref{eq:convex-ext} and
\eqref{eq:concave-ext} go in the opposite directions. In this case, we
also say that the sets $\Omega_f$ and \eqref{eq:concave-ext} are
concave.)  Then the set $\R^n\setminus \Omega_f$ is automatically
convex, but again our requirement is more restrictive.  The domain
$W_f$ is said to be \emph{smooth convex/concave toric domain} if in
this setting the extension of $f$ can be taken smooth. Finally,
we note that with our definitions a smooth convex toric domain need
not be convex as a subset of $\R^{2n}$; see \cite[Ex.\ 2.3]{BG}.

\begin{Theorem}[Thm.\ 2.4, \cite{BG}]
  \label{thm:toric-Reeb}
  Assume that $W_f$ is a toric domain which is real analytic, or
  smooth and convex or concave.  Then
\begin{equation}
  \label{eq:toric-Reeb}
\fb_\eps(s)\leq C_n(f) s^n + C_0(f)
\end{equation}
for all $s$, where the constants $C_n(f)$ and $C_0(f)$
are independent of $\eps>0$ and, of course, $s$.
\end{Theorem}

This constraint on the behavior of $\fb_\eps(s)$ also provides an
upper bound on the growth of another invariant, measuring the growth
of the filtered homology or to some extent the concentration of
bars. This invariant is essential for the Hofer--Zehnder conjecture type
questions. Namely, set
$$
\fh(s)=\sup_{s'<s}\dim \SH^{s'}(W)\leq \infty.
$$
Clearly,
$$
\fh(s)\leq \sup_{\eps>0} \fb_\eps(s).
$$
This invariant arises in the Hofer--Zehnder conjecture type
questions; see \cite{CGG:HZ,Sh:HZ}.  Since the constants $C_n(\Omega)$
and $C_0(\Omega)$ in \eqref{eq:toric-Reeb} are independent of $\eps$,
we arrive at the following.

\begin{Corollary}[Cor.\ 2.5, \cite{BG}]
  \label{cor:toric-Reeb}
In the setting of Theorem \ref{thm:toric-Reeb}, 
$$
\fh(s)\leq C_n(\Omega) s^n + C_0(\Omega) 
$$
for all $s$, where the constant $C_n(\Omega)$ and $C_0(\Omega)$ are
independent of $s$.
\end{Corollary}

Another consequence of Theorems \ref{thm:toric-sympl} and
\ref{thm:toric-Reeb} is that in the setting of the theorems the slow
barcode entropy is bounded from above by $n$ (see Remark
\ref{rmk:slow}), while it is infinite in the presence of a positive
entropy hyperbolic set.

\begin{Remark}
  \label{rmk:Invariance2}
  An analogue of Theorem \ref{thm:toric-Reeb} also holds for
  non-smooth convex/concave toric domains. This result is of interest
  for distinguishing open domains as in Remark \ref{rmk:Invariance};
  see \cite[Thm.\ 2.6 and Cor.\ 2.9]{BG}.
\end{Remark}  

The proof of Theorem \ref{thm:toric-Reeb} follows essentially the same
path as the proof of Theorem \ref{thm:toric-sympl} in the previous
section although some technical details are more involved. In the
convex/concave case the argument can also be extracted from \cite{GH}.

\subsection{Discussion: beyond toric integrability}
\label{sec:discussion}
There are several possible overlapping directions of generalizing or
refining Theorems \ref{thm:toric-sympl} and \ref{thm:toric-Reeb}. The
first one concerns relaxing the assumptions on the function $h$ or the
boundary $\p \Omega$. When $h$ or $\p \Omega$ are just smooth, we can have
$\fb_\eps\to \infty$ as $\eps\to 0$. Hence, $C_n$ and $C_0$ must be
dependent on $\eps$. Yet even then we see no reason why the upper
bounds \eqref{eq:toric-sympl} and \eqref{eq:toric-Reeb} would
hold. However, we do expect $\fb_\eps$ to grow at most
subexponentially in this case or perhaps even slower.

In essence, the main reason for a slow growth of $\fb_\eps$ in Theorems
\ref{thm:toric-sympl} and \ref{thm:toric-Reeb} is that the underlying
Hamiltonian flow is completely integrable, i.e., there are $n$
functions in involution (first integrals) preserved by the flow and
independent on a sufficiently large, at least open and dense, set; see
\cite[Sec.\ 49]{Ar}. Hence, another direction is relaxing the
condition that the first integrals generate a $\T^n$-action, i.e.,
that $W$ is toric, and turning to more general completely integrable
systems. (Note also that once this condition is dropped the
convexity/concavity requirements do not appear to be any longer
meaningful.) We conjecture that $\fb_\eps$ grows at most
subexponentially or perhaps even slower than that when the flow and
the first integrals are real analytic or the integrals meet suitable
non-degeneracy conditions; see Remark \ref{rmk:vol-int} and, e.g.,
\cite{El,MZ,Zu} and references therein.

Some restrictions on the first integrals are necessary. Indeed, an
example of a closed Riemannian 3-manifold $Q$ with a completely
$C^\infty$-integrable geodesic flow and an exponentially growing
$\pi_1(Q)$ was constructed in \cite{BT}. Then the geodesic flow
necessarily has positive topological entropy; \cite{Di}. Moreover,
then the set of conjugacy classes in $\pi_1(Q)$ also grows
exponentially, and hence so does the number of infinite bars
$\fb_\infty(s)$. Hence, $\fb_\eps(s)\geq \fb_\infty(s)$ grows
exponentially for every $\eps>0$.  In other words,
$\hbar_\eps\geq \hbar_\infty>0$. While at the moment we do not have a
reference or a construction, we believe that in a similar vein a
completely integrable Hamiltonian flow or diffeomorphism on a closed
symplectic manifold can have positive barcode entropy and therefore,
by \cite[Thm.\ A]{CGG:Entropy}, positive topological entropy. It would
also be interesting to have an example of a completely integrable Reeb
flow in dimension three with positive topological entropy. The same
question stands for Reeb flows on the standard contact sphere
$S^{2n-1\geq 3}$ for barcode entropy or at least topological
entropy. We believe that such a flow, at least for $2n-1\geq 5$,
can have a large hyperbolic invariant subset in its singular set, and
hence positive barcode entropy by Theorem \ref{thm:B}.

Do Theorems \ref{thm:toric-sympl} and \ref{thm:toric-Reeb} have a
relative analogue? Namely, in the settings of the theorems do we have
subexponential growth of the barcode function for any pair of
Lagrangian or Legendrian submanifolds. (While we are not aware of any
published results to this account, we firmly believe that the
topological entropy is zero in this case.)  Are there robust
polynomial growth upper bounds as in these theorems? For strategically
chosen submanifolds, the question should be accessible by reasoning
similar to the proofs of those theorems. However, for a general pair
of Lagrangian or Legendrian submanifolds the situation could be much
more involved. (Here the volume growth approach from Remark
\ref{rmk:vol-int} can be useful.)

Finally, note that while Theorems \ref{thm:toric-sympl} and
\ref{thm:toric-Reeb} provide an upper bound on the growth rate of
$\fb_\eps$, it is not clear what the actual behavior of $\fb_\eps$ is
in, say, the convex/concave case. Depending on $W$, the barcode
function $\fb_\eps$ can grow slower than these upper bounds. For
instance, it is not hard to see that when $\p W$ is an ellipsoid in
the setting of Section \ref{sec:Toric-Reeb}, i.e., $\Omega$ is cut out
from $\R^{n}_{\geq 0}$ by a hyperplane, $\fb_\eps(s)$ grows linearly
regardless of the dimension. When $W$ is a closed toric manifold and
$h$ is linear, $\fb_\eps\big(\varphi^k\big)$ is bounded. (However, in
both of these examples the system is not strictly convex.)  For
convex/concave toric domains, the situation might be simpler than for
symplectic toric manifolds, but the answer is still unknown. For
strictly convex/concave domains, we expect the upper bound from Theorem
\ref{thm:toric-sympl} to be sharp, i.e., $\fb_\eps(s)$ to grow as a
polynomial of degree $n$. A related question is that of the behavior
of the shortest bar $\beta_{\min}$: For instance, does $\beta_{\min}$
remain bounded away from 0 as $s\to\infty$ or $k\to\infty$ when
$\p\Omega$ or $h$ is strictly convex? Even for $n=1$ in Theorem
\ref{thm:toric-sympl}, i.e., when $h$ is a convex/concave
function of the height on $\CP^1=S^2\subset\R^3$, both questions are
quite non-trivial because of the affect of recapping. We will come
back to some of these questions in Problems \ref{prob:integrable} and
\ref{prob:pattern}.

\begin{Remark}[Volume growth]
  \label{rmk:vol-int}
  There is a different and slightly more general approach to the
  proofs of Theorems \ref{thm:toric-sympl} and \ref{thm:toric-Reeb}.
  By Proposition \ref{prop:fb-vol} and its Reeb analogues the growth
  of $\fb_\eps(s)$ and $\fb_\eps\big(\varphi^k\big)$ is bounded from
  above by the growth of the volume of the graph of $\varphi^s_\alpha$
  or $\varphi^k$. For completely integrable systems with toric
  singularities, it should not be hard to show that the volume grows
  polynomially with an upper bound as in the right hand sides of
  \eqref{eq:toric-sympl} and \eqref{eq:toric-Reeb}. The advantage of
  the approach we have chosen here is that it gives a ``hands-on''
  upper bound on the number of generators and also more readily
  extends to the non-smooth setting; see \cite{BG}. However, the
  volume growth upper bounds can also be used in the relative setting.
\end{Remark}




\section{Open problems}
\label{sec:problems}
In this section, we briefly touch upon several open problems concerning
barcode entropy, barcode growth, and related questions. A few
of these problems have already been mentioned or even discussed in
the previous sections. The problems vary in difficulty, significance,
and flavor; some are very specific, while others are open-ended. Whenever
possible, we provide references and brief comments.

\medskip
\noindent\emph{1. Recovering topological entropy.}
The variational principle for the entropy, \cite[Thm.\ 4.5.3]{KH}, asserts
that $\htop(\varphi)=\sup_\mu\hm_{\mu}(\varphi)$ for every homeomorphism
$\varphi$ of a compact metric space, where $\hm_{\mu}$ stands for the
metric entropy with respect to a Borel invariant measure $\mu$ and the
supremum is taken over all such (ergodic) measures. It is tempting to
think of $\hbar(\varphi; L, L')$ as an analogue of metric entropy.

\begin{Problem}[Variational Principle for Barcode Entropy]
  \label{prob:var-prin}
  Assume that $M$ admits a monotone Lagrangian submanifold as in Section
  \ref{sec:conv}.  Prove (or disprove) that
  \begin{equation}
    \label{eq:var-prin}
  \htop(\varphi)=\sup_{L,L'}\hbar(\varphi; L,L')
\end{equation}
for any Hamiltonian diffeomorphism $\varphi\colon M\to M$.
\end{Problem}
The question also has an analogue for Reeb flows with the relative
barcode entropy defined in \cite{Fe1,Fe2}. A more restrictive form of
the question is with $L=L'$. In dimensions two and three,
\eqref{eq:var-prin} holds. This follows from \cite{CGG:Entropy,Me} in the
Hamiltonian case and \cite{Fe1,Fe2} in the Reeb case, and the fact that in
low dimensions all topological entropy comes from hyperbolic
invariant sets; \cite{Ka,LY,LS}. If/when \eqref{eq:var-prin} holds, is
the supremum attained?

A consequence of \eqref{eq:var-prin} would be that $\htop(\varphi)=0$
for a Hamiltonian pseudo-rotation $\varphi$ of $\CP^n$. Indeed, it is
shown in \cite[Prop.\ 2.17 and Cor.\ 2.18]{CGG:Growth} that
$\hbar(\varphi; L,L')=0$ for all $(L,L')$ in this case. In dimension
two, this follows directly from \cite{Ka}. We refer the reader to
\cite{CGG:Growth} for a further discussion.

\medskip
\noindent\emph{2. Positive barcode entropy.}
At the time of writing, Theorem \ref{thm:B} and its relative counterparts are
the only known sources of positive barcode entropy. This and the
problem discussed next are the two main impediments to further
understanding barcode entropy.
\begin{Problem}
  \label{prob:ThmB}
  Find one instance when a positive lower bound on $\hbar$ does not
  come from the topological entropy of a hyperbolic set.
\end{Problem}
For example, one can aim at showing that
$\hm_{\mu}(\varphi)\leq \hbar(\varphi; L,L')$ for a measure $\mu$ and
$(L,L')$ meeting some additional requirements not related to true
hyperbolicity. Depending on the specifics, this could be crucial to
solving Problem \ref{prob:var-prin}.  An indication, admittedly
indirect, that this might be possible is a disconnect between the
relative version of Theorem \ref{thm:B} in \cite[Thm.\ 1]{Me} and the
more restrictive Theorem \ref{thm:B} stated here. The latter does not
directly follow from the former: when we apply the graph construction
passing from $\varphi$ to $\id\times \varphi$ all hyperbolicity is
lost. Yet, the ultimate lower bound still holds.

\medskip
\noindent\emph{3. Understanding $\hbar\leq \htop$.} The proof of Theorems
\ref{thm:A} and \ref{thm:A'} rely on Yomdin's theorem, \cite{Yo},
which is extremely difficult (and deep) and does not provide dynamics
information directly connected with the definition of topological
entropy. No dynamics counterparts of the definition such as
$\eps-d_k$-separated sets or covering numbers are visible from the barcode
entropy.
\begin{Problem}
  \label{prob:ThmA}
  Give a direct proof of Theorem \ref{thm:A}, not relying on Yomdin's
  theorem even implicitly.
\end{Problem}

In \cite{CGG:Metric}, for a pair $(L,L')$ we construct a class of
invariant probability measures $\mu$ such that
$\hbar(\varphi; L, L')\leq \hm_{\mu}(\varphi)$ for some, but probably
not all, of these measures. This generalizes Theorem
\ref{thm:A}. However, the proof still uses Yomdin's
theorem. Furthermore, this class of measures is quite large and we do
not know how to narrow it down to guarantee the inequality.

\medskip
\noindent\emph{4. Strict inequality in Theorem \ref{thm:A}.} As has
already been mentioned, the combined inequality in \eqref{eq:A} for
Hamiltonian diffeomorphisms can be strict. Namely, an example of a
Hamiltonian diffeomorphism $\varphi\colon M\to M$, where
$\dim M \geq 6$, such that $\hbar(\varphi)=0$ but $\htop(\varphi)>0$
is constructed in \cite{Ci}. When $2n\geq 8$, $\varphi$ can be
autonomous.
\begin{Problem}
  \label{prob:Reeb>}
  Construct a Liouville domain such that
  $\hbar(\varphi)<\htop(\varphi)$, e.g.,
  $\hbar(\varphi)=0<\htop(\varphi)$, for the Reeb flow on the boundary
  $M$.
\end{Problem}
Here, by Theorem \ref{thm:C}, we must of course have $\dim M\geq
5$. Note that in the example the symplectic manifold is
irrational, i.e., $[\omega]$ is irrational. Hence, one cannot just
pre-quantize the flow in the autonomous case. (Also, disk bundle
fillings are not Liouville although they could be aspherical.) A
related question is if one can have $\hbar(\varphi)<\htop(\varphi)$
for Hamiltonian diffeomorphisms when $[\omega]$ is rational. (This
question appears to be connected to the long-standing open problem of
constructing higher-dimensional counterexamples to the
$C^\infty$-closing lemma on rational symplectic manifolds;
\cite{He1,He2}.) Another direction where the example from \cite{Ci} can
possibly be refined is whether one can have the strict inequality
$\hbar(\varphi)<\hvol(\varphi)$. This question also makes sense in the
relative setting of Theorem \ref{thm:A'}, which might be simpler:
$\hbar(\varphi; L)<\hvol(\varphi, L)$. (Note that here we take
$L=L'$. Each inequality in
$\hbar(\varphi; L,L')\leq\hvol(\varphi; L)\leq \htop(\varphi)$ can
obviously be strict for a suitable choice of $L$ and $L'$ and
$\varphi$.)  In relation to Problem \ref{prob:var-prin}, we are not
aware of any connection between the absolute and relative barcode
entropy beyond low dimensions.

\medskip
\noindent\emph{5. Continuity properties of barcode entropy.}
The topological entropy $\htop(\varphi)$, as a function of $\varphi$, is
not continuous in general, but it does have some continuity features
albeit rather subtle. In $C^\infty$-topology, the map $\varphi\mapsto
\htop(\varphi)$ is upper semi-continuous, and so is the map $\mu\mapsto
\hm_{\mu}(\varphi)$; \cite{Ne,Yo}. Then, as follows from the results in
\cite{Ka}, the former map is continuous in dimension two. (These
results have their counterparts for flows.)
\begin{Problem}
  \label{prob:continuity}
  Investigate continuity properties of barcode entropy.
\end{Problem}
It is tempting to conjecture that barcode entropy is lower
semi-continuous with respect to the Hofer or spectral distance. There
is some indirect evidence supporting this conjecture. Topological
entropy is lower semi-continuous with respect to the Hofer metric for
Hamiltonian diffeomorphisms of surfaces, \cite{AM:braid}, and also in
a certain sense robust for Reeb flows in dimension three,
\cite{ADMM,ADMP}. By Theorem \ref{thm:C}, the same is true for barcode
entropy. Furthermore, $\hbar(\varphi; L,L')$ is lower semi-continuous
in $L$ and $L'$ with respect to the Hofer or spectral metric by Part
\ref{BE5} of Proposition \ref{prop:hbar-prop} and so is Lagrangian
volume at least in some cases; see Section \ref{sec:Lagr-vol}.

\medskip
\noindent\emph{6. Lower semi-continuity of Lagrangian volume.}
In Section \ref{sec:Lagr-vol}, following \cite{CGG:Vol}, we considered
the volume function $\vol\colon L\mapsto \vol(L)$ on the class $\CL$
of Hamiltonian isotopic closed monotone Lagrangian submanifolds $L$ of
a symplectic manifold $M$. We showed that $\vol$ is lower
semi-continuous in two disparate settings: the classical setting of $M$
having a large group of symmetries (Theorem \ref{thm:symmetry}) and when
the manifolds in $\CL$ are monotone tori (Theorem \ref{thm:torus}).
\begin{Problem}
  \label{prob:vol}
Prove that $\vol\CL\to \R$ is lower semi-continuous with respect
to the Hofer or spectral distance without additional restrictions on
the topology of $L$.
\end{Problem}
In other words, the goal is to remove the condition that $L$ is torus
in Theorem \ref{thm:torus}. One approach is to continue using the
framework from Section \ref{sec:Crofton2} and attempt to construct a
Lagrangian tomograph satisfying \eqref{eq:LT1} and \eqref{eq:LT2}.

\medskip
\noindent\emph{7. Basic properties.}
General properties of barcode entropy, along the lines of Section
\ref{sec:barcode_entropy-prop} are poorly understood. Here we just
mention two open questions lying close to the surface.
\begin{Problem}
  \label{prob:basic}
  \
  \begin{itemize}
    \item[\reflb{BP1}{\rm{(a)}}] Can the inequalities in Parts
      \ref{BE1} and \ref{BE2} of Proposition \ref{prop:hbar-prop}
      (sub-homogeneity and sub-additivity) be strict, in contrast with
      their counterparts for topological entropy?
  \item[\reflb{BP2}{\rm{(b)}}] Prove a version of Proposition
    \ref{prop:hbar-prop} for Reeb flows.
  \end{itemize}
\end{Problem}
It is not clear to us what to expect in Part \ref{BP1}. There are
several definitions of topological entropy interpreting the growth of
dynamical complexity in related but different ways: spanning sets,
separated sets, covers. The interplay between these definitions is
central to proving that \ref{BE1} and \ref{BE2} turn into equality for
topological entropy. Nothing like that is available for barcode
entropy at the moment. The problem is compounded by that at the small
scale no useful interaction between filtered Floer homology of
different iterations is known.

Regarding \ref{BP2}, some adjustments to the statement would obviously
be necessary. Focusing on the absolute case, note that \ref{BE2} would
involve the barcode entropy of the product of two Liouville
domains. The boundary of the product is a contact manifold with
corners and the Reeb flow is not everywhere defined. However, the
barcode function and barcode entropy are; cf.\ \cite[Sec.\ 2.3
and 4.2.4]{BG}. In \ref{BE5} from Proposition \ref{prop:hbar-prop},
the Hofer or spectral metric would have to be replaced by a suitable
variant of the symplectic Banach--Mazur distance;
\cite{PRSZ,SZ-SBM,Us-SBM}.  (The definition of relative barcode entropy
in the Reeb case, omitted in these notes, can be found in
\cite{Fe1,Fe2}.)

\medskip
\noindent\emph{8. Completely integrable Reeb flows.} Several questions
about barcode growth for completely integrable Reeb flows are
discussed in Section \ref{sec:discussion}. Here we just sum up two
of them and refer the reader to that section for further details and
references.
\begin{Problem}
  \label{prob:integrable}
  \
  \begin{itemize}
  \item[\reflb{I1}{\rm{(a)}}] Construct a completely integrable Reeb
    flow on the standard contact $S^{2n-1}$ with positive barcode or
    at least topological entropy for $2n-1\geq 5$ or even when $2n-1=3$.
  \item[\reflb{I2}{\rm{(b)}}] Prove that $\fb_\eps(s)$ grows
    subexponentially or even polynomially for a completely integrable
    Reeb flow or a Hamiltonian diffeomorphism with sufficiently
    non-degenerate singularities.
  \end{itemize}
\end{Problem}

\medskip
\noindent\emph{9. Actual barcode behavior and statistics or barcodes.} To the best of
our knowledge, virtually nothing is known about the behavior of the
barcode or the spectral norm as a function of time in non-trivial
systems or generically or even in the completely integrable case. The
problem is completely open even in dimension two in both relative and
absolute settings.  For instance, let $\varphi$ be a Hamiltonian
diffeomorphism of a closed surface $M$. 
\begin{Problem}
  \label{prob:pattern}
  Is there any pattern to $\CB\big(\varphi^k\big)$, e.g., the mean and
  the dispersion, and $\gamma\big(\varphi^k\big)$ as functions of time
  $k$ for a $C^\infty$- or $C^r$-generic $\varphi$? Is
  $\limsup \gamma\big(\varphi^k\big)=\infty$ generically when $M$ has
  positive genus; cf.\ Section \ref{sec:sectral+inv}?  What can be
  said about the behavior of the shortest bar? Going in the opposite
  direction, what if $M=S^2$ and $\varphi$ is generated by a strictly
  convex or concave function of the height; cf.\ Section
  \ref{sec:Toric-Ham}. Is then
  $\liminf \gamma\big(\varphi^k\big)=0$?
\end{Problem}

Similar questions can be asked in higher dimensions or for Reeb flows
or in the relative setting. In fact, the case of a toric integrable Reeb
flow on $S^3$ (cf.\ Section \ref{sec:Toric-Reeb}) is probably simpler than
the question for a convex/concave Hamiltonian on $S^2$. Some of these
questions are already touched upon in Section \ref{sec:discussion}. We
also note that at least in the integrable case one can try to look for
a pattern by experimenting numerically.



\begin{acknowledgement}
The authors thank Olga Bernardi, Anna Florio, and Alfonso Sorrentino for organizing this summer school, and the CIME foundation for their support.
\end{acknowledgement}

\ethics{Competing Interests}{The work is partially supported by the NSF grants DMS-2304207
  (BG) and DMS-2304206 (VG), the Simons Foundation grants 855299 (BG)
  and MP-TSM-00002529 (VG), the ERC Starting Grant 851701 via a
  postdoctoral fellowship (E\c{C}),
  and the ANR grant CoSyDy ANR-CE40-0014 (MM).\newline 
The authors have no conflicts of interest to declare that are relevant to the content of this chapter.}

\eject





\begin{thebibliography}{CdSGO16}

\bibitem[A$^2$S$^2$]{AASS}  
  A. Abbondandolo, M.R.R. Alves, M. Sa\u{g}lam, F. Schlenk, Entropy
  collapse versus entropy rigidity for Reeb and Finsler flows,
  \emph{Selecta Math.\ (NS)} \textbf{29} (2023), Article No.\ 67.
  
  
  \bibitem[ABE]{ABE}
A. Abbondandolo, G. Benedetti, O. Edtmair, Symplectic capacities of domains close to the ball and
Banach--Mazur geodesics in the space of contact forms, \emph{Duke Math.\ J.}, 
\textbf{174} (2025), 1567--1646.

\bibitem[Abu]{Abu} M. Abouzaid, Symplectic cohomology and Viterbo's
  theorem, in \emph{Free loop spaces in geometry and topology},
  271--485. IRMA Lect.\ Math.\ Theor.\ Phys., 24 European Mathematical
  Society (EMS), Zürich, 2015.


\bibitem[Ah]{Ah}
J. Ahn, Barcode entropy and relative symplectic cohomology, Preprint
arXiv:2601.15606.  

\bibitem[All]{Al} S. Allais, On periodic points of Hamiltonian
  diffeomorphisms of $\CP^d$ via generating functions,
  \emph{J. Symplectic Geom.}, \textbf{20} (2022), 1--48.

\bibitem[APF98]{APF98} J.C. \'Alvarez Paiva, E. Fernandes, Crofton
  formulas in projective Finsler spaces, \emph{Electron.\ Res.\
    Announc.\ Amer.\ Math.\ Soc.}, \textbf{4} (1998), 91--100.

\bibitem[APF07]{APF07} J.C. \'Alvarez Paiva, E. Fernandes, Gelfand
  transforms and Crofton formulas, \emph{Selecta Math.\, (N.S.)},
  \textbf{13} (2007), 369--390. 

\bibitem[AM]{AM}
  M.R.R. Alves, M. Meiwes, Dynamically exotic contact spheres in
  dimensions $\geq 7$, \emph{Comment.\ Math.\ Helv.}, \textbf{94}
  (2019), 569--622. 
  
\bibitem[AM]{AM21} M.R.R. Alves, M. Meiwes, Braid stability and the
  Hofer metric, Preprint arXiv:2112.11351.

  
\bibitem[Al16a]{Al:Anosov} M.R.R. Alves, Positive topological entropy
  for Reeb flows on 3-dimensional Anosov contact manifolds,
  \emph{J. Mod.\ Dyn.}, \textbf{10} (2016), 497--509.

\bibitem[Al16b]{Al:Cyl} M.R.R. Alves,
 Cylindrical contact homology and topological entropy,
 \emph{Geom.\ Topol.}, \textbf{20} (2016), 3519--3569. 

 \bibitem[Al19]{Al:Leg} M.R.R. Alves,
 Legendrian contact homology and topological entropy
\emph{J. Topol.\ Anal.}, \textbf{11} (2019), 53–108. 

\bibitem[ACH]{ACH}
  M.R.R. Alves, V. Colin, K. Honda, Topological entropy for Reeb
  vector fields in dimension three via open book decompositions,
  \emph{J.\ \'Ec.\ polytech.\ Math.}, \textbf{6} (2019), 119--148.

\bibitem[ADM$^2$]{ADMM} M.R.R. Alves, L. Dahinden, M. Meiwes,
  L. Merlin $C^0$-Robustness of topological entropy for geodesic
  flows, \emph{J. Fixed Point Theory Appl.}, \textbf{22} (2022), Paper
  no.\ 42, 43 pp.

\bibitem[ADMP]{ADMP} M.R.R. Alves, L. Dahinden, M. Meiwes,
  A. Pirnapasov. $C^0$-stability of topological en- tropy for Reeb
  flows in dimension 3, Preprint arXiv:2311.12001.

\bibitem[AM19]{AM} M.R.R. Alves, M. Meiwes, Dynamically exotic contact
  spheres in dimensions $\geq 7$, \emph{Comment.\ Math.\ Helv.},
  \textbf{94} (2019), 569--622. 

\bibitem[AM24]{AM:braid} M.R.R. Alves, M. Meiwes, Braid stability and the
Hofer metric, \emph{Ann.\ H. Lebesgue},
\textbf{7} (2024), 521--581. 

 \bibitem[AP]{AP}
  M.R.R. Alves, A. Pirnapasov, Reeb orbits that force topological
  entropy, \emph{Ergodic Theory Dynam.\ Systems}, \textbf{42} (2022),
  3025--3068.

\bibitem[ABC]{ABC} G. Ambrosioni, P. Biran, O. Cornea, Approximability
  for Lagrangian submanifolds, Preprint arXiv:2601.12506.

  \bibitem[AK]{AK}
  D.V. Anosov, A.B. Katok, New examples in smooth ergodic
  theory. Ergodic diffeomorphisms, \emph{Trudy Moskov.\ Mat.\ Ob\v s\v c.}, \textbf{23}
    (1970), 3--36.

\bibitem[Arn]{Arn} M.-C. Arnaud, Type des points fixes des
  diff\'eomorphismes symplectiques de $\T^n\times \R^n$,
  \emph{M\'em. Soc.\ Math.\ France (N.S.)}, (1992), no.\ 48, 63 pp.

  
  \bibitem[Ar89]{Ar} V.I. Arnold, \emph{Mathematical Methods of
      Classical Mechanics}, Grad.\ Texts in Math., 60,
    Springer-Verlag, New York, 1989.

\bibitem[Ar90a]{Ar1} V.I. Arnold, Dynamics of complexity of
  intersections, \emph{Bol.\ Soc.\ Brasil.\ Mat.\ (N.S.)}, \textbf{21}
  (1990), 1--10. 

\bibitem[Ar90b]{Ar2} V.I. Arnold, Dynamics of intersections,
  \emph{Analysis, et cetera}, 77--84, Academic Press, Boston, MA,
  1990.

  

\bibitem[As]{As} M. Asaoka, Abundance of fast growth of the number of
  periodic points in 2-dimensional area-preserving dynamics,
  \emph{Comm.\ Math.\ Phys.}, \textbf{356} (2017), 1--17.

\bibitem[ACW]{ACW} A. Avila, S. Crovisier, A. Wilkinson, $C^1$ density
  of stable ergodicity, \emph{Adv.\ Math.}, \textbf{379} (2021), Paper
  No.\ 107496, 68 pp.
  
  

\bibitem[BL]{BL} H. Bae, S. Lee, A comparison of categorical and
  topological entropies on Weinstein manifolds, Preprint
  arXiv:2208.14597.

\bibitem[BG]{BG} E. Barut, V.L. Ginzburg, Barcode growth for
  toric-integrable Hamiltonian systems, Preprint arXiv:2503.08922; to
  appear in \emph{Israel J.\ Math.}.
 

\bibitem[BiC]{BC} P. Biran, O. Cornea, Rigidity and unriddling for
  Lagrangian submanifolds, \emph{Geom.\ Topol.}, \textbf{13} (2009),
  2881--2989. 

  
  
\bibitem[BT]{BT} A.V. Bolsinov, I.A. Taimanov, Integrable geodesic
  flows with positive topological entropy, \emph{Invent.\ Math.},
\textbf{140} (2000), 639--650.

\bibitem[Bo]{Bo} F. Bourgeois, A Morse--Bott approach to contact
  homology, In \emph{Symplectic and contact topology: interactions and
    perspectives (Toronto, ON/Montreal, QC, 2001)}, Fields Inst.\
  Commum., AMS, \textbf{35}, (2003), 55--77.

  
\bibitem[BO09a]{BO0} F. Bourgeois, A. Oancea, Symplectic homology,
  autonomous Hamiltonians, and Morse--Bott moduli spaces, \emph{Duke
    Math.\ J.}, \textbf{146} (2009), 71--174.


\bibitem[BO09b]{BO}
F. Bourgeois, A. Oancea,
An exact sequence for contact- and symplectic homology, 
\emph{Invent.\ Math.},  \textbf{175}  (2009),   611--680.

  
\bibitem[Br]{Br} B. Bramham, Pseudo-rotations with sufficiently
  Liouvillean rotation number are $C^0$-rigid, \emph{Invent.\ Math.},
  \textbf{199} (2015), 561--580. 

\bibitem[BHPW]{BHPW}
P. Bubenik, M. Hull, D. Patel, B. Whittle, 
Persistent homology detects curvature,
\emph{Inverse Problems}, \textbf{36} (2020), 025008, 23 pp. 

  
\bibitem[BV]{BV} P. Bubenik, T.Vergili, Topological spaces of
  persistence modules and their properties, \emph{J. Appl.\ Comput.\
    Topol.}, \textbf{2} (2018), 233--269. 

  

  \bibitem[BP$^3$S$^2$]{BP3S2}
L. Buhovsky, J. Payette, I. Polterovich, L. Polterovich, E. Shelukhin, V. Stojisavljevi\'c,
Coarse nodal count and topological persistence, \emph{J. Eur.\ Math.\ Soc.\
    (JEMS)}, (2024), doi:10.4171/JEMS/1521.

\bibitem[BI]{BuIv} D. Burago, S. Ivanov, On asymptotic volume of tori,
  \emph{Geom.\ Funct.\ Anal.}, \textbf{5} (1995), 800--808.

  
\bibitem[BZ]{BZ} Yu.D. Burago, V.A. Zalgaller, \emph{Geometric
    Inequalities}, Grundlehren Math.\ Wiss.\ 285, Springer, Berlin
  (1988).

\bibitem[CdS]{CdS}A. Cannas da Silva, \emph{Lectures on Symplectic
    Geometry}, Lecture Notes in Math., 1764 Springer-Verlag, Berlin,
  2001. 

  
\bibitem[CZCG]{CZCG} G. Carlsson, A. Zomorodian, A. Collins,
  L. Guibas, Persistence barcodes for shapes, \emph{Int.\ J. Shape
    Model.}, \textbf{11} (2005), 149--187.

\bibitem [Ce]{Ce} L. Cesari, \emph{Surface Area}, Princeton University
  Press, Princeton, NJ, 1956.
 
  
\bibitem[Ch]{Ch} Yu.V. Chekanov, Lagrangian intersections, symplectic
  energy, and areas of holomorphic curves, \emph{Duke Math.\ J.},
  \textbf{95} (1998), 213--226. 

  
\bibitem[CFH]{CFH} K. Cieliebak, A. Floer, H. Hofer, Symplectic
  homology II: A general construction, \emph{Math.\ Z.},
  \textbf{218} (1995), 103--122.

\bibitem[CFHW]{CFHW} K. Cieliebak, A. Floer, H. Hofer, K. Wysocki,
  Applications of symplectic homology II: Stability of the action
  spectrum, \emph{Math.\ Z.}, \textbf{223} (1996), 27--45.

\bibitem[CO]{CO} K. Cieliebak, A. Oancea, Symplectic homology and the
  Eilenberg–Steenrod axioms, \emph{Algebr.\ Geom.\ Topol.}, \textbf{18} (2018),
    1953--2130.

  \bibitem[\c Ci]{Ci} E. \c Cineli, A generalized
    pseudo-rotation with positive topological entropy,
\emph{Bull.\ Lond.\ Math.\ Soc.}, 2025, doi: 10.1112/blms.70021.

\bibitem[\c CG$^2$22]{CGG:HZ} E. \c Cineli, V.L. Ginzburg, B.Z. G\"urel,
  Another look at the Hofer--Zehnder conjecture, \emph{J. Fixed Point
    Theory Appl., Claude Viterbo’s 60th Birthday Festschrift.},
  \textbf{24} (2022), doi:10.1007/s11784-022-00937-w.

\bibitem[\c CG$^2$24a]{CGG:Entropy} E. \c Cineli, V.L. Ginzburg,
  B.Z. G\"urel, Topological entropy of Hamiltonian diffeomorphisms: a
  persistence homology and Floer theory perspective, \emph{Math.\ Z.},
  \textbf{308} (2024), doi:10.1007/s00209-024-03627-0.
    
\bibitem[\c CG$^2$24b]{CGG:Growth} E. \c Cineli, V.L. Ginzburg,
  B.Z. G\"urel, On the growth of the Floer barcode, \emph{J. Mod.\
    Dyn.}, \textbf{20} (2024), 275--298.

\bibitem[\c CG$^2$24c]{CGG:Spectral} E. \c Cineli, V.L. Ginzburg,
  B.Z. G\"urel, On the generic behavior of the spectral norm,
  \emph{Pacific J. Math.}, \textbf{328} (2024), 119--135,

\bibitem[\c CG$^2$24d]{CGG:HZ} E. \c Cineli, V.L. Ginzburg,
  B.Z. G\"urel, Closed orbits of dynamically convex Reeb flows:
  Towards the HZ- and multiplicity conjectures, Preprint arXiv:2410.13093.


\bibitem[\c CG$^2$25a]{CGG:Vol} E. \c Cineli, V.L. Ginzburg,
  B.Z. G\"urel, Lower semi-continuity of Lagrangian volume,
  \emph{Israel J.\ Math.}, \textbf{267} (2025), 253-274. 

  \bibitem[\c CG$^2$25b]{CGG:Metric} E. \c Cineli, V.L. Ginzburg,
  B.Z. G\"urel, From barcode entropy to metric entropy, Preprint
  arXiv:2507.13215; to appear in \emph{J. Fixed Point Theory Appl.}

\bibitem[\c CG$^2$M25]{CGGM:Reeb} E. \c Cineli, V.L. Ginzburg,
  B.Z. G\"urel, M. Mazzucchelli, On the barcode entropy of Reeb flows,
  \emph{Selecta Math.\ (NS)}, \textbf{31} (2025), Article No.\ 64, doi:
  10.1007/s00029-025-01062-5.

\bibitem[\c CG$^2$M26]{CGGM:hyperbolic} E. \c Cineli, V.L. Ginzburg,
  B.Z. G\"urel, M. Mazzucchelli  
Invariant sets and hyperbolic closed Reeb orbits, 
\emph{Adv.\ Math.}, \textbf{484} (2026), 110709.
  

  \bibitem[\c CS]{CS} E. \c Cineli, S. Seyfaddini,
The strong closing lemma and Hamiltonian pseudo-rotations,
\emph{J. Mod.\ Dyn.}, \textbf{20} (2024), 299--318.

\bibitem[CdSGO]{CdSGO16} F. Chazal, V. de Silva, M. Glisse,
  S. Oudot, The structure and stability of persistence modules,
  Springer Briefs Math.  \emph{Springer, [Cham]}, 2016. x+120 pp.

\bibitem[CSEHM]{CSEHM}
  D. Cohen-Steiner, H. Edelsbrunner, J. Harer, Y. Mileyko,
Lipschitz functions have $L_p$-stable persistence, 
\emph{Found.\ Comput.\ Math.}, \textrm{10} (2010), 127--139.  
  
\bibitem[CDHR]{CDHR}  
V. Colin, P. Dehornoy, U. Hryniewicz, A. Rechtman, Generic
properties of 3-dimensional Reeb flows: Birkhoff sections and
entropy, Preprint arXiv:2202.01506.

\bibitem[EP]{EP} M. Entov, L. Polterovich, Calabi quasimorphism and
  quantum homology, \emph{Int.\ Math.\ Res.\ Not.}, 2003,
  1635--1676. 


\bibitem[CKMS]{CKMS}
G. Contreras, G. Knieper, M. Mazzucchelli, B.H. Schulz, Surfaces of
  section for geodesic flows of closed surfaces,
  Preprint arXiv:2204.11977.

\bibitem[CB]{CB} W. Crawley-Boevey, Decomposition of pointwise
  finite-dimensional persistence modules, \emph{J. Algebra Appl.},
  \textbf{14} (2015), no.\ 5, 1550066, 8 pp.

\bibitem[CG]{CG} D. Cristofaro-Gardiner, Symplectic embeddings from
  concave toric domains into convex ones,
  \emph{J. Differential Geom.}, \textbf{112} (2019), 199--232.
  
  
\bibitem[De]{De} T. Delzant, Hamiltoniens p\'eriodiques et images
  convexes de l'application moment, \emph{Soc.\ Math.\ France},
  \textbf{116} (1988), 315--339. 

  \bibitem[DHK$^2$]{DHKK} G. Dimitrov, F. Haiden, L. Katzarkov,
  M. Kontsevich, Dynamical systems and categories, The influence of
  Solomon Lefschetz in geometry and topology, \emph{Contemp.\ Math.},
  vol.\ 621, Amer. Math. Soc., Providence, RI, 2014, pp. 133--170.  
  
\bibitem[Di]{Di} E.I. Dinaburg
characterizations of dynamical systems, \emph{Izv.\ Akad.\ Nauk SSSR
Ser.\ Mat.}, \textbf{35} (1971), 324--366. 

\bibitem[El]{El} L.H. Eliasson, Normal forms for Hamiltonian systems
  with Poisson commuting integrals -- elliptic case, \emph{Comment.\
    Math.\ Helv.}, \textbf{65} (1990), 4--35. 

\bibitem[EP]{EP} M. Entov, L. Polterovich, Calabi quasimorphism and
  quantum homology, \emph{Int.\ Math.\ Res.\ Not.}, 2003,
  1635--1676. 

\bibitem[FK]{FK} B. Fayad, A. Katok, Constructions in elliptic
  dynamics, \emph{Ergodic Theory Dynam.\ Systems}, \textbf{24} (2004),
  1477--1520.

\bibitem[Fe52]{Fe52} H. Federer, Measure and area, \emph{Bull.\ Amer.\
    Math.\ Soc.}, \textbf{58} (1952), 306--378. 

  
\bibitem[FeLS]{FLS} E. Fender, S. Lee, B. Sohn, Barcode entropy for
  Reeb flows on contact manifolds with exact Liouville fillings,
  \emph{Comm.\ Contemp.\ Math.}, 2025, doi.org/10.1142/S0219199725500440.

\bibitem[Fe24]{Fe1} R. Fernandes, 
Barcode entropy and wrapped Floer homology, Preprint arXiv:2410.05528.

\bibitem[Fe25]{Fe2} R. Fernandes, Wrapped Floer homology and
hyperbolic sets, Preprint arXiv:2501.06654.

\bibitem[Fi]{Fi} J.W. Fish, Target-local Gromov compactness,
  \emph{Geom.\ Topol.}, \textbf{15} (2011), 765--826.


\bibitem[FiHa]{FH}
  T. Fisher, B. Hasselblatt, \emph{Hyperbolic Flows}, Zurich Lectures
  in Advanced Mathematics, European Mathematical Society, Berlin,
  2019.

\bibitem[Fr96]{Fr96} J. Franks, Area preserving homeomorphisms of open
  surfaces of genus zero, \emph{New York Jour.\ of Math.}, \textbf{2}
  (1996), 1--19.

\bibitem[Fr99]{Fr99} J. Franks, The Conley index and non-existence of
  minimal homeomorphisms. Proceedings of the Conference on
  Probability, Ergodic Theory, and Analysis (Evanston, IL,
  1997). \emph{Illinois J. Math.}, \textbf{43} (1999),
  457--464. 
 

  
\bibitem[FrHa]{FrHa} J. Franks, M. Handel, Periodic points of Hamiltonian
  surface diffeomorphisms, \emph{Geom.\ Topol.}, \textbf{7} (2003),
  713--756.


\bibitem[FM]{FM} J. Franks, M. Misiurewicz, Topological methods in
  dynamics. \emph{Handbook of dynamical systems}, Vol.\ 1A, 547–598,
  North-Holland, Amsterdam, 2002.
  
  
\bibitem[FGS]{FGS}  
  U. Frauenfelder, V.L. Ginzburg, F. Schlenk,
  Energy capacity inequalities via an action selector, \emph{Contemp.\
    Math.}, \textbf{387}, Israil Math.\ Conf.\ Proc., American
  Mathematical Society, RI, 2005, 129--152.
  
\bibitem[FrLS]{FrLS}
U. Frauenfelder, C. Labrousse, F. Schlenk, 
Slow volume growth for Reeb flows on spherizations and contact
Bott--Samelson theorems,
\emph{J. Topol.\ Anal.}, \textbf{7} (2015), 407--451.  


\bibitem[FS06]{FS} U. Frauenfelder, F. Schlenk, Fiberwise volume growth
  via Lagrangian intersections, \emph{J. Symplectic Geom.}, \textbf{4}
  (2006), 117--148.  

\bibitem[FS07]{FS:convex}  
U. Frauenfelder, F. Schlenk, Hamiltonian dynamics on convex symplectic
manifolds, \emph{Israel
    J. Math.}, \textbf{159} (2007), 1--56. 

\bibitem[FO${^3}$]{FOOO} K. Fukaya, Y.-G. Oh, H. Ohta, K. Ono,
  \emph{Lagrangian Intersection Floer Theory: Anomaly and Obstruction,
    I and II}, AMS/IP Studies in Advanced Math., vol.\ 46, Amer.\
  Math.\ Soc.\ and International Press, 2009.

  
\bibitem[GeSm]{GS} I.M.  Gelfand, M.M. Smirnov, Lagrangians satisfying
  Crofton formulas, Radon transforms, and nonlocal differentials,
  \emph{Adv.\ Math.}, \textbf{109} (1994), 188--227.  

\bibitem[G$^2$09]{GG:gaps} V.L. Ginzburg, B.Z. G\"urel, Action and index
  spectra and periodic orbits in Hamiltonian dynamics, \emph{Geom.\
    Topol.}, \textbf{13} (2009), 2745--2805.

  
\bibitem[G$^2$14]{GG:hyperbolic} V.L. Ginzburg, B.Z. G\"urel, Hyperbolic
  fixed points and periodic orbits of Hamiltonian diffeomorphisms,
  \emph{Duke Math.\ J.}, \textbf{163} (2014), 565--590.

\bibitem[G$^2$15]{GG:survey} V.L. Ginzburg, B.Z. G\"urel, The Conley
  conjecture and beyond, \emph{Arnold Math.\ J.}, \textbf{1} (2015),
  299--337.
  
\bibitem[G$^2$18]{GG:PR} V.L. Ginzburg, B.Z. G\"urel, Hamiltonian
  pseudo-rotations of projective spaces, \emph{Invent.\ Math.},
  \textbf{214} (2018), 1081--1130.

\bibitem[G$^2$20]{GG:LS} V.L. Ginzburg, B.Z. G\"urel,
  Lusternik--Schnirelmann theory and closed Reeb orbits, \emph{Math.\
    Z.}, \textbf{295} (2020),
  515--582. 

\bibitem[G$^2$22]{GG:PRvsR} V.L. Ginzburg, B.Z. G\"urel, Pseudo-rotations
  vs.\ rotations, \emph{J. London Math.\ Soc.\ (2)}, \textbf{106}
  (2022), 3411--3449.
  
\bibitem[G$^2$M]{GGM} V.L. Ginzburg, B.Z. G\"urel, M. Mazzucchelli,
  Barcode entropy of geodesic flows, \emph{J. Eur.\ Math.\ Soc.\
    (JEMS)},  doi: 10.4171/JEMS/1572; First online: 13 Dec 2024.

\bibitem[GO]{GO}  
B.R. Gelbaum, J.M.H. Olmsted, \emph{Counterexamples in Analysis},
Dover, 2003.

\bibitem[GuSt]{GuSt} V. Guillemin, S. Sternberg, The moment map
  revisited, \emph{J. Differential Geom.}, \textbf{69} (2005),
  137--162. 
 

\bibitem[Gu]{Gu} J. Gutt, The positive equivariant symplectic homology
  as an invariant for some contact manifolds, \emph{J. Symplectic
    Geom.}, \textbf{15} (2017), 1019--1069. 
  
\bibitem[GH]{GH} J. Gutt, M. Hutchings, Symplectic capacities from
  positive $S^1$-equivariant symplectic homology, \emph{Algebr.\
    Geom.\ Topol.}, \textbf{18} (2018), 3537--3600. 

  

\bibitem[He91a]{He1} M.-R. Herman, Exemples de flots hamiltoniens dont
  aucune perturbation en topologie $C^\infty$ n’a d’orbites
  p\'eriodiques sur un ouvert de surfaces d’\'energies,
  \emph{C. R. Acad.\ Sci.\ Paris S\'er.\ I Math.}, \textbf{312}
  (1991), 989--994.
  
  \bibitem[He91b]{He2}
  M.-R. Herman, Diff\'erentiabilit\'e optimale et contre-exemples \`a
  la fermeture en topologie $C^\infty$ des orbites r\'ecurrentes de flots
  hamiltoniens, \emph{C. R. Acad.\ Sci.\ Paris S\'er.\ I Math.},
  \textbf{313} (1991), 49--51.
  
  
\bibitem[Ho]{Ho} H. Hofer, On the topological properties of
    symplectic maps, \emph{Proc.\ Roy.\ Soc.\ Edinburgh Sect.\ A},
    \textbf{115} (1990), 25--38. 

  \bibitem[HS]{HS} H. Hofer, D.A. Salamon, Floer homology and Novikov
    rings, in \emph{The Floer memorial volume}, (H. Hofer,
    C.H. Taubes, A. Weinstein, E. Zehnder, editors), Progr.\ Math.\
    133, Birkh\"auser, Basel (1995) 483--524.
    
\bibitem[HZ]{HZ} H. Hofer, E. Zehnder, \emph{Symplectic Invariants and
    Hamiltonian Dynamics}, Birk\"auser Verlag, Basel, 1994.


\bibitem[Hum]{Hum} V. Humili\`ere, On some completions of the space of
  Hamiltonian maps, \emph{Bull.\ Soc.\ Math.\ France}, \textbf{136}
  (2008), 373--404. 

  
\bibitem[HLS]{HLS} V. Humili\`ere, R. Leclercq, S. Seyfaddini,
    Reduction of symplectic homeomorphisms, \emph{Ann.\ Sci.\ \'Ec.\
      Norm.\ Sup\'er. (4)}, \textbf{49} (2016), 633--668. 
  

\bibitem[Hu16]{Hu:GT16} M. Hutchings, Beyond ECH capacities,
  \emph{Geom.\ Topol.}, \textbf{20} (2016), 1085--1126.  
  
\bibitem[Hu24]{Hu} M. Hutchings, Zeta functions of dynamically tame
  Liouville domains, Preprint arXiv:2402.07003.

\bibitem[JS]{JS} D. Joksimovi\'c, S. Seyfaddini, A H\"older-type
  inequality for the $C^0$ distance and Anosov--Katok pseudo-rotations,
\emph{Int.\ Math.\ Res.\ Not.\ IMRN} 2024, no.\ 8, 6303--6324.

\bibitem[Iv]{Iv} S.V. Ivanov, Gromov--Hausdorff convergence and volumes
  of manifolds, \emph{Algebra i Analiz}, \textbf{9} (1997), 
  65--83; translation in \emph{St.\ Petersburg Math.\ J.}, \textbf{9}
  (1998), 
  945--959.


\bibitem[Kal]{Kal} V. Kaloshin, An extension of the Artin--Mazur
  theorem, \emph{Ann.\ of Math.} (2), \textbf{150} (1999),
  729--741. 
 

\bibitem[Ka73]{Ka:Izv}
A.B. Katok, 
Ergodic perturbations of degenerate integrable Hamiltonian systems (Russian),
\emph{Izv.\ Akad.\ Nauk SSSR Ser.\ Mat.},   \textbf{37} (1973), 539--576.

\bibitem[Ka80]{Ka} A. Katok, Lyapunov exponents, entropy and periodic
  orbits for diffeomorphisms, \emph{Inst.\ Hautes \'Etudes Sci.\
    Publ.\ Math.}, \textbf{51} (1980), 137--173.

  \bibitem[Ka82]{Ka82}
  A. Katok, Entropy and closed geodesics, \emph{Ergodic Theory Dynam.\
    Systems}, \textbf{2} (1982), 339--365. 

  
\bibitem[KH]{KH}
  A. Katok, B. Hasselblatt, \emph{Introduction to the
    Modern Theory of Dynamical Systems.} With a supplementary chapter
  by A. Katok and Mendoza. Encyclopedia of Mathematics and its
  Applications, 54. Cambridge University Press, Cambridge, 1995.


\bibitem[KT]{KT}
A. Katok, J.P. Thouvenot, Slow entropy type invariants and smooth
realization of commuting measure-preserving transformations, \emph{Ann.\ Inst.\
H. Poincar\'e Probab.\ Statist.}, \textbf{33} (1997), 323--338.  

\bibitem[KS]{KS}
A. Kislev, E. Shelukhin, 
Bounds on spectral norms and barcodes, 
\emph{Geom.\ Topol.},  \textbf{25} (2021), 3257--3350. 

  \bibitem[KLCN]{KLCN}
A. Koropecki, P. Le Calvez, M. Nassiri, Prime ends rotation numbers and periodic points,
\emph{Duke Math.\ J.}, \textbf{164} (2015), 403--472.  

\bibitem[LMcD]{LMcD} F. Lalonde, D. McDuff, The geometry of symplectic
  energy, \emph{Ann.\ of Math.\ (2)} \textbf{141} (1995),
  349--371. 

\bibitem[LeC]{LeC} P. Le Calvez, Periodic orbits of Hamiltonian
  homeomorphisms of surfaces, \emph{Duke Math.\ J.}, \textbf{133}
  (2006), 125--184.
  

\bibitem[LCS]{LCS} P. Le Calvez, M. Sambarino, Homoclinic orbits for
  area preserving diffeomorphisms of surfaces, \emph{Ergodic Theory
  Dyn.\ Syst.}, \textbf{42} (2022), 1122--1165 .

\bibitem[LCY]{LCY} P. Le Calvez, J.-C. Yoccoz, Un th\'eor\`eme d'indice
  pour les hom\'eomorphismes du plan au voisinage d'un point fixe,
  \emph{Ann.\ of Math.\ (2)}, \textbf{146} (1997), 241--293. 


\bibitem[LRS]{LRS}
 F. Le Roux, S. Seyfaddini, The Anosov--Katok method and pseudo-rotations in symplectic
dynamics, \emph{J. Fixed Point Theory Appl.}, \textbf{24} (2022), Paper No.\ 36, 39 pp.


\bibitem[Le]{Le} M. Lesnik, The theory of the interleaving distance on
  multidimensional persistence modules, \emph{Found.\ Comput.\ Math.},
  \textbf{15} (2015), 613--650.  

\bibitem[Li]{Li} W. Li, The $\Z$-graded symplectic Floer cohomology of
  monotone Lagrangian sub-manifolds, \emph{Algebr.\ Geom.\ Topol.},
  \textbf{4} (2004), 647--684.
 
  
\bibitem[LY]{LY} Z. Lian, L.-S. Young, Lyapunov exponents, periodic
  orbits, and horseshoes for semiflows on Hilbert spaces,
  \emph{J. Amer.\ Math.\ Soc.}, \textbf{25} (2012), 637--665.  

\bibitem[LS]{LS} Y. Lima, O.M. Sarig, Symbolic dynamics for
  three-dimensional flows with positive topological entropy,\emph{
    J. Eur.\ Math.\ Soc.\ (JEMS)}, \textbf{21} (2019),
  199--256. 

\bibitem[MS]{MS} L. Macarini, F. Schlenk, Positive topological entropy
  of Reeb flows on spherizations, \emph{Math.\ Proc.\ Cambridge
    Philos.\ Soc.}, \textbf{151} (2011), 103--128. 

  


\bibitem[McDS]{McDS} D. McDuff, D. Salamon, \emph{$J$-Holomorphic
    Curves and Symplectic Topology}, Colloquium Publications, Vol.\
  52, AMS, Providence, RI, 2004.

  
\bibitem[Me18]{Me:Th} M. Meiwes, \emph{Rabinowitz Floer Homology,
    Leafwise Intersections, and Topological Entropy}, Inaugural
  Dissertation zur Erlangung der Doktorw\"urde der
  Naturwissenschaftlich -- Mathematischen Gesamtfakultät der Ruprecht
  -- Karls -- Universit\"at Heidelberg, 2018; available at
  http://www.ub.uni-heidelberg.de/archiv/24153.
   

\bibitem[Me24]{Me} M. Meiwes, Relative barcode entropy and
horseshoes, Preprint arXiv:2401.07034.

\bibitem[MZ]{MZ} E. Miranda, N.T. Zung, Equivariant normal form for
  nondegenerate singular orbits of integrable Hamiltonian systems,
  \emph{Ann.\ Sci.\ \'Ecole Norm.\ Sup.} \textbf{37} (2004),
  819--839. 

\bibitem[Ne]{Ne} S.E. Newhouse, Continuity properties of entropy,
  \emph{Ann.\ Math.\ (2)}, \textbf{129} (1989), 215--235. 

\bibitem[Oh90]{Oh:Invent} Y.-G. Oh, Second variation and stabilities
  of minimal Lagrangian submanifolds in Kähler manifolds,
  \emph{Invent.\ Math.}, \textbf{101} (1990), 501--519. 

\bibitem[Oh93a]{Oh:MathZ} Y.-G. Oh, Volume minimization of Lagrangian
submanifolds under Hamiltonian deformations, \emph{Math.\ Z.},
\textbf{212} (1993), 175--192. 

  
\bibitem[Oh93b]{Oh} Oh, Y.-G., Floer cohomology of Lagrangian
  intersections and pseudo-holomorphic disks I and II, \emph{Comm.\
    Pure Appl.\ Math.} \textbf{46} (1993), 949--993 and 995--1012.


\bibitem[Oh99]{Oh:spec} Oh, Y.-G.,
Symplectic topology as the geometry of action functional. II. Pants product
and cohomological invariants, \emph{Comm\. Anal.\ Geom.}, \textbf{7} (1999), 1--54.

\bibitem[Oh05]{Oh:gamma} Y.-G. Oh, Spectral invariants, analysis of
  the Floer moduli space, and geometry of the Hamiltonian
  diffeomorphism group, \emph{Duke Math.\ J.}, \textbf{130} (2005),
  199--295.  




\bibitem[Pa97]{Pa} G.P. Paternain, Topological entropy for geodesic flows
  on fibre bundles over rationally hyperbolic manifolds, \emph{Proc.\
    Amer.\ Math.\ Soc.}, \textbf{125} (1997), 2759--2765. 

\bibitem[Pa99]{Pa:book} G.P. Paternain, \emph{Geodesic Flows},
Progress in Mathematics,  Birkh\"auser, (Boston, MA) 1999.

\bibitem[P$^2$S]{PPS}
I. Polterovich, L. Polterovich, V. Stojisavljevi\'c, Persistence barcodes and
Laplace eigenfunctions on surfaces, \emph{Geom.\ Dedicata},
\textbf{201} (2019), 111--138.

\bibitem[Po93]{Po:Hofer} L. Polterovich, Symplectic displacement
  energy for Lagrangian submanifolds, \emph{Ergodic Theory Dynam.\
    Systems}, \textbf{13} (1993), 357--367. 

 \bibitem[Po01]{Po:Book} L.  Polterovich, \emph{The Geometry of the
    Group of Symplectic Diffeomorphisms}, Lectures in Mathematics ETH
  Z\"urich, Birkh\"auser Verlag, Basel, 2001.
   
  
\bibitem[PS]{PS} L. Polterovich, E. Shelukhin, Autonomous Hamiltonian
  flows, Hofer’s geometry and persistence modules, \emph{Selecta
    Math.\ (N.S.)} \textbf{22} (2016), 227--296.
  
\bibitem[PRSZ]{PRSZ} L. Polterovich, D. Rosen, K. Samvelyan, J. Zhang,
  \emph{Topological Persistence in Geometry and Analysis}, University
  Lecture Series, vol.\ 74, Amer.\ Math.\ Soc., Providence, RI, 2020.


\bibitem[Sa]{Sa} D.A. Salamon, Lectures on Floer homology. In
  \emph{Symplectic Geometry and Topology}, IAS/Park City Math.\ Ser.,
  vol.\ 7, Amer.\ Math.\ Soc., Providence, RI, 1999, 143--229.

\bibitem[SaZe]{SaZe} D. Salamon, E. Zehnder, Morse theory for periodic
  solutions of Hamiltonian systems and the Maslov index, \emph{Comm.\
    Pure Appl.\ Math.}, \textbf{45} (1992), 1303--1360.

  
\bibitem[Sc]{Sc} M. Schwarz, On the action spectrum for closed symplectically aspherical manifolds,
\emph{Pacific J. Math.}, \textbf{193} (2000), 419--461.


  \bibitem[Se06]{Se:biased} P. Seidel, \emph{A Biased View of Symplectic
    Cohomology}, Current Developments in Mathematics, 2006
  (International Press, Somerville, MA, 2008) 211--253.

  
\bibitem[Se13]{Se} P. Seidel, \emph{Lectures on Categorical Dynamics and
    Symplectic Topology}, MIT lecture notes, 2013. 


  

\bibitem[Sh22]{Sh:HZ} E. Shelukhin, On the Hofer--Zehnder conjecture,
  \emph{Ann.\ of Math. (2)}, \textbf{195} (2022), 775--839.  

  
\bibitem[StZh]{SZ-SBM}
V. Stojisavljevi\'c, J. Zhang, 
 Persistence modules, symplectic Banach--Mazur distance and Riemannian metrics,
\emph{Internat. J.\ Math.}, \textbf{32} (2021), Paper No.\ 2150040, 76 pp. 

\bibitem[Su]{Su21} Y. Sugimoto, On the generic Conley conjecture,
  \emph{Arch.\ Math.}, \textbf{117} (2021), 423–432. 


\bibitem[Us08]{Us0} M. Usher, Spectral numbers in Floer theories,
  \emph{Compositio Math.}, \textbf{144} (2008), 1581--1592.
  
\bibitem[Us11]{Us1} M. Usher, Boundary depth in Floer theory and its
  applications to Hamiltonian dynamics and coisotropic submanifolds,
  \emph{Israel J. Math.}, \textbf{184} (2011), 1--57.
  

\bibitem[Us13]{Us} M. Usher, Hofer’s metrics and boundary depth,
  \emph{Ann.\ Sci.\ \'Ec. Norm.\ Sup\'er.}, \textbf{46} (2013),
  57--128.

  
\bibitem[Us22]{Us-SBM} M. Usher, Symplectic Banach--Mazur distances
  between subsets of $\C^n$, \emph{J. Topol.\ Anal.}, \textbf{14} (2022), 
    231--286. 
  
\bibitem[UZ]{UZ} M. Usher, J. Zhang, Persistent homology and
Floer--Novikov theory, \emph{Geom.\ Topol.}, \textbf{20} (2016),
3333--3430, 
  

\bibitem[Vi92]{Vi} C. Viterbo, Symplectic topology as the geometry of
  generating functions, \emph{Math.\ Ann.}, \textbf{292} (1992),
  685--710.  


\bibitem[Vi99]{Vi:GAFA} C. Viterbo, Functors and computations in Floer
  cohomology, I, \emph{Geom.\ Funct.\ Anal.}, \textbf{9} (1999),
  985--1033.

\bibitem[Vi00]{Vi:metric} C. Viterbo,
Metric and isoperimetric problems in symplectic geometry.
\emph{J. Amer.\ Math.\ Soc.}, \textbf{13} (2000), 411--431. 


  
\bibitem[Yo]{Yo} Y. Yomdin, Volume growth and entropy, \emph{Israel
    J. Math.}, \textbf{57} (1987), 285--300. 

  
\bibitem[ZC]{ZC} A. Zomorodian, G. Carlsson, Computing persistent
  homology, \emph{Discrete Comput.\ Geom.}, \textbf{33} (2005),
  249--274. 

\bibitem[Zu]{Zu} N.T.  Zung, Non-degenerate singularities of
  integrable dynamical systems, \emph{Ergodic Theory Dynam.\ Systems},
  \textbf{35} (2015), 994--1008. 

\end{thebibliography}
\end{document}